\documentclass{amsart}
\usepackage{times,array,amssymb,tikz,mathrsfs,xcolor}
\usetikzlibrary{petri}
\usepackage{tikz-cd}
\usetikzlibrary{cd}

\newcommand{\ccirc}{\kern0.5ex\vcenter{\hbox{$\scriptstyle\circ$}}\kern0.5ex}

\newcounter{lemma}[section]

\newtheorem{Theorem}[lemma]{Theorem}
\newtheorem{Lemma}[lemma]{Lemma}
\newtheorem{Proposition}[lemma]{Proposition}
\newtheorem{Corollary}[lemma]{Corollary}

\theoremstyle{definition}
\newtheorem{Example}[lemma]{Example}

\newtheorem{Definition}[lemma]{Definition}
\newtheorem{Remark}[lemma]{Remark}
\newtheorem*{$H_1$}{$(H_1)$}
\numberwithin{lemma}{section}
\numberwithin{equation}{section}

\def\H{\mathbb H}

\def\GL{\mathrm{GL}}

\def\SL{\mathrm{SL}}

\def\sM{\mathscr{M}}

\def\N{\mathbb N}
\def\C{\mathbb C}

\def\R{\mathbb R}

\def\Z{\mathbb Z}
\def\mod{\  \mathrm{mod}\ }
\def\ord{\operatorname{ord}}

\def\Im{\mathrm{Im\,}}
\def\gen#1{\langle #1\rangle}
\def\JS#1#2{\left(\frac{#1}{#2}\right)}
\def\gauss#1{\left\lfloor #1\right\rfloor}
\def\SL{\mathrm{SL}}
\def\PSL{\mathrm{PSL}}
\def\ol#1{\overline{#1}}
\def\wt#1{\widetilde{#1}}
\def\M#1#2#3#4{\begin{pmatrix}#1&#2\\#3&#4\end{pmatrix}}
\def\SM#1#2#3#4{\left(\begin{smallmatrix}#1&#2\\#3&#4\end{smallmatrix}
  \right)}
\def\column #1#2{\begin{pmatrix}#1\\#2\end{pmatrix}}

\def\X #1,#2{X^{#1}_0(#2)}

\begin{document}

\title[Quasimodular forms and modular differential
equations]{Quasimodular forms and modular differential
  equations which are not apparent at cusps: I}

\author{Chang-Shou Lin and Yifan Yang}

\address{Center for Advanced Study in Theoretical Sciences (CASTS),
  National Taiwan University, Taipei, Taiwan 10617.}
\email{cslin@math.ntu.edu.tw}

\address{Department of Mathematics, National Taiwan
University and National Center for Theoretical Sciences, Taipei,
Taiwan 10617.}
\email{yangyifan@ntu.edu.tw}

\begin{abstract} In this paper, we explore a two-way connection
  between quasimodular forms of depth $1$ and a class of second-order
  modular differential equations with regular singularities on the
  upper half-plane and the cusps. Here we consider the cases
  $\Gamma=\Gamma_0^+(N)$ generated by $\Gamma_0(N)$ and the
  Atkin-Lehner involutions for $N=1,2,3$
  ($\Gamma_0^+(1)=\SL(2,\Z)$). Firstly, we note that a quasimodular
  form of depth $1$, after divided by some modular form with the same
  weight, is a solution of a modular differential equation. Our main
  results are the converse of the above statement for the groups
  $\Gamma_0^+(N)$, $N=1,2,3$.
\end{abstract}

\date{\today}
\maketitle

\section{Introduction}
\label{section: introduction}
Let $\Gamma$ be a discrete subgroup of $\SL(2,\R)$ commensurable with
$\SL(2,\Z)$. The aim of our study is to explore a two-way connection
between quasimodular forms with depth $\ell$ on $\Gamma$ and modular
differential equations of order $\ell+1$. In this paper, we shall
consider the case $\ell=1$.

The notion of quasimodular forms was first introduced by Kaneko and
Zagier \cite{Kaneko-Zagier}. For example, if $\Gamma=\SL(2,\Z)$, a
function $f(z)$ on the upper half-plane $\H$ is called a quasimodular
form of weight $k$ and depth $\le 1$ if
\begin{equation} \label{equation: f=f1E2+f0}
  f(z)=f_1(z)E_2(z)+f_0(z), \qquad z\in\H,
\end{equation}
for some modular forms $f_0(z)$ and $f_1(z)$ of weight $k$ and $k-2$,
respectively, on $\SL(2,\Z)$, where $E_2(z)$ is the classical
Eisenstein series of weight $2$ on $\SL(2,\Z)$ with $E_2(\infty)=1$.

For $f(z)$ of \eqref{equation: f=f1E2+f0},
Pellarin \cite{Pellarin} introduced the Wronskian 
\begin{equation} \label{equation: Pellarin Wronskian}
  W_f(z)=\det\M{zf(z)+\frac6{\pi i}f_1(z)}{f(z)}
  {\left(zf(z)+\frac6{\pi i}f_1(z)\right)^{(1)}}{f^{(1)}(z)},
\end{equation}
where $g^{(k)}$ is the $k$-th derivative of a function $g$ with
respect to $z$, and a straightforward computation finds that $W_f(z)$
is a modular form of weight $2k$ (see \cite{Pellarin} or
Section \ref{section: quasimodular} for more details). Let
$$
g_1(z)=\frac{zf(z)+6f_1(z)/\pi i}{\sqrt{W_f(z)}} \quad
\text{and} \quad
g_2(z)=\frac{f(z)}{\sqrt{W_f(z)}}.
$$
Then we have $\det\SM{g_1}{g_2}{g_1'}{g_2'}=1$, and hence
$\det\SM{g_1}{g_2}{g_1''}{g_2''}=0$, by a further
differentiation. Thus, both $g_1$ and $g_2$ are solutions of
\begin{equation} \label{equation: y''=-Qy 0}
  y''=-4\pi^2Q(z)y, \qquad z\in\H,
\end{equation}
where
\begin{equation} \label{equation: g''/g}
  Q(z)=-\frac1{4\pi^2}\frac{g_1''}{g_1}
  =-\frac1{4\pi^2}\frac{g_2''}{g_2}.
\end{equation}
Although generally, $g_i$, $i=1,2$, are not single-valued functions,
it is easy to see that $Q(z)$ is a single-valued meromorphic function
on $\H$.

Equation \eqref{equation: y''=-Qy 0} is a second-order ODE in a
complex variable. Such a differential equation is called
Fuchsian if the order of pole of $Q(z)$ is $\le 2$. Let $z_0$
be a singular point and assume that
$Q(z)=A(z-z_0)^{-2}+O\left((z-z_0)^{-1}\right)$. Then the
indicial equation at $z_0$ is
$$
\rho(\rho-1)=-4\pi^2A,
$$
and its roots are $-\alpha$ and $\alpha+1$ for some
$\alpha\in\C$. These two roots are called the \emph{local exponents}
of \eqref{equation: y''=-Qy 0} (or of $Q(z)$). If
$\alpha\in\frac12\Z$, then equation \eqref{equation: y''=-Qy 0} might
have a solution with a logarithmic singularity near $z_0$. The
(regular) singular point $z_0$ is called \emph{apparent} if (a)
$\alpha\in\frac12\Z$ and (b) any solution of \eqref{equation: y''=-Qy
  0} has no logarithmic singularity at $z_0$. If (a) holds, but (b)
does not, we will simply say $z_0$ is \emph{not apparent}. In this
paper, we always assume that the local exponents of any singularity
$z_0\in\H$ are integers or half-integers, which will be denoted by
$1/2\pm\kappa_{z_0}$, $\kappa_{z_0}\in\frac12\Z_{\ge 0}$. It is easy
to see that (i) if $\kappa_{z_0}=0$, then $z_0$ is always not
apparent, and (ii) if $\kappa_{z_0}=1/2$, then $z_0$ is apparent if
and only if $Q(z)$ is holomorphic at $z_0$.

We start with the following observation.

\begin{Theorem} \label{proposition: quasi satisfies ODE}
  Let $Q(z)$ be defined \eqref{equation: g''/g}. Then $Q(z)$ is a
  meromorphic modular form of weight $4$ on $\SL(2,\Z)$. Moreover,
  the following hold.
  \begin{enumerate}
    \item[(i)] The ODE \eqref{equation: y''=-Qy 0} is Fuchsian on
      $\H\cup\{\text{cusps}\}$.
    \item[(ii)] Any pole of $Q$ on $\H$ is an apparent
      singularity for \eqref{equation: y''=-Qy 0}.
    \item[(iii)] The cusp $\infty$ is not apparent.
    \item[(iv)] Let $1/2\pm\kappa_i$ and $1/2\pm\kappa_\rho$,
      $\kappa_i,\kappa_\rho\in\frac12\N$, be the local exponents of
      \eqref{equation: y''=-Qy 0} at $i=\sqrt{-1}$
      and $\rho=(1+\sqrt{-3})/2$, respectively. Then
      $(2\kappa_i,2)=(2\kappa_\rho,3)=1$.
  \end{enumerate}
\end{Theorem}

This theorem is a special case of Theorem \ref{theorem: Q is modular}
and Proposition \ref{proposition: local exponents at elliptic}, proved
in Section \ref{section: ODE and quasi}.

The unexpected result of Theorem \ref{proposition: quasi satisfies
  ODE} is the modularity of $Q$. In view of Theorem \ref{proposition:
  quasi satisfies ODE},
a natural question is whether the converse of Theorem
\ref{proposition: quasi satisfies ODE} holds or not.
This is the main question studied in this paper, not only for the
group $\SL(2,\Z)$. Our methods for proving the main results depend on
the structure of $\Gamma$ and its modular forms. Thus, in this paper,
we restrict our attention to $\Gamma=\Gamma_0^+(N)$, $N=1,2,3$, where
$\Gamma_0^+(N)$ is the group generated by $\Gamma_0(N)$ and all the
Atkin-Lehner involutions. (Note that $\Gamma_0^+(1)=\SL(2,\Z)$. In the case $N=2,3$, $\Gamma_0^+(N)$ is simply the group
generated by $\Gamma_0(N)$ and $\frac1{\sqrt N}\SM0{-1}N0$.)
For such a group $\Gamma$, there are only one cusp $\infty$ and two
elliptic points $\rho_i$, $i=1,2$, where the order of $\rho_1$ is $2$,
and the order of $\rho_2\in\{3,4,6\}$, depending on $N$.
Throughout the paper, we refer to
\begin{equation} \label{equation: y''=-Qy}
  y''(z)=-4\pi^2Q(z)y(z)
\end{equation}
as a modular ordinary differential equation (MODE) if $Q(z)$ is a
nonzero \emph{meromorphic modular form of weight $4$} on
$\Gamma$. Motivated
by Theorem \ref{proposition: quasi satisfies ODE}, we consider
equations \eqref{equation: y''=-Qy} satisfying the following condition
\eqref{equation: H}.
\begin{equation} \label{equation: H}
  \tag{$\text{H}_\Gamma$}
  \parbox{\dimexpr\linewidth-5em}{
    \begin{enumerate}
    \item[(i)] The ODE \eqref{equation: y''=-Qy} is Fuchsian on $\H$
      and all cusps.
    \item[(ii)]
      The local exponents of \eqref{equation: y''=-Qy} at any singular
      point $z_0$ on $\H$ are $1/2\pm\kappa_{z_0}$,
      $\kappa_{z_0}\in\frac12\N$. Moreover, $z_0$ is apparent
      for \eqref{equation: y''=-Qy}.
    \item[(iii)] Let $e_{\rho_i}$, $i=1,2$, be the order of the
      elliptic point $\rho_i$ and $1/2\pm\kappa_{\rho_i}$,
      $\kappa_{\rho_i}\in\frac12\N$, be the local
      exponents of \eqref{equation: y''=-Qy} at $\rho_i$. Then
      $(2\kappa_{\rho_i},e_i)=1$ for $i=1,2$.
    \end{enumerate}
  }
\end{equation}

In Section \ref{section: ODE and quasi}, we will briefly explain the
basic notions related to ODE in a complex variable. (We refer to
\cite{Hille} for an introduction to elementary theory of complex
ODE.) Here we shall give some explanation pertaining to our condition
\eqref{equation: H}. Firstly, suppose
that the cusp $\infty$ has width $N$. Let $q_N=e^{2\pi
  iz/N}$. Then \eqref{equation: y''=-Qy} can be written as
\begin{equation} \label{equation: y''=-Qy in Dq}
  \left(q_N\frac d{dq_N}\right)^2y=N^2Qy
\end{equation}
since $d/dz=2\pi iN^{-1}q_Nd/dq_N$. For other cusps $s$, we choose
$\gamma=\SM abcd\in\SL(2,\Z)$ such that $\gamma\infty=s$. Set
\begin{equation} \label{equation: slash}
  \hat y(z)=\left(y\big|_{-1}\gamma\right)(z):=(cz+d)y(\gamma z).
\end{equation}
Then $\hat y$ satisfies
\begin{equation} \label{equation: y''=-Qy at s}
  \hat y''=-4\pi^2\hat Q(z)\hat y, \qquad
  \hat Q(z)=\left(Q\big|_4\gamma\right)(z).
\end{equation}
By \eqref{equation: y''=-Qy in Dq} and \eqref{equation: y''=-Qy at s},
we see that \eqref{equation: y''=-Qy} is Fuchsian at a cusp $s$ if and
only if $Q$ is \emph{holomorphic at} $s$, and the local expondents at
$s$ are $\pm\kappa_s$, where
\begin{equation} \label{equation: local exponents at s}
  \kappa_s=N\sqrt{Q(s)}.
\end{equation}

Secondly, suppose that $z_0\in\H$ is a singular point and
\eqref{equation: y''=-Qy} is a modular ODE. Then the local exponents
at $\gamma z_0$, $\gamma\in\Gamma$, are the same as those at $z_0$.
Also, if $z_0$ is apparent, then so is $\gamma z_0$. See
\cite{Guo-Lin-Yang} for a proof.

  
For a cusp $s$, we can define a similar notion of apparentness (or
non-apparentness) when $\kappa_s$ in \eqref{equation: local exponents
  at s} is in $\frac12\Z_{\ge 0}$.
Suppose that at a cusp $s$ we have $\kappa_s\in\frac12\Z_{\ge
  0}$. Then \eqref{equation: y''=-Qy} always has a unique solution
$y_+(z)$ of the form
\begin{equation} \label{equation: y+ introduction}
  y_+(z)=q_N^{\kappa_s}\left(1+\sum_{j\ge 1}c_jq_N^j\right)
\end{equation}
An elementary theorem in theory of complex ODE says that if $s$ is not
apparent and $y(z)$ is a solution of \eqref{equation: y''=-Qy}
independent of $y_+(z)$, then there exists $e\neq 0\in\C$ and $m(z)$
such that
\begin{equation} \label{equation: second solution}
  y(z)=ezy_+(z)+m(z), \qquad
  m(z)=q_N^{-\kappa_s}\sum_{j=0}^\infty \hat c_jq_N^j, \quad \hat
  c_0\neq 0.
\end{equation}

We now introduce a representation of $\Gamma$ associated to
a modular ODE as follows.
It is an elementary fact that any (local) solution $y(z)$ of
\eqref{equation: y''=-Qy} can be defined on the whole $\H$ through
analytic continuation. Fix a point $z_0\in\H$ that is not a singular
point of $Q(z)$. Let $U$ be a simply connected (small) open set in
$\H$ containing $z_0$, but not any singularities of $Q(z)$. For
$\gamma\in\Gamma$, choose a path $\sigma$ from $z_0$ to $\gamma z_0$
and consider the analytic continuation of $y(z)$, $z\in U$, along the
path. Then $y(\gamma z)$ is well-defined in $U$.
We define $y\big|_{-1}\gamma$ by
$$
\left(y\big|_{-1}\gamma\right)(z):=(cz+d)y(\gamma z), \qquad
z\in U.
$$
Then $y\big|_{-1}\gamma$ is also a solution of \eqref{equation:
  y''=-Qy}. This can be proved by a direct computation or by using
Bol's identity in literature \cite{Bol}. By choosing a fundamental
system $(y_1(z),y_2(z))$ of solutions,
it gives a matrix $\hat\gamma\in\SL(2,\C)$ such that
$$
  \column{y_1\big|_{-1}\gamma}{y_2\big|_{-1}\gamma}(z)
  =\hat\gamma\column{y_1(z)}{y_2(z)}.
$$
(The reason why $\hat\gamma$ has determinant $1$ is due to the fact
that $y_1|_{-1}\gamma$ and $y_2|_{-1}\gamma$ has the same Wronskian as
$y_1$ and $y_2$.) Of course, this matrix $\hat\gamma$ depends on the
choice of the path $\sigma$. However, under the condition
\eqref{equation: H}, all local monodromy matrices are $\pm I$. Hence,
different choices of $\sigma$ will only possibly change $\hat\gamma$
to $-\hat\gamma$. From this, we 
see that there is a well-defined homomorphism $\rho$ from $\Gamma$ to
$\PSL(2,\C)$ such that
\begin{equation} \label{equation: Bol representation}
  \column{y_1\big|_{-1}\gamma}{y_2\big|_{-1}\gamma}(z)
  =\pm\rho(\gamma)\column{y_1(z)}{y_2(z)}
\end{equation}
for all $\gamma\in\Gamma$, where $y_j(\gamma z)$, $j=1,2$, are always
understood to take the same path for the analytic continuation. This
homomorphism $\rho$ will be called the \emph{Bol
  representation} in this paper. For the group $\Gamma$ under
consideration in this paper, the Bol
representation can be lifted to a homomorphism $\hat\rho$ from
$\Gamma$ to $\GL(2,\C)$. To achieve this, we let $F(z)^2$ be a modular
form of weight $2(\ell+1)$ with some character for some integer $\ell$
such that $\hat y(z)=F(z)y(z)$ is well-defined on $\H$ for any
solution $y(z)$ of \eqref{equation: y''=-Qy}.
Note that since $\hat y(z)$ is single-valued, for any
$\gamma\in\Gamma$, there exists
$\hat\rho(\gamma)\in\GL(2,\C)$ such that
$$
  \column{\hat y_1\big|_{\ell}\gamma}{y_2\big|_{\ell}\gamma}
  =\hat\rho(\gamma)\column{\hat y_1}{\hat y_2}.
$$
Obviously, $\hat\rho$ is a lifting of $\rho$. We remark that there are
many possible choices for $F(z)$. In Sections \ref{section:
  SL(2,Z)}--\ref{section: Gamma0+(3)}, we will choose $F(z)$ such that
$\hat y(z)$ is holomorphic on $\H$ and at cusps. We emphasize that the
integer $\ell$ will give rise to some important information in this
series of papers. For example, we have $\det\hat\rho(\gamma)=1$ for
all $\gamma\in\Gamma$ if and only if $\ell$ is odd.

From now on, we assume that $\kappa_\infty\in\frac12\Z_{\ge 0}$, where
$\pm\kappa_\infty$ are the local exponents of \eqref{equation:
  y''=-Qy} at the cusp $\infty$, and $\{z_1,\ldots,z_m\}$ is the set
of poles of $Q(z)$ such that $z_i$ is not an elliptic point and $z_i$
is not $\Gamma$-equivalent to $z_j$ for $i\neq j$. Let
$1/2\pm\kappa_j$ be the local exponents of \eqref{equation: y''=-Qy}
at $z_j$.

When $\Gamma=\SL(2,\Z)$, the integer $\ell$ in the lift of the Bol
representation is defined by
\begin{equation} \label{equation: ell introduction}
  \ell=-1+12\kappa_\infty+4\left(\kappa_\rho-\frac12\right)
  +6\left(\kappa_i-\frac12\right)
  +12\sum_{j=1}^m\left(\kappa_j-\frac12\right),
\end{equation}
where $i=\sqrt{-1}$ and $\rho=(1+\sqrt{-3})/2$ are the elliptic points
and $1/2\pm\kappa_\rho$ and $1/2\pm\kappa_i$ are the local exponents
of \eqref{equation: y''=-Qy} at $\rho$ and $i$, respectively. Since
all the $\kappa$'s are in $\frac12\Z$, $\ell$ is always an
integer. Furthermore, \emph{the integer $\ell$ is odd if and only if
$\kappa_i\in\frac12+\Z$}. Thus, \eqref{equation: H} implies that
$\ell$ is odd.

Let $F(z)$ be defined by \eqref{equation: F SL(2,Z)}. Then $F(z)^2$ is
a modular form of weight $2(\ell+1)$ such that for any solution $y(z)$
of \eqref{equation: y''=-Qy}, the function $\hat y(z)=F(z)y(z)$ is a
single-valued holomorphic function on $\H$. We use the standard
notations $T=\SM1101$, $S=\SM0{-1}10$, and $R=TS=\SM1{-1}10$. They
satisfy $S^2=R^3=-I$. Our first main result asserts that the converse
of Theorem \ref{proposition: quasi satisfies ODE} holds true.

\begin{Theorem} \label{theorem: main 1}
  Let $\Gamma=\SL(2,\Z)$. Suppose that $Q(z)$, not identically $0$,
  satisfies \eqref{equation: H} with $\kappa_\infty\in\frac12\Z_{\ge
    0}$. Then
  the following statements hold true.
  \begin{enumerate}
    \item[(i)] The cusp $\infty$ is not apparent.
    \item[(ii)] Let $y_+(z)$ be the solution of \eqref{equation:
        y''=-Qy} of the form \eqref{equation: y+ introduction} and set
      $y_1(z)=\left(\hat y_+|_{\ell}S\right)(z)/F(z)$ and
      $y_2(z)=y_+(z)$. Then $y_1(z)$ and $y_2(z)$ form a fundamental
      system of solutions of
      \eqref{equation: y''=-Qy}. Moreover, we have
      $$
      \hat y_1(z)=z\hat y_2(z)+\hat m_1(z)
      $$
      for some modular form $\hat m_1(z)$ of weight $\ell-1$ on
      $\SL(2,\Z)$.
    \item[(iii)] Using $y_1(z)$ and $y_2(z)$ as the basis, Bol's
      representation satisfies
      $$
      \hat\rho(\gamma)=\gamma
      $$
      for all $\gamma\in\SL(2,\Z)$. In particular, the ratio
      $h(z)=y_1(z)/y_2(z)$ is equivariant.
    \item[(iv)] Write $\hat y_+(z)=\frac{\pi i}6\hat m_1(z)E_2(z)+\hat
      m_2(z)$. Then $\hat m_2(z)$ is a modular form of weight $\ell+1$
      with respect to $\SL(2,\Z)$. In other words, $\hat y_+(z)$ is a
      quasimodular form of weight $\ell+1$ and depth $1$ on
      $\SL(2,\Z)$.
  \end{enumerate}
\end{Theorem}

\begin{Remark} Since \eqref{equation: y''=-Qy} is assumed to be
  apparent at any pole of $Q$ on $\H$, the ratio $h(z)$ of any two
  linearly independent solution of \eqref{equation: y''=-Qy} is a
  meromorphic (single-valued) function on $\H$. We say a meromorphic
  function $h(z)$ on $\H$ is \emph{equivariant} with respect to
  $\SL(2,\Z)$ if $h(z)$ satisfies
  $$
  h(\gamma z)=\gamma\cdot h(z):=\frac{ah(z)+b}{ch(z)+d}
  $$
  for all $\gamma=\SM abcd\in\SL(2,\Z)$ (see \cite{Sebbar}).
\end{Remark}

We note that $y_1(z)$ is a second solution of \eqref{equation:
  y''=-Qy}, so $y_1(z)$ has an expression of the form \eqref{equation:
  second solution} for some $e\neq 0$. Then Part (ii) says that
$e=1$. It is surprising to see that if the second solution is suitably
chosen, then $m(z)$ of \eqref{equation: second solution} is a
meromorphic modular form of weight $-2$.

For the case $\Gamma=\Gamma_0^+(N)$, $N=2,3$, we define
$$
T=\M1101, \quad
S=\frac1{\sqrt N}\M0{-1}N0, \quad
R=TS=\frac1{\sqrt N}\M N{-1}N0
$$
such that
$$
S^2=R^{2N}=-I.
$$
The group $\Gamma_0^+(N)$, $N=2,3$, has one cusp $\infty$ with width
$1$ and two elliptic point $\rho_1$ of order $2$ and $\rho_2$ of order
$2N$, which are the fixed points of $S$ and $R$, respectively, i.e.,
\begin{equation} \label{equation: elliptic points}
\rho_1=\frac i{\sqrt N}, \qquad
\rho_2=\begin{cases}
  (1+i)/2, &\text{if }N=2, \\
  (3+\sqrt{-3})/6, &\text{if }N=3. \end{cases}
\end{equation}
The integer $\ell$ in the lifting in the case $\Gamma_0^+(2)$ is
$$
\ell=-1+8\kappa_\infty+4\left(\kappa_{\rho_1}-\frac12\right)
+2\left(\kappa_{\rho_2}-\frac12\right)
+8\sum_{j=1}^m\left(\kappa_j-\frac12\right).
$$
Note that \eqref{equation: H} implies that both $\kappa_{\rho_1}-1/2$
and $\kappa_{\rho_2}-1/2$ are nonnegative integers. Hence $\ell$ is an
odd integer. Let $F(z)$ be defined by \eqref{equation: F Gamma0+(2)}
and for a solution $y(z)$ of \eqref{equation: y''=-Qy}, set $\hat
y(z)=F(z)y(z)$. The main theorem for $\Gamma_0^+(2)$ is as follows.

\begin{Theorem} \label{theorem: main 2}
  Let $\Gamma=\Gamma_0^+(2)$. Suppose that $Q(z)$ satisfies
  \eqref{equation: H} with $\kappa_\infty\in\frac12\Z_{\ge 0}$.
  Then the following statements hold.
  \begin{enumerate}
    \item[(i)] The cusp $\infty$ is not apparent for \eqref{equation:
        y''=-Qy}.
    \item[(ii)] Let $y_1(z)=\left(\hat y_+|_\ell S\right)(z)/F(z)$ and
      $y_2(z)=y_+(z)$, where $y_+(z)$ is the unique solution of
      \eqref{equation: y''=-Qy} of the form \eqref{equation: y+
        introduction}. Then $\{y_1(z),y_2(z)\}$ is a fundamental
      solution of \eqref{equation: y''=-Qy}. Moreover, we have
      \begin{equation} \label{equation: y1 2}
      \hat y_1(z)=\JS2\ell\sqrt2z\hat y_2(z)+\hat m_1(z)
      \end{equation}
      for some modular form $\hat m_1(z)$ in
      $\sM_{\ell-1}(\Gamma_0^+(2),\JS2\ell)$, where $\JS2\ell$ is the
      Legendre symbol whose values are given by
      $$
      \JS2\ell=\begin{cases}
        1, &\text{if }\ell\equiv 1,7\mod 8, \\
        -1, &\text{if }\ell\equiv 3,5\mod 8. \end{cases}
      $$
    \item[(iii)] The ratio $h(z)=\JS2\ell y_1(z)/\sqrt2y_2(z)$ is
      equivariant with respect to $\Gamma_0^+(2)$. That is, for all
      $\gamma\in\Gamma_0^+(2)$, we have $h(\gamma z)=\gamma\cdot
      h(z)$.
    \item[(iv)] Let  $M_2^\ast(z)=(2E_2(2z)+E_2(z))/3$ and write $\hat
      y_+(z)$ as
      $$
      \hat y_+(z)=\JS2\ell\frac{\pi i}{4\sqrt 2}\hat m_1(z)
      M_2^\ast(z)+\hat m_2(z).
      $$
      Then $\hat m_2(z)$ is a modular form in
      $\sM_{\ell+1}(\Gamma_0^+(2),\JS2\ell)$. Consequently, $\hat
      y_+(z)$ is a quasimodular form in $\wt\sM_{\ell+1}^{\le
        1}(\Gamma_0^+(2),\JS2\ell)$.
  \end{enumerate}
\end{Theorem}

Note that $M_2^\ast(z)$ is a quasimodular form of weight $2$ and depth
$1$ on $\Gamma_0^+(2)$. Also, $\sM_k(\Gamma_0^+(2),\pm1)$ denotes the
space of modular forms $f(z)$ of weight $k$ on $\Gamma_0(2)$ such that
$\left(f|_kS\right)(z)=\pm f(z)$. (The reason for the choice of the
notation $\sM_k(\Gamma_0^+(2),\pm)$ instead of the more common
$\sM_k(\Gamma_0(2),\pm)$ is that we regard such an $f$ as a modular
form with character on $\Gamma_0^+(2)$.) For the meaning of the
notation $\wt\sM_k^{\le 1}(\Gamma_0^+(2),\pm1)$, see Section
\ref{subsection: Atkin-Lehner}.

\begin{Remark} The identity \eqref{equation: y1 2} says that the
  constant $e$ in \eqref{equation: second solution} is
  $\JS2\ell\sqrt2$, which depends on $\ell$. This result is different
  from Theorem \ref{theorem: main 1}.
\end{Remark}

For the case $\Gamma=\Gamma_0^+(3)$, the group $\Gamma$ has one cusp
$\infty$ and two elliptic points $\rho_1$ and $\rho_2$ given by
\eqref{equation: elliptic points}. In this case, $\ell$ is defined by
$$
\ell=-1+6\kappa_\infty+3\left(\kappa_{\rho_1}-\frac12\right)
+\left(\kappa_{\rho_2}-\frac12\right)
+6\sum_{j=1}^m\left(\kappa_j-\frac12\right).
$$
Note that by \eqref{equation: H}, both $\kappa_{\rho_1}-1/2$ and
$\kappa_{\rho_2}-1/2$ are integers, and $\ell$ is an integer not
divisible by $3$. Let $F(z)$ be defined by \eqref{equation: F
  Gamma0+(3)}. As before, for a solution $y(z)$ of \eqref{equation:
  y''=-Qy}, we let $\hat y(z)=F(z)y(z)$.

\begin{Theorem} \label{theorem: main 3}
  Let $\Gamma=\Gamma_0^+(3)$. Suppose that $Q(z)$ satisfy
  \eqref{equation: H} with $\kappa_\infty\in\frac12\Z_{\ge 0}$. Set
  $$
  \delta=\begin{cases}
    \chi^0, &\text{if }\ell\equiv1,11\mod 12,\\
    \chi^1, &\text{if }\ell\equiv 2,4\mod 12,\\
    \chi^2, &\text{if }\ell\equiv 5,7\mod 12, \\
    \chi^3, &\text{if }\ell\equiv 8,10\mod 12,\end{cases}
  $$
  where $\chi$ is the character of $\Gamma_0^+(3)$ defined by
  $\chi(S)=\chi(R)=-i$. Then the following statements hold.
  \begin{enumerate}
    \item[(i)] The cusp $\infty$ is not apparent.
    \item[(ii)] Let $y_1(z)=\left(\hat y_+|_\ell S\right)(z)/F(z)$ and
      $y_2(z)=y_+(z)$. Then $\{y_1(z),y_2(z)\}$ is a fundamental
      solution of \eqref{equation: y''=-Qy}. Moreover, we have
      $$
      \hat y_1(z)=\delta(S)^{-1}\sqrt3z\hat y_2(z)+\hat m_1(z)
      $$
      for some modular form $\hat m_1(z)$ in
      $\sM_{\ell-1}(\Gamma_0^+(3),\delta)$.
    \item[(iii)] The ratio $h(z)=\delta(S)^{-1}y_1(z)/\sqrt3y_2(z)$ is
      equivariant with respect to $\Gamma_0^+(3)$. That is, for all
      $\gamma\in\Gamma_0^+(3)$, we have $h(\gamma z)=\gamma\cdot
      h(z)$.
    \item[(iv)] Let $M_2^\ast(z)=(3E_2(3z)+E_2(z))/4$ and write $\hat
      y_+(z)$ as
      $$
      \hat y_+(z)=\delta(S)^{-1}\frac{\pi i}{3\sqrt3}
      \hat m_1(z)M_2^\ast(z)+\hat m_2(z).
      $$
      Then $\hat m_2(z)$ is a modular form in
      $\sM_{\ell+1}(\Gamma_0^+(3),\delta)$. Consequently, $\hat
      y_+(z)$ is a quasimodular form in $\wt\sM_{\ell+1}^{\le
        1}(\Gamma_0^+(3),\delta)$.
  \end{enumerate}
\end{Theorem}

Note that $M_2^\ast(z)$ is a quasimodular form of weight $2$ and depth
$1$ on $\Gamma_0^+(3)$. Also, $\sM_k(\Gamma_0^+(3),\delta)$ denote the
space of modular forms of weight $k$ with character $\delta$ on
$\Gamma_0^+(3)$. See Section \ref{subsection: odd} for a discussion
about these spaces and the space $\wt\sM_k^{\le 1}
(\Gamma_0^+(3),\delta)$.

The rest of the paper is organized as follows. In Section
\ref{section: quasimodular}, we give a quick introduction to
quasimodular forms, discuss the Atkin-Lehner decomposition of the
space of quasimodular forms on $\Gamma_0(N)$, and then prove the
existence of extremal quasimodular forms in certain spaces, which will
used to provide examples of modular differential equations satisfied
by quasimodular forms. In Section \ref{section: ODE and quasi}, we
will prove that every quasimodular form of depth $1$ will give rise to
a modular differential equation of the form \eqref{equation: y''=-Qy}
and show that such a differential equation in the case
$\Gamma=\SL(2,\Z)$, $\Gamma_0^+(2)$, or $\Gamma_0^+(3)$ will satisfy
the condition \eqref{equation: H}. In Sections \ref{section:
  SL(2,Z)}--\ref{section: Gamma0+(3)}, we consider the converse
direction. We will prove that if a modular differential equation
\eqref{equation: y''=-Qy} satisfies \eqref{equation: H} with
$\Gamma=\SL(2,\Z)$, $\Gamma_0^+(2)$, or $\Gamma_0^+(3)$, then its
solutions must come from quasimodular forms. Finally, in the appendix,
we discuss how to determine apparentness of a modular differential
equation and prove the existence of meromorphic modular forms $Q(z)$
satisfying \eqref{equation: H}.

\section{Quasimodular forms}
\label{section: quasimodular}

For convenience of non-specialists, we shall give a quick introduction
of quasimodular forms in this section. See \cite{Choie-Lee,Zagier123}
for more detailed overview and discussion. Note that most of the
definitions and properties presented here are valid or have analogues
in the setting of general Fuchsian subgroups of the first kind of
$\SL(2,\R)$, but for simplicity, here we restrict our attention to
subgroups $\Gamma$ of $\SL(2,\R)$ commensurable with $\SL(2,\Z)$,
which we assume throughout the section.

\subsection{Basic definitions and properties}
\label{subsection: definition of quasi}
The notion of quasimodular forms was first introduced by Kaneko and
Zagier \cite{Kaneko-Zagier}. It was defined through the notion of
nearly holomorphic modular forms (called almost holomorphic modular
forms in \cite{Kaneko-Zagier}).

\begin{Definition}
  A function $f:\H\to\C$ is said to be \emph{nearly holomorphic} if it
  is of the form
  \begin{equation} \label{equation: nearly holomorphic}
    f(z)=\sum_{r=0}^n\frac{f_r(z)}{(z-\ol z)^r}
  \end{equation}
  for some holomorphic functions $f_r$.

  Let $\Gamma$ be a congruence subgroup of $\SL(2,\R)$. A function
  $f:\H\to\C$ is said to be a \emph{nearly holomorphic modular form}
  on $\Gamma$ if it is nearly holomorphic, satisfies
  \begin{equation} \label{equation: nearly modular}
    f\left(\frac{az+b}{cz+d}\right)=(cz+d)^kf(z)
  \end{equation}
  for all $\SM abcd\in\Gamma$, and has a finite limit at each cusp of
  $\Gamma$. The largest integer $r$ such that $f_r\not\equiv 0$ in
  \eqref{equation: nearly holomorphic} is called the \emph{depth} of
  $f$. We let $\sM_k^\ast(\Gamma)$ denote the space of nearly
  holomorphic modular forms of weight $k$ on $\Gamma$.
\end{Definition}

On the spaces of nearly holomorphic modular forms, we have the
so-called Shimura-Maass operator.

\begin{Definition} \label{definition: Shimura-Maass}
  For a nearly holomorphic function $f:\H\to\C$ and
  an integer $k$, define the \emph{Shimura-Maass operator}
  $\partial_k$
  of weight 
  $k$ by
$$
(\partial_k f)(z)=\frac1{2\pi i}\left(
f'(z)+\frac{k f(z)}{z-\overline z}\right).
$$
Also, for $\gamma=\SM abcd\in\mathrm{GL}^+(2,\R)$, the slash
operator of weight $k$ is defined by
\begin{equation*}
  (f\big|_k\gamma)(z)=\frac{(\det\gamma)^{k/2}}{(cz+d)^k}
  f\left(\frac{az+b}{cz+d}\right).
\end{equation*}
\end{Definition}

Then we have the following properties.

\begin{Lemma}[{\cite[Equations (1.5) and (1.8)]{Shimura-Maass}}]
\label{lemma: Maass}
For any nearly holomorphic functions $f,g:\H\to\C$, any integers $k$
and $\ell$, and any $\gamma\in\GL^+(2,\R)$, we have
$$
  \partial_{k+\ell}(fg)=(\partial_k f)g+f(\partial_\ell g)
$$
and
$$
  \partial_k\left(f\big|_k\gamma\right)=(\partial_k f)\big|_{k+2}\gamma.
$$
\end{Lemma}

In particular, the second property in the lemma says that if $f$ is
a nearly holomorphic modular form of weight $k$ on $\Gamma$, then
$\partial_kf$ is a nearly holomorphic form of weight 
$k+2$ on $\Gamma$. That is, $\partial_k$ is a linear transformation
from $\sM^\ast_k(\Gamma)$ to $\sM^\ast_{k+2}(\Gamma)$.

\begin{Definition} We say a holomorphic function $f:\H\to\C$ is a
  \emph{quasimodular form} of weight $k$ and depth $r$ on $\Gamma$ if
  it is the holomorphic part of a nearly holomorphic modular form of
  weight $k$ and depth $r$ on $\Gamma$. We let
  $\wt\sM_k^{\le r}(\Gamma)$ denote the space of quasimodular forms of
  weight $k$ and depth $\le r$ on $\Gamma$. We also let
  $$
  \wt\sM(\Gamma)=\bigoplus_k\bigcup_r\wt\sM_k^{\le r}(\Gamma)
  $$
  denote the graded ring of quasimodular forms of all
  weights and all depths on $\Gamma$.
\end{Definition}

Using the modular property \eqref{equation: nearly modular}, we can
show that if $f_0\in\wt\sM_k(\Gamma)$ is the holomorphic part of the
nearly holomorphic modular form
$$
f(z)=\sum_{r=0}^n\frac{f_r(z)}{(z-\ol z)^r},
$$
then
\begin{equation} \label{equation: quasimodular transformation}
\left(f_0\big|_k\gamma\right)(z)=\sum_{r=0}^\infty f_r(z)
\left(\frac c{cz+d}\right)^r
\end{equation}
for $\SM abcd\in\Gamma$ and vice versa.

The archetypal example of a quasimodular form is $E_2(z)$, which is
the holomorphic part of the nearly holomorphic modular form
$$
E_2(z)+\frac 6{\pi i(z-\ol z)}=E_2(z)-\frac3{\pi\,\Im z}
$$
of weight $2$ and depth $1$ on $\SL(2,\Z)$. It can be shown that the
graded ring $\widetilde\sM(\SL_2(\Z))$ of quasimodular forms on
$\SL(2,\Z)$ is generated by $E_2(z)$, $E_4(z)$, and $E_6(z)$.
Also, by Ramanujan's identities
$$
D_qE_2=\frac{E_2^2-E_4}{12}, \qquad
D_qE_4=\frac{E_2E_4-E_6}3, \qquad
D_qE_6=\frac{E_2E_6-E_4^2}2,
$$
the ring $\widetilde\sM(\SL(2,\Z))$ is closed under the differential
operator
$$
D_q:=q\frac d{dq}=\frac1{2\pi i}\frac d{dz}.
$$
More generally, we have the following description of $\wt\sM_k^{\le
  r}(\Gamma)$.

\begin{Proposition}[{\cite[Proposition 20]{Zagier123}}]
  \label{proposition: sM dim}
  Let $\Gamma$ be a subgroup of $\SL(2,\R)$ commensurable with $\SL(2,\Z)$.
  \begin{enumerate}
  \item[(i)] The graded ring $\wt\sM(\Gamma)$ of quasimodular forms on
    $\Gamma$ is closed under differentiation. More precisely, we
    have $D_q(\wt\sM_k^{\le r}(\Gamma))\subseteq\wt\sM_{k+2}^{\le
      r+1}(\Gamma)$ for all $k,r\ge 0$.
  \item[(ii)] Let $\phi$ be a quasimodular form of weight $2$ and
    depth $1$ on $\Gamma$. Then
    $$
    \wt\sM_k^{\le r}(\Gamma)=\bigoplus_{j=0}^r\phi^j
    \sM_{k-2j}(\Gamma)
    $$
    for all $k,r\ge 0$. In particular, we have
    $$
    \dim\sM_k^{\le r}(\Gamma)=\sum_{j=0}^r\dim\sM_{k-2j}(\Gamma).
    $$
  \end{enumerate}
\end{Proposition}

The quasimodular form $\phi$ in the proposition can be constructed by
taking the logarithmic derivative of a meromorphic modular form of
nonzero weight that is nonvanishing on $\H$. In general, the choice of
$\phi$ is not unique. For example, when $\Gamma=\Gamma_0(p)$, we can
take $\phi$ to be the logarithmic derivative of any
$\Delta(z)^m\Delta(pz)^n$, as long as $m+n\neq 0$.

More generally, we can define quasimodular forms on $\Gamma$ with
character $\chi$ to be the holomorphic parts of nearly holomorphic 
function $f:\H\to\C$ satisfying
$$
f(\gamma z)=\chi(\gamma)(cz+d)^kf(z), \qquad
\gamma=\M abcd\in\Gamma,
$$
and analytic conditions at cusps. (Thus, if $-I\in\Gamma$ and
$\chi(-I)=-1$, we will assume that $k$ is odd.) We similarly let
$\wt\sM_k^{\le r}(\Gamma,\chi)$ denote the space of all quasimodular
forms on $\Gamma$ of weight $k$ with character $\chi$. In this paper,
we will only consider characters of finite order.

The first examples that provides a link between quasimodular forms and
modular differential equations were due to Kaneko and Koike
\cite{Kaneko-Koike-2003,Kaneko-Koike}. Among other things, they proved
that if $6|k$, then there is a solution $f(z)$ of the differential
equation
\begin{equation} \label{equation: KK example}
D_q^2y(z)-\frac k6E_2(z)D_qy(z)+\frac{k(k-1)}{12}(D_qE_2(z))y(z)=0,
\end{equation}
such that $f$ is an extremal quasimodular form in $\wt\sM^{\le 
  1}_k(\SL(2,\Z))$, defined as an element whose vanishing order at
$\infty$ is $\dim\wt\sM_k^{\le1}(\SL(2,\Z))-1$.
Then by a simple transformation of the differential equation, we see
that $y(z)=f(z)/\Delta(z)^{k/12}$ is
a solution of the differential equation
\begin{equation} \label{equation: example of KK}
D_q^2y(z)=\left(\frac k{12}\right)^2E_4(z)y(z).
\end{equation}
Analogues of \eqref{equation: KK example} for the groups
$\Gamma_0^+(N)$, $N=2,3$, and $\Gamma_0(N)$, $N=2,3,4$, were obtained
in \cite{Sakai} and \cite{Sakai-Tsutsumi}, respectively.
In Section \ref{section: ODE and quasi} we will give a new proof of
these results, stated in the same form as \eqref{equation: example of
  KK}.


\subsection{Atkin-Lehner decomposition of $\wt M_k^{\le
    r}(\Gamma_0(N))$}
\label{subsection: Atkin-Lehner}

Again, for convenience of non-specialists, here we shall give a quick overview of the Atkin-Lehner decomposition of spaces of modular forms.

Let $N$ be a positive integer. It is easy to see that if $\gamma\in\SL(2,\R)$ normalizes $\Gamma_0(N)$, then $f\mapsto f|_k\gamma$
defines an automorphism of $\sM_k(\Gamma_0(N))$ and $\sM_k^\ast(\Gamma_0(N))$, where $f|_k\gamma$ is the usual slash operator of weight $k$ acting on $f$. Let $\mathrm{N}(\Gamma_0(N))$ denote the normalizer of $\Gamma_0(N)$ in $\SL(2,\R)$. Inside $\mathrm{N}(\Gamma_0(N))/\Gamma_0(N)$, there is a special subgroup, called the Atkin-Lehner subgroup, whose definition is given as follows.

Let $e$ be a positive divisor of $N$ such that $(e,N/e)=1$. Such a divisor is often called a Hall divisor of $N$. We check that a matrix of the form
$$
\frac1{\sqrt e}\M{ae}b{cN}{de}, \qquad ade^2+bcN=e
$$
normalizes $\Gamma_0(N)$. Let $W_e$ denote the set of such
matrices. (In particular, $W_1=\Gamma_0(N)$.) It is easy to see that
$W_e=w_e\Gamma_0(N)$ for any element $w_e$ of $W_e$. Also, if $e$ and
$e'$ are two Hall divisors of $N$, $\gamma\in W_e$, and $\gamma'\in
W_{e'}$, then $\gamma\gamma'\in W_{e''}$, where
$e''=ee'/(e,e')^2$. Hence, if we let
$$
\Gamma_0^+(N)=\bigcup_{e|N,(e,N/e)=1}W_e,
$$
then $\Gamma_0^+(N)$ is a subgroup of $\mathrm{N}(\Gamma_0(N))$ and
the quotient group $\Gamma_0^+(N)/\Gamma_0(N)$ is an elementary
abelian $2$-group of order $2^m$, where $m$ is the number of distinct
prime divisors of $N$. This quotient group $\Gamma_0^+(N)/\Gamma_0(N)$
is called the \emph{Atkin-Lehner subgroup} of
$\mathrm{N}(\Gamma_0(N))/\Gamma_0(N)$.

Elements of the Atkin-Lehner groups acts on many things related to $\Gamma_0(N)$. For example, they give rise to automorphisms of the modular curve $X_0(N)$. They also act on $\sM_k(\Gamma_0(N))$ and $\sM_k^\ast(\Gamma_0(N))$, as mentioned at the beginning of the section. In all case, they are called \emph{Atkin-Lehner involutions} and denoted by $w_e$. Since every $w_e$ has order $1$ or $2$ in $\mathrm{N}(\Gamma_0(N))/\Gamma_0(N)$ and they commute with each other, the space $\sM_k(\Gamma_0(N))$ (respectively, $\sM^\ast(\Gamma_0(N))$) decomposes into a direct sum of $2^m$ subspaces that are simultaneous eigenspaces with eigenvalues $\pm1$ for all Atkin-Lehner involutions. For convenience, we introduce the following notation.

\begin{Definition} Let $G_N$ denote the group of characters of
  $\Gamma_0^+(N)/\Gamma_0(N)$, that is, let $G_N$ be the group of all
  homomorphisms from $\Gamma_0^+(N)/\Gamma_0(N)$ to $\{\pm 1\}$. We
  shall let $\epsilon_0$ denote the trivial character of
  $\Gamma_0^+(N)/\Gamma_0(N)$.

  For $\epsilon\in G_N$, we let
  $$
  \sM_k(\Gamma_0^+(N),\epsilon)=\left\{f\in\sM_k(\Gamma_0(N)):
  f|_kw_e=\epsilon(W_e)f\text{ for all }e\right\}.
$$
The notation $\sM_k^\ast(\Gamma_0^+(N),\epsilon)$ is similarly
defined. (The choice of the notation $\sM_k(\Gamma_0^+(N),\epsilon)$
instead of $\sM_k(\Gamma_0(N),\epsilon)$ is to stress that we are
considering elements of $\sM_k(\Gamma_0(N),\epsilon)$ as modular forms
on $\Gamma_0^+(N)$ with characters.)
\end{Definition}

With this notation, the Atkin-Lehner decomposition we discussed above can be written as
$$
\sM_k(\Gamma_0(N))=\bigoplus_{\epsilon\in G_N}
\sM_k(\Gamma_0^+(N),\epsilon).
$$
We remark that if $\Gamma$ is a subgroup of $\Gamma_0^+(N)$ containing $\Gamma_0(N)$, then
$$
\sM_k(\Gamma)=\bigoplus_{\epsilon\in G_N,\Gamma/\Gamma_0(N)\in\ker\epsilon}
\sM_k(\Gamma_0^+(N),\epsilon)
$$
Thus, the dimension of each $\sM_k(\Gamma_0^+(N),\epsilon)$ can be
derived from the following general formula and the inclusion-exclusion
principle.

\begin{Proposition}[{\cite[Theorem 2.23]{Shimura-book}}]
  \label{proposition: dimension formula}
  Let $\Gamma$ be a subgroup of $\SL(2,\R)$ commensurable with
  $\SL(2,\Z)$. Suppose that the modular curve $X(\Gamma)$ has genus
  $g$, $c$ cusps, and $m$ elliptic points of order $e_1,\ldots,e_m$,
  respectively. Then for a nonnegative even integer $k$, we have
  $$
  \dim\sM_k(\Gamma)=(k-1)(g-1)+\sum_{j=1}^m\gauss{\frac k2\left(1-\frac1{e_j}
      \right)}+\frac{ck}2.
  $$
\end{Proposition}

\begin{Example} \label{example: Gamma0(2)+w2}
  Consider the case of $\Gamma_0(2)$. Since there is only one prime
  divisor here, we shall use the following more intuitive notations
  instead.
  
  Let $\sM_k(\Gamma_0^+(2),\pm)$ denote the Atkin-Lehner eigenspaces
  with eigenvalue $\pm1$, respectively. We observe that
  $\sM_k(\Gamma_0^+(2),+)$ is simply $\sM_k(\Gamma_0^+(2))$. To
  determine its dimension, we need to know how many cusps and elliptic
  points $X_0(2)/w_2$ has.
  
The modular curve $X_0(2)$ has genus $0$, $2$ cusps, and an elliptic
point $(1+i)/2$ of order $2$. Thus,
  $$
  \dim\sM_k(\Gamma_0(2))=1+\left\lfloor\frac k4\right\rfloor.
  $$
To determine how many cusps and elliptic points $X_0(2)/w_2$ has, we
note that $X_0(2)\to X_0(2)/w_2$ is a covering of degree $2$. Thus, by
the  Riemann-Hurwitz formula, the covering ramifies at $2$ points and
these two points are fixed points of $w_2$. Now the matrix
$\frac1{\sqrt2}\SM2{-1}20$ fixes $(1+i)/2$, i.e., the Atkin-Lehner
$w_2$ fixes $(1+i)/2$. Thus, $(1+i)/2$ becomes an elliptic point of
order $4$ on $X_0(2)/w_2$. Another fixed point of $w_2$ is $i/\sqrt2$,
fixed by $\frac1{\sqrt2}\SM0{-1}20$. Since we have found two ramified
points, all other points are unramified. We conclude that $X_0(2)/w_2$
has one cusp, one elliptic point of order $2$, and one elliptic point
of order $4$. Hence,
$$
\dim\sM_k(\Gamma_0^+(2),+)=\dim\sM_k(\Gamma_0^+(2))
=1-\frac k2+\left\lfloor\frac k4\right\rfloor
+\left\lfloor\frac{3k}8\right\rfloor.
$$
It follows that
$$
\dim\sM_k(\Gamma_0^+(2),-)=\dim\sM_k(\Gamma_0(2))
-\dim\sM_k(\Gamma_0^+(2),+)
=\frac k2-\left\lfloor\frac{3k}8\right\rfloor.
$$
\end{Example}

Now we consider the Atkin-Lehner decomposition of the spaces of nearly holomorphic modular forms and quasimodular forms. Just like $\sM_k(\Gamma_0(N))$, we can decompose $\sM_k^\ast(\Gamma_0(N))$ into a direct sum
$$
\sM_k^\ast(\Gamma_0(N))=\bigoplus_{\epsilon\in G_N}
\sM_k^\ast(\Gamma_0^+(N),\epsilon).
$$

\begin{Definition} For $\epsilon\in G_N$, we let
  $\wt\sM_k(\Gamma_0^+(N),\epsilon)$ denote the space of quasimodular
  forms that are holomorphic parts of functions in
  $\sM_k^\ast(\Gamma_0^+(N),\epsilon)$. 
  Also, for a nonnegative integer $r$, we let $\wt\sM_k^{\le
    r}(\Gamma_0^+(N),\epsilon)$ denote the subspace of
  $\wt\sM_k(\Gamma_0^+(N),\epsilon)$ consisting of quasimodular forms of
  depth $\le r$. When $\epsilon=\epsilon_0$ is the trivial character,
  we will often use the notation $\wt
  \sM_k^{\le r}(\Gamma_0^+(N))$ in place of $\wt\sM_k^{\le
    r}(\Gamma_0^+(N),\epsilon_0)$.

  In addition, when $N=p^n$ is a prime power, we will also use the
  more intuitive notations $\wt\sM_k^{\le r}(\Gamma_0^+(N),+)$
  (or simplfy $\wt\sM_k^{\le r}(\Gamma_0^+(N))$) and
  $\wt\sM_k^{\le r}(\Gamma_0^+(N),-)$ to denote the spaces
  $\wt\sM_k^{\le r}(\Gamma_0^+(N),\epsilon_0)$ and $\wt\sM_k^{\le
    r}(\Gamma_0^+(N),\epsilon)$, where $\epsilon$ is the nontrivial
  element in $G_N$.
\end{Definition}

\begin{Remark} \label{proposition: AL decomposition of quasi}
  Let $N$ be a positive integer. Let
  $$
  \phi(z)=\sum_{e|N,(e,N/e)=1}eE_2(ez),
  $$
  which is a scalar multiple of the logarithmic derivative of
  $\prod_e\Delta(ez)$ and hence a quasimodular form of weight $2$ and
  depth $1$ on $\Gamma_0^+(N)$.
  Then for $\epsilon\in G_N$, it is easy to see that
  $$
  \wt\sM_k^{\le r}(\Gamma_0^+(N),\epsilon)
  =\bigoplus_{j=0}^r\phi^j\sM_{k-2j}(\Gamma_0^+(N),\epsilon).
  $$
  and hence
  $$
  \dim\wt\sM_k^{\le r}(\Gamma_0^+(N),\epsilon)
  =\sum_{j=0}^r\dim\sM_{k-2j}(\Gamma_0^+(N),\epsilon).
  $$
  Note that the choice of $\phi$ is not unique
  in general. For example, when $N=4$, we can choose
  $\phi(z)=E_2(2z)$. Then the statement $\wt\sM_k^{\le
    r}(\Gamma_0^+(4),\epsilon)=\oplus_{j=0}^r\phi^j\sM_{k-2j}
  (\Gamma_0^+(4),\epsilon)$
  still holds. 
\end{Remark}


We will need the following property in the subsequent discussion.

\begin{Lemma} \label{lemma: Dq on AL} Let $\epsilon\in G_N$. We have
  $$
  D_q\wt\sM_k^{\le r}(\Gamma_0^+(N),\epsilon)\subseteq
  \wt\sM_{k+2}^{\le r+1}(\Gamma_0^+(N),\epsilon).
  $$
\end{Lemma}

\begin{proof} This follows from Lemma \ref{lemma: Maass}.
\end{proof}

\subsection{Extremal quasimodular forms}

Several examples given in the paper are involved with the notion of
extremal quasimodular forms, introduced first in \cite{Kaneko-Koike},
so here we shall review their definition and discuss the existence and
uniqueness of such quasimodular forms.

\begin{Definition} Given $\epsilon\in G_N$, let $\wt\sM=\wt\sM_k^{\le
    r}(\Gamma_0^+(N),\epsilon)$ be the Atkin-Lehner subspace of
  $\wt\sM_k^{\le r}(\Gamma_0(N))$ corresponding to $\epsilon$. An 
  element $f\in\wt\sM$ is said to be
  \emph{extremal} if its vanishing order at $\infty$
  is equal to $\dim\wt\sM-1$. We say
  $f$ is \emph{normalized} if its leading Fourier coefficient is $1$.
\end{Definition}

\begin{Remark} \label{remark: analytically extremal}
Note that in \cite{Pellarin}, Pellarin called a nonzero quasimodular form
$f$ \emph{analytically extremal} if its vanishing order at $\infty$ is
the largest among all nonzero elements of the space $f$ belongs
to. Pellarin's definition of extremality seems to be closer to the
literal meaning of the word extremal. However, we will stick to the
definition introduced by Kaneko and Koike. It is clear that if $f$ is
analytic extremal in a space $\wt\sM$ of quasimodular forms and its
vanishing order at $\infty$ is $\dim\wt\sM-1$, then it is extremal in
the sense of Kaneko and Koike.
\end{Remark}

In general, it is a complicated problem to show the existence and the
uniqueness of extremal quasimodular forms. In \cite{Pellarin},
Pellarin proved that if $r\le 4$, then a normalized extremal
quasimodular form in $\wt\sM_k^{\le r}(\SL(2,\Z))$ exists and is
unique. His argument also works in some other cases. For the
convenience of the reader, we reproduce the proof, adapted to the
current setting.

\begin{Lemma} \label{lemma: Wronskian is modular}
  Let $\Gamma$ be a subgroup of $\SL(2,\R)$ commensurable
  with $\SL(2,\Z)$ and $\chi$ be a character of $\Gamma$. Let $\phi$
  be a quasimodular form of weight $2$ and depth $1$ on $\Gamma$,
  i.e.,
  $$
  (cz+d)^{-2}\phi(\gamma z)=\phi(z)+\frac{\alpha c}{cz+d}
  $$
  for all $\gamma=\SM abcd\in\Gamma$ for some nonzero complex number
  $\alpha$. For
  $f(z)=f_0(z)+\phi(z)f_1(z)\in\wt\sM_k^{\le1}(\Gamma,\chi)$, where
  $k$ is a positive integer and $f_j\in\sM_{k-2j}(\Gamma,\chi)$,
  $j=1,2$, let 
  $$
  F_f(z)=\column{zf(z)+\alpha f_1(z)}{f(z)}
  $$
  and define the Wronskian $W_f(z)$ by
  $W_f(z)=\det(F_f(z),F_f'(z))$. Then $W_f(z)$ is a
  modular form of weight $2k$ with character $\chi^2$ on $\Gamma$.
  Moreover, we have
  \begin{equation} \label{equation: order of Wf at infinity}
  v_\infty(W_f)=\begin{cases}
    2v_\infty(f), &\text{if }v_\infty(f)\le v_\infty(f_1), \\
    v_\infty(f)+v_\infty(f_1), &\text{if }v_\infty(f)>v_\infty(f_1),
    \end{cases}
  \end{equation}
  where for a modular form or a quasimodular form $g$ we let
  $v_\infty(g)$ denote the vanishing order of $g$ at $\infty$.
\end{Lemma}

Note that in the case $\Gamma=\SL2,\Z)$, a proof of the fact that
$W_f(z)$ is a modular form of weight $2k$ can be found in
\cite{Grabner}. (This is also mentioned in \cite{Pellarin}, although
no proof is given.) Since our setting involves quasimodular forms with
characters on general $\Gamma$, we shall present a proof here.

\begin{proof}[Proof of Lemma \ref{lemma: Wronskian is modular}]
  First of all, by the definition of $W_f(z)$, we have
  $$
  W_f(z)=-f(z)^2+\alpha(f_1(z)f'(z)-f_1'(z)f(z)),
  $$
  from which \eqref{equation: order of Wf at infinity} immediately
  follows. Also, by (an extension of) Proposition \ref{proposition: sM
    dim}(i), the function $W_f(z)$ belongs to
  $\wt\sM_{2k}^{\le2}(\Gamma,\chi^2)$.
  On the other hand, for $\gamma=\SM abcd\in\Gamma$, we compute that
  $$
  f(\gamma z)=\chi(\gamma)(cz+d)^k\left(f(z)
    +\frac{\alpha c}{cz+d}f_1(z)\right)
  $$
  and
  \begin{equation*}
    \begin{split}
      (\gamma z)f(\gamma z)+\alpha f_1(\gamma z)
      &=\chi(\gamma)\frac{az+b}{cz+d}(cz+d)^k
      \left(f(z)+\frac{\alpha c}{cz+d}f_1(z)\right) \\
      &\qquad+\chi(\gamma)\alpha(cz+d)^{k-2}f_1(z).
    \end{split}
  \end{equation*}
  Using the relation $ad-bc=1$, we may simplify the last expression to
  $$
  \chi(\gamma)(cz+d)^{k-1}\left((az+b)f(z)+\alpha af_1(z)\right).
  $$
  Hence,
  $$
  F_f(\gamma z)=\chi(\gamma)(cz+d)^{k-1}\M abcd F_f(z)
  $$
  and
  $$
  F_f'(\gamma z)=\chi(\gamma)\M abcd\left(
    (k-1)(cz+d)^kF_f(z)+(cz+d)^{k+1}F_f'(z)\right).
  $$
  From the computation, we see that the Wronskian $W_f(z)$ satisfies
  $$
  W_f(\gamma z)=\chi(\gamma)^2(cz+d)^{2k}W_f(z).
  $$
  That is, $W_f(z)$ is a quasimodular form of weight $2k$ and depth
  $\le 2$ with character $\chi^2$ on $\Gamma$ that is actually
  modular. This proves that $W_f(z)\in\sM_{2k}(\Gamma,\chi^2)$.
\end{proof}

We now prove an extension of \cite[Theorem 1.3]{Pellarin}.

\begin{Proposition} \label{proposition: Pellarin}
  Suppose that $\wt\sM_k^{\le
    1}(\Gamma_0^+(N),\epsilon)$ is an Atkin-Lehner subspace of
  $\wt\sM_k^{\le 1}(\Gamma_0(N))$ such that
  \begin{equation} \label{equation: dim assumption}
    \dim\sM_{2k}(\Gamma_0^+(N))
    =\dim\wt\sM_k^{\le 1}(\Gamma_0^+(N),\epsilon)
  \end{equation}
  and the maximal vanishing order at $\infty$ of a nonzero element of
  $\sM_{2k}(\Gamma_0^+(N))$ is equal to
  $\dim\sM_{2k}(\Gamma_0^+(N))-1$.
  Then
  \begin{enumerate}
    \item[(i)] extremal quasimodular forms exist in $\wt\sM_k^{\le
    1}(\Gamma_0^+(N),\epsilon)$ and are unique up to scalars, and
    \item[(ii)] if $f$ is an extremal quasimodular form in
      $\wt\sM_k^{\le 1}(\Gamma_0^+(N),\epsilon)$, then we have
      $v_\infty(f)=v_\infty(W_f)$, where $W_f$ is the Wronskian
      associated to $f$ defined in Lemma \ref{lemma: Wronskian is
        modular}.
    \end{enumerate}
\end{Proposition}

\begin{proof} The proof follows that of \cite[Theorem 1.3]{Pellarin}.
  Set $\wt\sM=\wt\sM_k^{\le1}(\Gamma_0^+(N),\epsilon)$. Let
  $$
  \phi(z)=\sum_{e|N,(e,N/e)=1}eE_2(ez),
  $$
  which is a scalar multiple of the logarithmic derivative of
  $\prod_e\Delta(ez)$. It is a holomorphic quasimodular form of weight
  $2$ and depth $1$ on $\Gamma_0^+(N)$ and satisfies
  $$
  \left(\phi\big|_2\gamma\right)(z)=\phi(z)+\frac{\alpha c}{cz+d}
  $$
  for all $\SM abcd\in\Gamma_0^+(N)$ for some nonzero complex number
  $\alpha$. Assume that $f$ is an element of
  $\wt\sM$. According to Remark
  \ref{proposition: AL decomposition of quasi}, we have $f=f_0+\phi
  f_1$ for some $f_j\in\sM_{k-2j}(\Gamma_0^+(N),\epsilon)$, $j=1,2$.
  Let $W_f(z)$ be
  Wronskian associated to $f$ defined in Lemma \ref{lemma: Wronskian
    is modular}. By the lemma, $W_f(z)$ is a modular form in
  $\sM_{2k}(\Gamma_0^+(N))$. (Note that $\epsilon^2=\epsilon_0$ for
  any $\epsilon\in G_N$.) Moreover, \eqref{equation: order of Wf at
    infinity} implies that $W_f\neq 0$ whenever $f\neq 0$. Another
  consequence of \eqref{equation: order of Wf at infinity} is that, by
  the assumptions that \eqref{equation: dim assumption} holds and that
  the maximal vanishing order of any nonzero element of
  $\sM_{2k}(\Gamma_0^+(N))$ is
  $\dim\sM_{2k}(\Gamma_0^+(N))-1$, we have
  \begin{equation*}
    \begin{split}
      \dim\wt\sM-1
      &=\dim\sM_{2k}(\Gamma_0^+(N))-1\ge v_\infty(W_f)\ge v_\infty(f)
    \end{split}
  \end{equation*}
  for all nonzero elements $f$ of $\wt\sM$. This implies that the
  vanishing order at $\infty$ of an analytically extremal quasimodular
  form in $\wt\sM$ is $\dim\wt\sM-1$ and hence extremal quasimodular
  forms exists in $\wt\sM=\wt\sM_k^{\le1}(\Gamma_0^+(N),\epsilon)$ and
  are unique up to scalars. The argument above also shows that if $f$
  is an extremal quasimodular form in
  $\wt\sM_k^{\le1}(\Gamma_0^+(N),\epsilon)$, then
  $v_\infty(f)=v_\infty(W_f)$.
\end{proof}

Now we use this result to prove the existence and uniqueness of
extremal quasimodular forms in the case $N=2$ or $N=3$. Note that
Sakai \cite{Sakai} already established the existence in the case of
$\wt\sM_k(\Gamma_0^+(N),+)$, $N=2,3$.

\begin{Corollary} \label{corollary: existence of extremal}
  Let $N=2$ or $N=3$. Then for all positive even
  integers $k$, extremal quasimodular forms exist in
  $\wt\sM_k^{\le1}(\Gamma_0^+(N))$ and
  $\wt\sM_k^{\le1}(\Gamma_0^+(N),-)$ and are unique up to scalars.
\end{Corollary}

\begin{proof} In Example \ref{example: Gamma0(2)+w2}, we have seen
  that
  \begin{equation} \label{equation: Gamma0(2) dims}
  \begin{split}
  \dim\sM_k(\Gamma_0(2))&=1+\left\lfloor\frac k4\right\rfloor, \\
  \dim\sM_k(\Gamma_0^+(2))&
  =1-\frac k2+\left\lfloor\frac k4\right\rfloor
+\left\lfloor\frac{3k}8\right\rfloor, \\
\dim\sM_k(\Gamma_0^+(2),-)&=\frac k2-\left\lfloor\frac{3k}8\right\rfloor.
\end{split}
\end{equation}
Hence, according to Remark \ref{proposition: AL decomposition of
  quasi}, we have
\begin{equation} \label{equation: dim Gamma0(2) +}
\dim\wt\sM^{\le 1}_k(\Gamma_0^+(2))
=\begin{cases}
  1+k/4, &\text{if }k\equiv 0\mod 8, \\
  (k+2)/4, &\text{if }k\equiv 2,6\mod 8, \\
  k/4, &\text{if }k\equiv 4\mod 8,
  \end{cases}
\end{equation}
and
\begin{equation} \label{equation: dim Gamma0(2) -}
\dim\wt\sM^{\le 1}_k(\Gamma_0^+(2),-)
=\begin{cases}
  k/4, &\text{if }k\equiv 0\mod 8, \\
  (k+2)/4, &\text{if }k\equiv 2,6\mod 8, \\
  1+k/4, &\text{if }k\equiv 4\mod 8.
  \end{cases}
  \end{equation}
  It follows that if
  $k\not\equiv4\mod 8$, then
  $$
  \dim\sM_{2k}(\Gamma_0^+(2))=\dim\wt\sM_k^{\le 1}(\Gamma_0^+(2))
  $$
  and if $k\not\equiv0\mod 8$, then
  $$
  \dim\sM_{2k}(\Gamma_0^+(2))=\dim\wt\sM_k^{\le 1}(\Gamma_0^+(2),-).
  $$
  The condition \eqref{equation: dim assumption} holds for these
  spaces.

  For the other condition of Proposition \ref{proposition:
    Pellarin}, we observe that the condition is equivalent to
  that there is a basis $\{g_j\}$ for $\sM_{2k}(\Gamma_0^+(2))$ such
  that $v_\infty(g_j)=j-1$ for each
  $j=1,\ldots,\dim\sM_{2k}(\Gamma_0^+(2))$. Indeed, we
  find that $g_j(z)=M_4(z)^{k/2-2j}M_8(z)^j$ for
  $j=0,\ldots,\gauss{k/4}$, where
  $$
  M_4(z)=\frac15(4E_4(2z)+E_4(z)), \qquad M_8(z)=\eta(z)^8\eta(2z)^8,
  $$
  form such a basis. Therefore, by Proposition \ref{proposition:
    Pellarin}, extremal quasimodular forms exist and are unique up to
  scalars in the spaces above. For the remaining two cases
  $\wt\sM_k^{\le1}(\Gamma_0^+(2))$ for $k\equiv 4\mod 8$ and
  $\wt\sM_k^{\le1}(\Gamma_0^+(2),-)$ for 
  $k\equiv 0\mod 8$, we observe that when $k\equiv 4\mod 8$, we have
  $$
  \dim\wt\sM_{k-4}^{\le1}(\Gamma_0^+(2))
  =\dim\wt\sM_k^{\le 1}(\Gamma_0^+(2)),
  $$
  and the map
  $$
  f(z)\longmapsto M_4(z)f(z)
  $$
  defines an isomorphism between the two spaces. It is clear that a
  quasimodular form $f(z)$
  is extremal in $\wt\sM_{k-4}^{\le1}(\Gamma_0^+(2))$ if and only if
  $M_4(z)f(z)$ is extremal in 
  $\wt\sM_k^{\le1}(\Gamma_0^+(2))$. Since we have seen that extremal
  quasimodular forms exist and are unique up to scalars in the former
  space, the same thing holds for the latter space. The same argument
  works for the last case $\wt\sM_k^{\le 1}(\Gamma_0^+(2),-)$, $k\equiv
  0\mod 8$. This proves the case $N=2$.

  For the case $N=3$, we note that $\Gamma_0(3)$ has two cusps and one
  elliptic point of order $3$, while $\Gamma_0^+(3)$ has one cusp, one
  elliptic point of order $2$, represented by $i/\sqrt3$ with
  stabilizer subgroup generated by $\frac1{\sqrt3}\SM0{-1}30$, and one
  elliptic point of order $6$, represented by $(3+\sqrt{-3})/6$ with
  stabilizer subgroup generated by $\frac1{\sqrt3}\SM3{-1}30$.
  Thus, by Proposition \ref{proposition: dimension formula}, we have
  \begin{equation} \label{equation: Gamma0(3) dims}
    \begin{split}
      \dim\sM_k(\Gamma_0(3))&=1+\gauss{\frac k3}, \\
      \dim\sM_k(\Gamma_0^+(3))&=1-\frac k2+\gauss{\frac k4}
      +\gauss{\frac{5k}{12}}, \\
      \dim\sM_k(\Gamma_0^+(3),-)&=\frac k2+\gauss{\frac k3}
      -\gauss{\frac k4}-\gauss{\frac{5k}{12}}.
    \end{split}
  \end{equation}
  From these formulas, we deduce that
  $$
  \dim\wt\sM_k^{\le1}(\Gamma_0^+(3))
  =\begin{cases}
    1+k/3, &\text{if }k\equiv 0\mod 12,\\
    (k+1)/3, &\text{if }k\equiv 2,8\mod 12, \\
    (k-1)/3, &\text{if }k\equiv 4\mod 12, \\
    k/3, &\text{if }k\equiv 6\mod 12, \\
    (k+2)/3, &\text{if }k\equiv 10\mod 12,\end{cases}
  $$
  and
  $$
  \dim\wt\sM_k^{\le1}(\Gamma_0^+(3),-)=\begin{cases}
    k/3, &\text{if }k\equiv 0\mod 12,\\
    (k+1)/3, &\text{if }k\equiv 2,8\mod 12, \\
    (k+2)/3, &\text{if }k\equiv 4\mod 12, \\
    1+k/3, &\text{if }k\equiv 6\mod 12, \\
    (k-1)/3, &\text{if }k\equiv 10\mod 12,\end{cases}
  $$
  Therefore, the condition \eqref{equation: dim assumption} in
  Proposition \ref{proposition: Pellarin} 
  is met for $\wt\sM_k^{\le1}(\Gamma_0^+(3))$ with
  $k\equiv0,2,8,10\mod 12$ and $\wt\sM_k^{\le1}(\Gamma_0^+(3),-)$ with
  $k\equiv 2,4,6,8\mod 12$. For the other condition of the
  proposition, we may use
  \begin{equation*}
    \begin{split}
  M_4(z)&=\frac1{10}(9E_4(3z)+E_4(z)), \\
  M_8(z)&=\frac12(3E_2(3z)-E_2(z))\eta(z)^6\eta(3z)^6, \\
  M_{12}(z)&=\eta(z)^{12}\eta(3z)^{12}
    \end{split}
  \end{equation*}
  to construct a basis $\{g_j\}$ for $\sM_{2k}(\Gamma_0^+(3))$ with
  the property $v_\infty(g_j)=j-1$ for all
  $j=1,\ldots,\dim\sM_{2k}(\Gamma_0^+(3))$. Hence, the second
  condition of Proposition \ref{proposition: Pellarin} also holds and
  we conclude that extremal quasimodular forms exist in
  $\wt\sM_k^{\le1}(\Gamma_0^+(3))$ for $k\equiv 0,2,8,10\mod 12$ and
  $\wt\sM_k^{\le1}(\Gamma_0^+(3),-)$ for $k\equiv 2,4,6,8\mod 12$ and
  are unique up to scalars.
  
  Finally, similar to the case $N=2$, for
  the remaining cases, we observe that
  \begin{equation} \label{equation: equal dim 1}
    \begin{split}
  \dim\wt\sM_{k-4}^{\le1}(\Gamma_0^+(3))
  &=\dim\wt\sM_k^{\le1}(\Gamma_0^+(3)),\
  \text{for }k\equiv 4\mod12, \\
  \dim\wt\sM_{k-4}^{\le1}(\Gamma_0^+(3),-)
  &=\dim\wt\sM_k^{\le1}(\Gamma_0^+(3),-), \
  \text{for }k\equiv 10\mod 12,
    \end{split}
  \end{equation}
  \begin{equation} \label{equation: equal dim 2}
    \begin{split}
  \dim\wt\sM_{k-2}^{\le1}(\Gamma_0^+(3),-)
  &=\dim\wt\sM_k^{\le1}(\Gamma_0^+(3)),\
  \text{for }k\equiv 6\mod12, \\
  \dim\wt\sM_{k-2}^{\le1}(\Gamma_0^+(3))
  &=\dim\wt\sM_k^{\le1}(\Gamma_0^+(3),-), \
  \text{for }k\equiv 0\mod 12,
    \end{split}
  \end{equation}
  and that the map
  $$
  f(z)\longmapsto(9E_4(3z)+E_4(z))f(z)
  $$
  defines an isomorphism for the pairs of spaces in \eqref{equation:
    equal dim 1}, while the map
  $$
  f(z)\longmapsto(3E_2(3z)-E_2(z))f(z)
  $$
  defines an isomorphism for the pairs of spaces in \eqref{equation:
    equal dim 2}. From these observations we conclude that extremal
  quasimodular forms exists and are unique up to scalars in all
  $\wt\sM_k^{\le1}(\Gamma_0^+(3))$ and $\wt\sM_k^{\le
    1}(\Gamma_0^+(3),-)$.
\end{proof}

\subsection{Quasimodular forms of odd weights on $\Gamma_1(3)$}
\label{subsection: odd}
In this paper, quasimodular forms of odd weights on
$$
\Gamma_1(3):=\left\{\M abcd\in\SL(2,\Z):c\equiv0\mod3,
  a,d\equiv1\mod 3\right\}
$$
will also appear as
solutions of modular differential equations when
$\Gamma=\Gamma_0^+(3)$, so we will review
their properties in this section.

For a nonnegative integer $k$ (even or odd), let
$\sM_k(\Gamma_1(3))$ denote the set of
modular forms of weight $k$ on $\Gamma_1(3)$.
Note that when $k$ is odd, modular forms of weight $k$ on
$\Gamma_1(3)$ can also be defined as modular forms on $\Gamma_0(3)$
with nebentype character $\chi\left(\SM abcd\right)=\JS{-3}d$.
By Theorem 2.25 of \cite{Shimura-book}, we have
\begin{equation} \label{equation: dimension Gamma1(3)}
\dim\sM_k(\Gamma_1(3))=1+\gauss{\frac k3}.
\end{equation}
Hence, we have
$$
\sum_{k=0}^\infty\dim\sM_k(\Gamma_1(3))x^k
=\frac1{(1-x)(1-x^3)},
$$
and the graded ring $\bigoplus_{k=0}^\infty\sM(\Gamma_1(3))$
of modular forms is freely generated by
\begin{equation} \label{equation: M1 Gamma0+(3)}
M_1(z)=\sum_{(m,n)\in\Z^2}q^{m^2+mn+n^2}, \qquad q=e^{2\pi iz},
\end{equation}
and
\begin{equation} \label{equation: M3 Gamma0+(3)}
  M_3(z)=\frac{9E_4(3z)-E_4(z)}{8M_1(z)}
=\frac{\eta(z)^9}{\eta(3z)^3}-27\frac{\eta(3z)^9}{\eta(z)^3}.
\end{equation}


Similar to the case of $\Gamma_0(N)$, the Atkin-Lehner involutions
normalize $\Gamma_1(N)$. Hence, we may decompose $\sM_k(\Gamma_1(3))$
into eigenspaces of the Atkin-Lehner involution $w_3$. Let
$S=\frac1{\sqrt3}\SM0{-1}{3}0$. Because $S^2=-I$, the eigenvalues of
$w_3$ in the case of odd $k$ can only be $\pm i$. For convenience, we
shall consider the eigenfunctions as modular forms
on $\Gamma_0^+(3)$ with characters. More concretely, we recall that
$\Gamma_0^+(3)$ is generated by
$$
S=\frac1{\sqrt3}\M0{-1}30, \qquad
R=TS=\frac1{\sqrt3}\M3{-1}30,
$$
where $T=\SM1101$. They satisfy $S^2=R^6=-I$
and there is no other relation between $S$ and $R$. In other words, a
presentation of the group $\Gamma_0^+(3)$ is
$\gen{S,R:S^4=1,S^2=R^6}$. We let $\chi$ be the character of
$\Gamma_0^+(3)$ such that
\begin{equation} \label{equation: chi}
  \chi(S)=\chi(R)=-i,
\end{equation}
and $\overline\chi$ be
its complex conjugate. Then the eigenspace of $w_3$ with eigenvalue
$-i$ (respectively, $i$) can be identified with
$\sM(\Gamma_0^+(3),\chi)$ (respectively,
$\sM(\Gamma_0^+(3),\overline\chi)$). Using the transformation formula
for the Dedekind eta function (see \cite[Pages 125--127]{Weber}), we
can show that
$$
M_3(z)\in\sM_3(\Gamma_0^+(3),\chi)
$$
and hence also
$$
M_1(z)\in\sM_1(\Gamma_0^+(3),\chi),
$$
where $M_1(z)$ and $M_3(z)$ are defined by \eqref{equation: M1
  Gamma0+(3)} and \eqref{equation: M3 Gamma0+(3)}, respectively.

\begin{Lemma} \label{lemma: dimension Gamma0+(3) chi}
  Let $k$ be a positive odd integer. Then
  $$
  \dim\sM_k(\Gamma_0^+(3),\chi)=\begin{cases}
    (k+5)/6, &\text{if }k\equiv 1\mod 12,\\
    (k+3)/6, &\text{if }k\equiv 3,9\mod 12, \\
    (k+1)/6, &\text{if }k\equiv 5,11\mod 12, \\
    (k-1)/6, &\text{if }k\equiv 7\mod 12, \end{cases}
  $$
  and
  $$
  \dim\sM_k(\Gamma_0^+(3),\overline\chi)=\begin{cases}
    (k-1)/6, &\text{if }k\equiv 1\mod 12,\\
    (k+3)/6, &\text{if }k\equiv 3,9\mod 12, \\
    (k+1)/6, &\text{if }k\equiv 5,11\mod 12, \\
    (k+5)/6, &\text{if }k\equiv 7\mod 12. \end{cases}
  $$
\end{Lemma}

\begin{proof} Observe that
  $\sM_{k-1}(\Gamma_0^+(3))\cup\sM_{k-1}(\Gamma_0^+(3),-)
  =\sM_{k-1}(\Gamma_0(3))$
  and $\sM_k(\Gamma_0^+(3),\chi)\cup
  \sM_k(\Gamma_0^+(3),\overline\chi)=\sM_k(\Gamma_1(3))$. Also,
  \begin{equation*}
    \begin{split}
      M_1\sM_{k-1}(\Gamma_0^+(3))&\subseteq\sM_k(\Gamma_0^+(3),\chi),\\
      M_1\sM_{k-1}(\Gamma_0^+(3),-)&\subseteq
      \sM_k(\Gamma_0^+(3),\overline\chi).
    \end{split}
  \end{equation*}
  When $k$ is not a multiple of $3$, by \eqref{equation: dimension
  Gamma1(3)}, we have
  $$
  \dim\sM_k(\Gamma_1(3))=\dim\sM_{k-1}(\Gamma_0(3)).
  $$
  Hence, both inclusions above are equality. When $3|k$, we use
  \begin{equation*}
    \begin{split}
      M_1\sM_k(\Gamma_0^+(3),\chi)&\subseteq\sM_{k+1}(\Gamma_0^+(3),-),\\
      M_1\sM_k(\Gamma_0^+(3),\overline\chi)&\subseteq
      \sM_{k+1}(\Gamma_0^+(3))
    \end{split}
  \end{equation*}
  instead. Again, by the dimension formula \eqref{equation: dimension
  Gamma1(3)}, both inclusions are equality. This proves the lemma.
\end{proof}

Passing to quasimodular forms, we have the following dimension
formulas.

\begin{Lemma} Let $k$ be a positive odd integer. We have
  $$
  \dim\wt\sM_k^{\le1}(\Gamma_0^+(3),\chi)
  =\begin{cases}
    (k+2)/3, &\text{if }k\equiv1\mod 12, \\
    1+k/3, &\text{if }k\equiv 3\mod 12, \\
    (k+1)/3, &\text{if }k\equiv5,11\mod 12,\\
    (k-1)/3, &\text{if }k\equiv 7\mod 12, \\
    k/3, &\text{if }k\equiv 9\mod 12, \end{cases}
  $$
  and
  $$
  \dim\wt\sM_k^{\le1}(\Gamma_0^+(3),\overline\chi)
  =\begin{cases}
    (k-1)/3, &\text{if }k\equiv1\mod 12, \\
    k/3, &\text{if }k\equiv 3\mod 12, \\
    (k+1)/3, &\text{if }k\equiv5,11\mod 12,\\
    (k+2)/3, &\text{if }k\equiv 7\mod 12, \\
    1+k/3, &\text{if }k\equiv 9\mod 12. \end{cases}
  $$
\end{Lemma}

\begin{Proposition} For each positive odd integer $k$, an extremal
  quasimodular form exists in $\wt\sM_k(\Gamma_0^+(3),\chi)$ or
  $\wt\sM_k(\Gamma_0^+(3),\overline\chi)$ and is unique up to scalars.
\end{Proposition}

\begin{proof}
For a quasimodular form $f$ in $\wt\sM_k^{\le1}(\Gamma_0^+(3),\chi)$
or $\wt\sM_k^{\le1}(\Gamma_0^+(3),\overline\chi)$, let $W_f$ be the
Wronskian associated $f$ defined in Lemma \ref{lemma: Wronskian is
  modular}. By the lemma, we have
$W_f\in\sM_{2k}(\Gamma_0^+(3),\chi^2)=\sM_{2k}(\Gamma_0^+(3),-)$.
Since
$$
\dim\sM_{2k}(\Gamma_0^+(3),-)
=\begin{cases}
  (k+2)/3, &\text{if }k\equiv 1\mod 6, \\
  1+k/3, &\text{if }k\equiv 3\mod 6, \\
  (k+1)/3, &\text{if }k\equiv 5\mod 6, \end{cases}
$$
we find that
$$
\dim\wt\sM_k^{\le1}(\Gamma_0^+(3),\chi)=\dim\sM_{2k}(\Gamma_0^+(3),-)
$$
for $k\equiv 1,3,5,11\mod 12$, and
$$
\dim\wt\sM_k^{\le1}(\Gamma_0^+(3),\overline\chi)
=\dim\sM_{2k}(\Gamma_0^+(3),-)
$$
for $k\equiv5,7,9,11\mod 12$. Hence, extremal quasimodular forms
exist in these spaces and is unique up to scalars. For the remaining
four spaces, we use the facts that
$$
M_1\wt\sM_{k-1}^{\le1}(\Gamma_0^+(3))
=\wt\sM_k^{\le1}(\Gamma_0^+(3),\chi)
$$
for $k\equiv 1,7,9,11\mod 12$, and
$$
M_1\wt\sM_{k-1}^{\le1}(\Gamma_0^+(3),-)
=\wt\sM_k^{\le1}(\Gamma_0^+(3),\overline\chi)
$$
for $k\equiv1,3,5,7\mod 12$ to obtain the same conclusion.
\end{proof}

\section{Modular ordinary differential equations}
\label{section: ODE and quasi}

The purpose of this section is to prove that for any quasimodular form
$f(z)$ of depth $1$, $f(z)/\sqrt{W_f(z)}$ is a solution of a modular
differential equation. Throughout
the section, we assume that $\Gamma$ is a discrete subgroup of
$\SL(2,\R)$ that is commensurable with $\SL(2,\Z)$, and $\phi$ is the
quasimodular form of weight $2$ and depth $1$ appearing in Proposition
\ref{proposition: sM dim}, i.e.,
\begin{equation} \label{equation: phi transform}
  (cz+d)^{-2}\phi(\gamma z)=\phi(z)+\frac{\alpha c}{cz+d}, \qquad
  \gamma=\M abcd\in\Gamma,
\end{equation}
for some complex number $\alpha\neq 0$. Suppose that $f(z)$ is a
quasimodular form of weight $k$ and depth $1$ with character
$\chi$ of finite order on $\Gamma$. We have
\begin{equation} \label{equation: f=f1phi+f0}
  f(z)=f_1(z)\phi(z)+f_0(z)
\end{equation}
for some $f_0\in\sM_k(\Gamma,\chi)$ and
$f_1\neq 0\in\sM_{k-2}(\Gamma,\chi)$. Let $W_f(z)$ be the Wronskian
associated to $f$ defined in Lemma \ref{lemma: Wronskian is
  modular}. It is a modular form of weight $2k$ with character
$\chi^2$ on $\Gamma$. Define
\begin{equation} \label{equation: g1, g2}
g_1(z)=\frac{zf(z)+\alpha f_1(z)}{\sqrt{W_f(z)}}
\quad\text{and}\quad
g_2(z)=\frac{f(z)}{\sqrt{W_f(z)}}.
\end{equation}
Then we have $\det\SM{g_1}{g_2}{g_1'}{g_2'}=1$ and hence
$\det\SM{g_1}{g_2}{g_1''}{g_2''}=0$. Therefore, both $g_1$ and $g_2$
are solutions of
\begin{equation} \label{equation: y''=-Qy sect 3}
  y''(z)=-4\pi^2Q(z)y(z),
\end{equation}
where
\begin{equation} \label{equation: Q=g''/g}
Q(z)=-\frac1{4\pi^2}\frac{g_1''(z)}{g_1(z)}
=-\frac1{4\pi^2}\frac{g_2''(z)}{g_2(z)}.
\end{equation}
Clearly, $Q(z)$ is a single-valued meromorphic function not
identically $0$ on $\H$.


\begin{Proposition} Let $Q(z)$ be defined by \eqref{equation:
    Q=g''/g}. Then $Q(z)$ is a meromorphic function on $\H$ and
  \eqref{equation: y''=-Qy sect 3} is Fuchsian. Moreover, if
  $z_0\in\H$ is a pole of $Q(z)$, then $z_0$ is an apparent
  singularity for \eqref{equation: y''=-Qy sect 3}.
\end{Proposition}

\begin{proof} Clearly, $g_i(z)=(z-z_0)^{\alpha_i}(c_i+O(z-z_0))$ near
  $z_0$ for some $c_i\neq 0$ and $\alpha_i\in\frac12\Z$. Thus,
  $$
  Q(z)=\alpha_i(\alpha_i-1)(z-z_0)^{-2}+O\left((z-z_0)^{-1}\right)
  $$
  near $z_0$ and we have $\alpha_1(\alpha_1-1)=\alpha_2(\alpha_2-1)$,
  which implies that the local exponents at $z_i$ are $\alpha_i$ and
  $-\alpha_i+1$. (The two sets $\{\alpha_1,-\alpha_1+1\}$ and
  $\{\alpha_2,-\alpha_2+1\}$ are identical.) Hence \eqref{equation:
    y''=-Qy sect 3} is Fuchsian and the difference of the local
  exponents are $1+2\alpha_i\in\Z$. Since $g_1$ and $g_2$ are linearly
  independent and has no logarithmic singularities near $z_0$, $z_0$
  is obviously an apparent singularity.
\end{proof}

\begin{Remark} \label{remark: apparent}
  Suppose that \eqref{equation: y''=-Qy sect 3} has a
  solution $y(z)$ with a logarithmic singularity near $z_0$. Then the
  local exponents $\rho_1$ and $\rho_2$ at $z_0$ satisfy
  $\rho_1-\rho_2\in\Z$, and $z_0$ is not apparent. This remark is used
  in the proof of Theorem \ref{theorem: Q is modular}. See below.
\end{Remark}

The following result is a generalization of Theorem \ref{proposition:
  quasi satisfies ODE}.

\begin{Theorem} \label{theorem: Q is modular}
  Let $Q(z)$ be defined by \eqref{equation:
    Q=g''/g}. Then the following hold.
  \begin{enumerate}
    \item[(i)] The function $Q(z)$ is a meromorphic modular form of
      weight $4$ (with trivial character) on $\Gamma$.
    \item[(ii)] Moreover, $Q(z)$ is holomorphic at $\infty$ with
      $Q(\infty)=\kappa_\infty^2/N^2$ for some
      $\kappa_\infty\in\frac12\Z_{\ge 0}$, where $N$ is the width of
      the cusp $\infty$ of $\Gamma$. Also, the cusp $\infty$ is not
      apparent for \eqref{equation: y''=-Qy sect 3}.
    \item[(iii)] For any cusp $s$, $Q(z)$ is holomorphic at $s$ with
      $Q(s)\ge 0$ and $s$ is not apparent for \eqref{equation: y''=-Qy
        sect 3}.
  \end{enumerate}
\end{Theorem}

\begin{proof} To prove the modularity of $Q(z)$, we consider
  $y(z)=\left(g_2|_{-1}\gamma\right)(z):=(cz+d)g_2(\gamma z)$ for
  $\gamma=\SM abcd\in\Gamma$. By the
  definition of $g_2$, $y(z)$ is well-defined up to $\pm 1$.
  A direct computation show that
  $$
  \left(g_2\big|_{-1}\gamma\right)''(z)
  =\left(g_2''\big|_3\gamma\right)(z)
  $$
  (which is a special case of Bol's identity \cite{Bol}),
  which implies that
  $$
  y''(z)=\left(Q\big|_4\gamma\right)(z)y(z)
  $$
  On the other hand, we have
  \begin{equation*}
    \begin{split}
      y(z)&=(cz+d)g_2(\gamma z)
      =\frac{(cz+d)f(\gamma z)}{\sqrt{W_f(\gamma z)}} \\
      &=\pm\frac{\chi(\gamma)(cz+d)(f(z)+\alpha cf_1(z)/(cz+d))}
      {\sqrt{\chi(\gamma)^2W_f(z)}}\\
      &=\pm(cg_1(z)+dg_2(z)).
    \end{split}
  \end{equation*}
  Thus, $y(z)$ is also a solution of \eqref{equation: y''=-Qy sect 3},
  which implies that $\left(Q|_4\gamma\right)(z)=Q(z)$ for all
  $\gamma\in\Gamma$. This proves (i).

  Let $N$ be the width of the cusp $\infty$ and $q_N=e^{2\pi iz/N}$.
  By \eqref{equation: order
    of Wf at infinity}, the order of $g_2$ at $\infty$ is
  $$
  \kappa_\infty=\frac12v_\infty(f)
  -\frac12\min(v_\infty(f),v_\infty(f_1)),
  $$
  which is in $\frac12\Z_{\ge 0}$. (Note that 
  when $\chi\left(\SM1N01\right)\neq 1$, $v_\infty(f)$ and
  $v_\infty(f_1)$ are not integers, but we still have
  $v_\infty(f)\equiv v_\infty(f_1)\mod 1$ nonetheless.) Since
  \begin{equation} \label{equation: Q at infinity}
  Q(z)=-\frac1{4\pi^2}\frac{g_2''(z)}{g_2(z)}
  =\frac1{N^2}\frac{\left(q_N\frac d{dq_N}\right)^2g_2}{g_2},
  \end{equation}
  we find that $Q(\infty)=\kappa_\infty^2/N^2$.
  
  Notice that $z$ appears in the expression of $g_1(z)$ and $g_1(z)$
  is another solution of \eqref{equation: y''=-Qy sect 3}. This
  obviously implies that $\infty$ is not apparent. We now prove
  (iii).

  Let $s\neq\infty$ be another cusp of $\Gamma$ and $\sigma=\SM abcd$
  be an element of $\SL(2,\Z)$ such that $\sigma\infty=s$. Regarding
  $f(z)$ as a quasimodular form on $\Gamma'=\ker\chi\cap\SL(2,\Z)$, we
  have $f(z)=h_1(z)E_2(z)+h_0(z)$ for some $h_j\in\sM_{k-2j}(\Gamma')$.
  We have
  \begin{equation*}
    \begin{split}
      \left(g_2|_{-1}\sigma\right)(z)
      &=\frac{(cz+d)f(\sigma z)}{\sqrt{W_f(\sigma z)}}
      =\frac{(cz+d)\left(f|_k\sigma\right)(z)}
      {\sqrt{\left(W_f|_{2k}\sigma\right)(z)}} \\
      &=\frac{(cz+d)\left(\left(h_1|_{k-2}\sigma\right)
          E_2(z)+\left(h_0|_k\sigma\right)\right)
        +6c\left(h_1|_{k-2}\sigma\right)/\pi i}
      {\sqrt{\left(W_f|_{2k}\sigma\right)(z)}}.
    \end{split}
  \end{equation*}
  Since all terms in the above expression except for $cz+d$ has a
  $q_M$-expansion, where $M$ is the width of the cusp $s$, we see that
  $z$ appears in the expression of $g_2|_{-1}\gamma$. (Note that
  $c\neq 0$ since $s$ is assumed to be a cusp different from
  $\infty$.) This implies that $s$ is not apparent and the local
  exponents of \eqref{equation: y''=-Qy sect 3} at $s$ are
  $\pm\kappa_s$ for some $\kappa_s\in\frac12\Z_{\ge 0}$ (see Remark
  \ref{remark: apparent}). As discussed earlier, near the cusp $s$,
  there is a solution $y_+(z)$ of \eqref{equation: y''=-Qy sect 3} of
  the form
  $$
  y_+(z)=q_M^{\kappa_s}\left(1+\sum_{j\ge 1}c_jq_M^j\right).
  $$
  By the same computation as \eqref{equation: Q at infinity}, we see
  that $Q(s)=M^{-2}\kappa_s^2\ge 0$. This completes the proof of the
  theorem.
\end{proof}

\begin{Remark} \label{remark: same exponents}
  We remark that for a given point $z_0\in\H$, the local exponents of
  \eqref{equation: y''=-Qy sect 3} at $z_0$ and those at $\gamma z_0$
  are the same for any $\gamma\in\Gamma$. Also, if $z_0$ is
  apparent, then any equivalent point $\gamma z_0$,
  $\gamma\in\Gamma$, is also apparent. See \cite{Guo-Lin-Yang} for a
  proof of this fact.
\end{Remark}

For the special case $\Gamma$ is one of $\SL(2,\Z)$, $\Gamma_0^+(2)$,
and $\Gamma_0^+(3)$, we have the following informations about local
exponents at elliptic points. We first prove a lemma.

\begin{Lemma} \label{lemma: E2 nonvanishing}
  Let $\phi(z)$ be a quasimodular form of weight $2$ and depth $1$ on
  $\Gamma$. Then $\phi(z)$ does not vanish at any elliptic point of
  $\Gamma$.
\end{Lemma}

\begin{proof} Let $\wt\phi(z)$ be the nearly holomorphic modular form
  of weight $2$ corresponding to $\phi(z)$. We have
  $$
  \wt\phi(z)=\phi(z)+\frac{\alpha}{z-\overline z}
  $$
  for some nonzero complex number $\alpha$. Since $\wt\phi(z)$
  transforms like a modular form of weight $2$ on $\Gamma$, it
  vanishes at all elliptic points of $\Gamma$. It follows that
  $\phi(z)$ does not vanish at any elliptic point of $\Gamma$.
\end{proof}

\begin{Proposition} \label{proposition: local exponents at elliptic}
  We have the following properties about local exponents of
  \eqref{equation: y''=-Qy sect 3} at an elliptic point.
  \begin{enumerate}
  \item[(i)] Let $f$ be a quasimodular form of depth $1$ in
    $\wt\sM_k^{\le1}(\SL(2,\Z))$. Let $1/2\pm\kappa_\rho$ and
    $1/2\pm\kappa_i$,
    $\kappa_\rho,\kappa_i\in\frac12\N$, be the local exponents of
    \eqref{equation: y''=-Qy sect 3} at the elliptic points
    $\rho=(1+\sqrt{-3})/2$ and $i=\sqrt{-1}$. Then
    $(2\kappa_\rho,3)=(2\kappa_i,2)=1$.
  \item[(ii)] Let $f$ be a quasimodular form of depth $1$ in
    $\wt\sM_k^{\le1}(\Gamma_0^+(2),\pm)$. Let $1/2\pm\kappa_{\rho_j}$
    with $\kappa_{\rho_j}\in\frac12\N$, $j=1,2$, be the local
    exponents of \eqref{equation: y''=-Qy sect 3} at the elliptic
    points $\rho_1=i/\sqrt2$ and $\rho_2=(1+i)/2$ of
    $\Gamma_0^+(2)$. Then $(2\kappa_{\rho_1},2)=(2\kappa_{\rho_2},4)=1$.
  \item[(iii)] Let $f$ be a quasimodular form of depth $1$ in
    $\wt\sM_k^{\le1}(\Gamma_0^+(3),\chi^m)$ with
    $m\in\{0,1,2,3\}$ and $k\equiv m\mod 2$ that
    is not a modular form, where $\chi$ is the character of
    $\Gamma_0^+(3)$ defined by \eqref{equation: chi}. Let
    $1/2\pm\kappa_{\rho_j}$ with 
    $\kappa_{\rho_j}\in\frac12\N$, $j=1,2$, be the local exponents of
    \eqref{equation: y''=-Qy sect 3} at the elliptic points
    $\rho_1=i/\sqrt3$ and $\rho_2=(3+\sqrt{-3})/6$ of
    $\Gamma_0^+(3)$. Then $(2\kappa_{\rho_1},2)=(2\kappa_{\rho_2},6)=1$.
  \end{enumerate}
\end{Proposition}

\begin{proof} We first make a general remark that since the local
  exponents of \eqref{equation: y''=-Qy sect 3} at a point $z_0\in\H$
  are of the form $1/2\pm\kappa_{z_0}$ for some
  $\kappa_{z_0}\in\frac12\N$, if the order of $f(z)/\sqrt{W_f(z)}$ at
  $z_0$ is $-n_{z_0}$ for some $n_{z_0}\in\frac12\Z_{\ge 0}$, then we
  must have $1/2-\kappa_{z_0}=-n_{z_0}$, i.e.,
  $\kappa_{z_0}=1/2+n_{z_0}$.
  
  Assume that $f(z)\in\wt\sM_k^{\le 1}(\SL(2,\Z))$ that is
  not a modular form. Write it as $f(z)=f_1(z)E_2(z)+f_0(z)$ with
  $f_j\in\sM_{k-2j}(\SL(2,\Z))$. We first
  observe that if $f_0(z)$ and $f_1(z)$ have a common zero at $i$
  (respectively, $\rho$), then $f_j(z)/E_6(z)$ (respectively,
  $f_1(z)/E_4(z)$), $j=1,2$, are holomorphic modular forms, and
  $f(z)/E_6(z)$ (respectively, $f(z)/E_4(z)$) is a quasimodular
  form. Also, it is easy to check that $W_f(z)=E_6(z)^2W_{f/E_6}(z)$
  (respectively, $W_f(z)=E_4(z)^2W_{f/E_4}(z)$), so the differential
  equation \eqref{equation: y''=-Qy sect 3} from $f(z)$
  and that from $f(z)/E_6(z)$ (respectively, $f(z)/E_4(z)$) are the
  same. Therefore, without loss of generality, we may assume that
  $f_0(z)$ and $f_1(z)$ have no common zero at $i$ or $\rho$. We note
  that this assumption forces that the weight $k$ to be congruent to
  $0$ or $2$ modulo $6$ since any modular form on $\SL(2,\Z)$ whose
  weight is not divisible by $6$ must have a zero at $\rho$.

  By Lemma \ref{lemma: E2 nonvanishing}, $E_2(z)$ does not vanish at
  $\rho$ and $i$. Then using
  $$
  \left(f|_k\gamma\right)(z)=f(z)+\frac{12cf_1(z)}{2\pi i(cz+d)},
  \qquad\gamma=\M abcd\in\SL(2,\Z)
  $$
  and the assumption that $f_0(z)$ and $f_1(z)$ have no common zero at
  $i$ and $\rho$, we see that there exists an element $\gamma$ in
  $\SL(2,\Z)$ such that $f(\gamma\rho),f(\gamma i)\neq 0$. Since the
  local exponents of \eqref{equation: y''=-Qy sect 3} at $\rho$
  (respectively, $i$) and at $\gamma\rho$ (respectively, $\gamma i$)
  are the same, by the remark we made at the beginning of the proof,
  we have
  $$
  2\kappa_\rho=1+\ord_{\gamma\rho} W_f(z)
  =1+\ord_\rho W_f(z)
  $$
  (respectively, $2\kappa_i=1+\ord_iW_f(z)$).

  Now consider first the case where the weight $k$ is congruent to $0$
  modulo $6$. In this case, the weight of $W_f(z)$ is a multiple of
  $12$ and hence
  $$
  \frac{\ord_\rho W_f(z)}3+\frac{\ord_i W_f(z)}2
  $$
  must be an integer. It follows that $3|\ord_\rho W_f$ and
  $2|\ord_iW_f$ and therefore $(2\kappa_\rho,3)=(2\kappa_i,2)=1$.
  Similarly, when the weight $k$ is congruent to $2$ modulo $6$, we
  have
  $$
  \frac{\ord_\rho W_f(z)}3+\frac{\ord_iW_f(z)}2\equiv\frac13\mod 1.
  $$
  Then $2|\ord_iW_f$ and $\ord_\rho W_f\equiv 1\mod 3$, from which we
  conclude that $(2\kappa_\rho,3)=(2\kappa_i,2)=1$. This proves Part
  (i).

  The proof of Part (ii) is similar. We may assume without loss of
  generality that $f(z)$ does not vanish at $\gamma\rho_1$ and
  $\gamma\rho_2$ for some $\gamma\in\Gamma_0^+(2)$. Then
  $$
  2\kappa_{\rho_j}=1+\ord_{\rho_j}W_f(z), \qquad j=1,2,
  $$
  and it suffices to prove that $\ord_{\rho_j}W_f$ is always even.
  Indeed, if the weight $k$ is a multiple of $4$, then the weight of
  $W_f\in\sM_{2k}(\Gamma_0^+(2))$ is a multiple of $8$. The orders of
  such a modular form at any elliptic point must be even.
  If the weight $k$ is congruent to $2$ modulo $4$ so that the weight
  of $W_f$ is congruent to $4$ modulo $8$, then $W_f$ must be
  divisible by $M_4(z)=(4E_4(2z)+E_2(z))/5$, which has a zero of order
  $2$ at $\rho_2$ and is nonvashing at $\rho_1$. Since $W_f/M_4$ is
  again a modular form of a weight congruent to $0$ modulo $8$, we
  conclude that the order of $W_f$ at $\rho_j$ must be even for
  $j=1,2$. This proves Part (ii).

  The proof of (iii) is slightly more complicated. Write $f(z)$ as
  $f(z)=f_1(z)M_2^\ast(z)+f_0(z)$, where
  $M_2^\ast(z)=(3E_2(3z)+E_2(z))/4$ and
  $f_j\in\sM_{k-2j}(\Gamma_0^+(3),\chi^m)$. Again, without loss of
  generality, we assume that $f_0(z)$ and $f_1(z)$ have no common zero
  at $\rho_1$ and $\rho_2$. This assumption forces the weight $k$ to
  be congruent to $0$ or $2$ modulo $3$. Then, as above, we have
  $2\kappa_{\rho_j}=1+\ord_{\rho_j}W_f$.

  By Lemma \ref{lemma: Wronskian is modular}, we have
  $W_f\in\sM_{2k}(\Gamma_0^+(3),\chi^{2m})
  =\sM_{2k}(\Gamma_0^+(3),(-1)^k)$. Now
  recall that the ring of modular  
  forms $\oplus_{n=0}^\infty\sM_n(\Gamma_1(3))$, where $n$ runs
  through both even and odd integers, is freely generated by
  the modular forms $\sM_1(z)\in\sM_1(\Gamma_0^+(3),\chi)$ and
  $\sM_3(z)\in\sM_3(\Gamma_0^+(3),\chi)$ defined by \eqref{equation:
    M1 Gamma0+(3)} and \eqref{equation: M3 Gamma0+(3)}.
  Thus, $W_f$ is a linear combination of products of the form
  $M_1(z)^aM_3(z)^b$, where $(a,b)$ are pairs of nonnegative integers
  satisfying
  \begin{equation} \label{equation: ab}
  a+3b=2k, \qquad a+b\equiv 2k\mod 4.
  \end{equation}
  (The latter condition corresponds to the requirement that
  $M_1(z)^aM_3(z)^b$ must be a modular form on $\Gamma_0^+(3)$ with
  character $\chi^{2k}$.) The conditions imply that $b$
  must be even, while
  $$
  a\equiv\begin{cases}
    0\mod 6, &\text{if }k\equiv 0\mod 3, \\
    4\mod 6, &\text{if }k\equiv 2\mod 3, \end{cases}
  $$
  Since $M_1(z)$ and $M_3(z)$ have a simple zero at $\rho_2$ and
  $\rho_1$, respectively, and are nonvanishing elsewhere, the above
  properties of $(a,b)$ show that $\ord_{\rho_1}W_f\equiv 0\mod 2$ and
  $$
  \ord_{\rho_2}W_f\equiv\begin{cases}
    0\mod 6, &\text{if }k\equiv 0\mod 3, \\
    4\mod 6, &\text{if }k\equiv 2\mod 3. \end{cases}
  $$
  In either case, we find that
  $(2\kappa_{\rho_1},2)=(2\kappa_{\rho_2},6)=1$. The proof of the
  proposition is completed.
\end{proof}

We now give several examples of quasimodular forms
that are solutions of modular differential equations.
Consider first the case $\Gamma=\Gamma_0^+(2)$.
We recall that, by Corollary \ref{corollary:
    existence of extremal}, extremal quasimodular forms exist
  in $\wt\sM_k(\Gamma_0^+(2),\pm)$ and are unique up to scalars.

\begin{Theorem} \label{theorem: DE Gamma0+(2)}
  Set
  \begin{equation*}
    \begin{split}
      M_2(z)&=2E_2(2z)-E_2(z), \\
      M_4(z)&=\frac15(4E_4(2z)+E_4(z))=M_2(z)^2, \\
      M_8(z)&=\eta(z)^8\eta(2z)^8.
    \end{split}
  \end{equation*}
  Assume that $k$ is a nonnegative integer.
  \begin{enumerate}
  \item[(i)] Let $f(z)$ be an extremal quasimodular form in
  $\wt\sM_{4k}^{\le 1}(\Gamma_0^+(2),(-1)^k)$.
  Then $f(z)/M_8(z)^{k/2}$ is a solution of
  $$
  y''(z)=-4\pi^2\left(\frac k2\right)^2M_4(z)y(z).
  $$
  \item[(ii)] Let $f(z)$ be an extremal quasimodular form in
    $\wt\sM_{4k+2}^{\le1}(\Gamma_0^+(2),(-1)^k)$. Then
    $f(z)/M_8(z)^{k/2}M_2(z)$ is a solution of
    $$
    y''(z)=-4\pi^2\left(\left(\frac k2\right)^2M_4(z)
      -32\frac{M_8(z)}{M_4(z)}\right)y(z).
    $$
  \end{enumerate}
\end{Theorem}

\begin{proof}
  Let $W_f(z)$ be the Wronskian associated to $f$ defined
  in  Lemma \ref{lemma: Wronskian is modular}. By the lemma, $W_f(z)$
  is a modular form of weight $8k$ on $\Gamma_0^+(2)$ (with trivial
  character). We then observe that, by \eqref{equation: dim Gamma0(2)
    +} and \eqref{equation: dim Gamma0(2) -},
  $$
  \dim\wt\sM_{4k}^{\le 1}(\Gamma_0^+(2),(-1)^{k})=1+k
  =\dim\sM_{8k}(\Gamma_0^+(2)).
  $$
  It follows that, by Proposition \ref{proposition: Pellarin}, if $f$
  is an extremal quasimodular form in
  $\wt\sM_{4k}^{\le1}(\Gamma_0^+(2),(-1)^k)$, then 
  $v_\infty(f)=v_\infty(W_f)=k$. However, up to scalars, $M_8(z)^k$ is
  the only modular form in $\sM_{8k}(\Gamma_0^+(2))$ whose order
  of vanishing at $\infty$ is $k$ (see the proof of Corollary
  \ref{corollary: existence of extremal}). In
  other words, $W_f(z)=cM_8(z)^k$ for some nonzero complex number $c$
  and the functions $g_1(z)$ and $g_2(z)$ in \eqref{equation: g1, g2}
  are holomorphic throughout $\H$. It follows that the differential
  equation \eqref{equation: y''=-Qy sect 3} has no singularities in
  $\H$. Then the meromorphic modular form $Q(z)$ in Theorem
  \ref{theorem: Q is modular} must be a holomorphic modular form of
  $4$ and hence a multiple of $M_4(z)$ due to the fact that
  $\dim\sM_4(\Gamma_0^+(2))=1$. Since $g_2(z)=cq^{k/2}+\cdots$ for
  some nonzero constant $c$, we conclude that $Q(z)=(k/2)^2M_4(z)$.
  This proves Part (i).

  The proof of Part (ii) is slightly more involved. Let $f(z)$ be an
  extremal quasimodular form in $\wt\sM_{4k+2}(\Gamma_0^+(2),(-1)^k)$.
  By the same reasoning as in the proof of Part (i), using
  \eqref{equation: dim Gamma0(2) +}, \eqref{equation: dim Gamma0(2)
    -}, and Proposition \ref{proposition: Pellarin}, we see that
  $W_f(z)$ is a modular form of weight $8k+4$ such that
  $v_\infty(W_f)=k$ and hence must be a multiple of $M_4(z)M_8(z)^k$.
  Now $M_4(z)$ has a zero of order $2$ at $\rho_2=(1+i)/2$ (see
  Section \ref{section: Gamma0+(2)}), so the differential equation
  \eqref{equation: y''=-Qy sect 3} has potentially a singularity at
  $\rho_2$ and we need to determine the local exponents at the
  point.

  Write $f(z)$ as $f(z)=M_2^\ast(z)f_1(z)+f_0(z)$, where
  $M_2^\ast(z)=(2E_2(2z)+E_2(z))/3$ and
  $f_j\in\sM_{4k+2-2j}(\Gamma_0^+(2),(-1)^k)$, $j=0,1$. Since
  $\rho_2=(1+i)/2$ is an elliptic point of order $2$ of $\Gamma_0(2)$,
  we must have $f_0(\rho_2)=0$. We claim that $f_1(\rho_2)\neq 0$.
  Indeed, since $M_2(z)=2E_2(2z)-E_2(z)$ has a simple zero at points
  equivalent to $\rho_2$ and is nonvanishing elsewhere, if
  $f_1(\rho)=0$, then $f_0(z)/M_2(z)$ and $f_1(z)/M_2(z)$ are
  holomorphic modular forms on $\Gamma_0(2)$ with Atkin-Lehner
  eigenvalue $(-1)^{k-1}$. In other
  words, $f(z)/M_2(z)\in\wt\sM_{4k}^{\le 1}(\Gamma_0^+(2),(-1)^{k-1})$.
  However, this is impossible because there is no quasimodular form in
  $\wt\sM_{4k}^{\le 1}(\Gamma_0^+(2),(-1)^{k-1})$ whose order of
  vanishing at $\infty$ is $k$, by \eqref{equation: dim Gamma0(2) +},
  \eqref{equation: dim Gamma0(2) -}, and Corollary \ref{corollary:
    existence of extremal}. We conclude that $f_1(\rho_2)\neq 0$.

  We next observe that, by Lemma \ref{lemma: E2 nonvanishing},
  $M_2^\ast(z)$ does not vanish at any elliptic point of
  $\Gamma_0^+(2)$. Therefore, $f(z)=M^\ast_2(z)f_1(z)+f_0(z)$ does not
  vanish at $\rho_2$ and $f(z)/\sqrt{W_f(z)}$ has a simple zero at
  $\rho_2$. It follows that the local exponents of \eqref{equation:
    y''=-Qy sect 3} at $\rho$ must be $-1$ and $2$, i.e., the indicial
  equation at $\rho$ is $x^2-x-2=0$. By \eqref{equation: indicial
    Gamma0+(2)}, we must have
  $$
  Q(z)=\left(\frac k2\right)^2M_4(z)-32\frac{M_8(z)}{M_4(z)}.
  $$
  This proves Part (ii).
\end{proof}

By the same token, we have the following analogous result for
$\Gamma_0^+(3)$.

\begin{Theorem} \label{theorem: DE Gamma0+(3)}
  Set
  \begin{equation} \label{equation: Ms Gamma0+(3)}
    \begin{split}
      M_2(z)&=\frac{3E_2(3z)-E_2(z)}2, \\
      M_4(z)&=\frac{9E_4(3z)+E_4(z)}{10}=M_2(z)^2, \\
      M_6(z)&=\eta(z)^6\eta(3z)^6.
    \end{split}
  \end{equation}
  Let $\chi$ be the character of $\Gamma_0^+(3)$ defined by
  $$
  \chi\left(\frac1{\sqrt3}\M0{-1}30\right)
  =\chi\left(\frac1{\sqrt3}\M3{-1}30\right)=-i.
  $$
  Let $k$ be a nonnegative integer.
  \begin{enumerate}
    \item[(i)] Let $f(z)$ be an extremal quasimodular form in
    $\wt\sM_{3k}^{\le1}(\Gamma_0^+(3),\chi^k)$. Then
    $f(z)/M_6(z)^{k/2}$ is a solution of the differential equation
    $$
    y''(z)=-4\pi^2\left(\frac k2\right)^2M_4(z)y(z).
    $$
  \item[(ii)] Let $f(z)$ be an extremal quasimodular form in
    $\wt\sM_{3k+2}^{\le 1}(\Gamma_0^+(3),\chi^k)$. Then
    $f(z)/M_6(z)^{k/2}M_2(z)$ is a solution of the differential
    equation
    $$
    y''(z)=-4\pi^2\left(\left(\frac k2\right)^2M_4(z)
      -18\frac{M_6(z)}{M_2(z)}\right)y(z).
    $$
  \end{enumerate}
\end{Theorem}

\begin{proof} The proof is similar to that of Theorem \ref{theorem: DE
    Gamma0+(2)}. We omit most of the details and only mention that the
  Wronskian $W_f$ associated to $f$ in Part (ii) is a scalar multiple
  of $M_6(z)^kM_4(z)$, where $M_4(z)$ has a zero of order $4$ at
  $\rho_2=(3+\sqrt{-3})/6$. Hence, $f(z)/\sqrt{W_f(z)}$ has a pole of
  order $2$ at $\rho_2$, and the indicial equation of \eqref{equation:
    y''=-Qy sect 3} at $\rho_2$ is $x^2-x-6=0$. By \eqref{equation:
    indicial Gamma0+(3)}, we therefore have
  $$
  Q(z)=\left(\frac k2\right)^2M_4(z)-18\frac{M_6(z)}{M_2(z)},
  $$
  as claimed.
\end{proof}

\begin{Remark} The proof of the two theorems also applies to
  quasimodular forms on $\SL(2,\Z)$. We can similarly show that
  if $f$ is an extremal quasimodular form of weight $6k$, then
  $f/\Delta^{k/2}$ is a solution of $y''(z)=-\pi^2k^2E_4(z)y(z)$. This
  gives a new proof of \eqref{equation: example of KK}. Also, if $f$
  is an extremal quasimodular form of weight $6k+2$ and depth $1$ on
  $\SL(2,\Z)$, then the Wronskian $W_f(z)$ is a multiple of
  $\Delta(z)^kE_4(z)$. Therefore, the order of $f(z)/\sqrt{W_f(z)}$
  at $\rho=(1+\sqrt{-3})/2$ is $-1/2$ and the indicial equation at
  $\rho$ is $x^2-x-3/4=0$. By \eqref{equation: indicial
    SL(2,Z)}, $f(z)/\sqrt{W_f(z)}$ is a solution of
  $$
  y''(z)=-4\pi^2\left(\left(\frac k2\right)^2E_4(z)-144
    \frac{\Delta(z)}{E_4(z)^2}\right)y(z).
  $$
  This corresponds to Theorem 2.1(2) of \cite{Kaneko-Koike}.

  We remark that the reason why we do not have a corresponding
  statement for extremal quasimodular forms of weight $6k+4$ and depth
  $1$ on $\SL(2,\Z)$ is that such a quasimodular form $f(z)$ must be a
  multiple of $E_4(z)$, but the differential equation satisfied by
  $f(z)/\sqrt{W_f(z)}$ is the same as that for
  $f(z)/E_4(z)\sqrt{W_{f/E_4}(z)}$.
\end{Remark}


\section{MODE in the case of $\SL(2,\Z)$}
\label{section: SL(2,Z)}
In this section, we shall consider Fuchsian differential
  equations of the form
  \begin{equation} \label{equation: DE SL(2,Z)}
  y''(z)=-4\pi^2Q(z)y(z)
  \end{equation}
on $\H$ for some meromorphic modular forms $Q(z)$ of weight $4$ on
$\SL(2,\Z)$ satisfying the conditions in \eqref{equation: H}. In
particular, $Q(z)$ is holomorphic at $q=0$, $q=e^{2\pi iz}$, with
$Q(\infty)\ge 0$. Let $D_q=qd/dq=(2\pi i)^{-1}d/dz$. At $q=0$,
\eqref{equation: DE SL(2,Z)} becomes
\begin{equation} \label{equation: general DE in Dq}
  D_q^2y=Q(z)y.
\end{equation}
Thus, \eqref{equation: DE SL(2,Z)} is also Fuchsian at the cusp
$\infty$ and its local exponents at $\infty$ are $\pm\kappa_\infty$,
where
\begin{equation} \label{equation: kappa at infinity}
  \kappa_\infty=\sqrt{Q(\infty)}.
\end{equation}
Note that in \cite[Proposition 3.6]{Guo-Lin-Yang} we have
shown that if \eqref{equation: DE SL(2,Z)} is apparent at a point
$z_0$, then it is apparent at all $\gamma z_0$, $\gamma\in\SL(2,\Z)$.
Let $z_1,\ldots,z_m$ be the poles of $Q(z)$ such that $z_j$ and $z_k$
are not equivalent under $\SL(2,\Z)$ when $j\neq k$, i.e.,
$z_k\notin\SL(2,\Z)z_j$, and
$z_j\notin\{\rho=(1+\sqrt{-3})/2,i=\sqrt{-1}\}$. The apparentness
condition implies that the local exponents at $\rho$, $i$, and $z_j$,
$j=1,\ldots,m$, are $1/2\pm\kappa_\rho$, $1/2\pm\kappa_i$, and
$1/2\pm\kappa_j$, $1\leq j\leq m$, where all $\kappa$'s are in
$\frac{1}{2}\N$.

To the modular differential equation \eqref{equation: DE SL(2,Z)},
there associates the Bol representation $\rho$, a homomorphism from
$\SL(2,\Z)$ to $\PSL(2,\C)$, that is, by choosing a fundamental
solution $Y(z)=(y_1(z),y_2(z))^t$, there is $\rho(\gamma)\in\SL(2,\C)$
for any $\gamma\in\SL(2,\C)$ such that
\begin{equation} \label{equation: Bol section 4}
  \column{y_1\big|_{-1}\gamma}{y_2\big|_{-1}\gamma}
  =\pm\rho(\gamma)\column{y_1}{y_2}.
\end{equation}
For our purpose, we would like to lift $\rho$ to a homomorphism
$\hat\rho:\SL(2,\Z)\to\SL(2,\C)$ such that $\rho=\pi\circ\hat\rho$,
where $\pi$ is the natural projection from $\SL(2,\C)$ to
$\PSL(2,\C)$. This will be done as follows.

Define $t_j=E_6(z_j)^2/E_4(z_j)^3$ such that
$F_j(z_j)=0$, where
\begin{equation}\label{equation: F_j}
F_j(z)=E_6(z)^2-t_jE_4(z)^3.
\end{equation}

\begin{Lemma}\label{lemma: F_j has simple zero} If $t_j\neq 0,1$, then
$E_6(z)^2-t_jE_4(z)^3$ has zeros only at $z_j$, and $z_j$ is a simple zero.
\end{Lemma}
The lemma can be proved by the theorem of counting zeros of modular
forms (\cite[p. 85, Theorem 3]{Serre}).

Set
\begin{equation} \label{equation: F SL(2,Z)}
F(z)=\Delta(z)^{\kappa_\infty}E_4(z)^{\kappa_\rho-1/2}
E_6(z)^{\kappa_i-1/2}\prod_{j=1}^m
F_j(z)^{\kappa_j-1/2}
\end{equation}
and
\begin{equation} \label{equation: ell}
  \ell=-1+12\kappa_\infty+4\left(\kappa_\rho-\frac12\right)
  +6\left(\kappa_i-\frac12\right)
  +12\sum^m_{j=1}\left(\kappa_j-\frac12\right),
\end{equation}
so that $F(z)^2$ is a modular form of weight $2(\ell+1)$ on $\SL(2,\Z)$. Note that by assumption \eqref{equation: H}, $\kappa_i-1/2$ is an integer. Thus, $\ell$ is always an odd integer.

For any function $f(z)$ on $\H$, we set
$$
\hat f(z)=F(z)f(z).
$$
It is easy to see that if $y(z)$ is a solution of \eqref{equation: DE
  SL(2,Z)}, then $\hat{y}(z)$ is single-valued function that is
holomorphic throughout $\H$. Furthermore, \emph{its order at $z_j$ is
  either $0$ or $2\kappa_j$}, and a similar property holds at
$\infty$, $\rho$,
and $i$. Now choose a fundamental solution $Y(z)=(y_1(z),y_2(z))^t$ of
\eqref{equation: DE SL(2,Z)}. Then we have a homomorphism
$\hat\rho:\SL(2,\Z)\to\GL(2,\C)$ defined by
\begin{equation} \label{equation: Y weight ell}
\begin{split}
  \left(\hat{Y}\big\vert_{\ell}\gamma\right)(z)
  =\hat\rho(\gamma)\hat{Y}(z),\quad \text{where }
  \hat Y(z)=F(z)Y(z).
\end{split}
\end{equation}
The next lemma shows that this homomorphism $\hat\rho$ has an image in
$\SL(2,\C)$.

\begin{Lemma} \label{lemma: det rho SL(2,Z)}
  We have $\det\hat\rho(\gamma)=1$ for all $\gamma\in\SL(2,\Z)$. That
  is, $\hat\rho$ is a homomorphism from $\SL(2,\Z)$ to $\SL(2,\C)$.
\end{Lemma}

\begin{proof} Let
  $$
  W(z)=\det\M{y_1}{y_1'}{y_2}{y_2'}, \qquad
  \hat W(z)=\det\M{\hat y_1}{\hat y_1'}{\hat y_2}{\hat y_2'}.
  $$
  On the one hand, \eqref{equation: Y weight ell} yields that for
  $\gamma\in\SL(2,\Z)$,
  $$
  \hat W(z)\big|_{2(\ell+1)}\gamma
  =\det\hat\rho(\gamma)\hat W(z).
  $$
  On the other hand, we have $\hat W(z)=F(z)^2W(z)$. Hence,
  \begin{equation*}
    \begin{split}
      \hat W(z)\big|_{2(\ell+1)}
     &=\left(F(z)^2\big|_{2(\ell+1)}\gamma\right)
       \left(W(z)\big|_0\gamma\right) \\
     &=\left(F(z)^2\big|_{2(\ell+1)}\gamma\right)
      \det\rho(\gamma)W(z)
     =\frac{\left(F(z)^2\big|_{2(\ell+1)}\gamma\right)}{F(z)^2}
     \hat W(z).
    \end{split}
  \end{equation*}
  Comparing the two expressions, we find that
  $$
  \det\hat\rho(\gamma)
  =\frac{\left(F(z)^2\big|_{2(\ell+1)}\gamma\right)}{F(z)^2}=1
  $$
  for all $\gamma\in\SL(2,\Z)$ since $F(z)^2$ is a modular form of
  weight $2(\ell+1)$ on $\SL(2,\Z)$. This proves the lemma.
\end{proof}

If no confusion arises, we will also call $\hat\rho$ the Bol  representation.

Throughout the remainder of the section, we will let $y_+(z)$ denote the unique solution of \eqref{equation: DE SL(2,Z)} of the form
\begin{equation} \label{equation: y+}
  y_+(z)=q^{\kappa_\infty}\left(1+\sum_{j\ge 1}c_jq^j\right).
\end{equation}
Also, we will use the standard notations $T=\SM1101$,
$S=\SM0{-1}10$ and $R=TS=\SM1{-1}10$. They satisfy
\begin{equation}\label{equation: S^2 and R^3}
S^2=R^3=-I.
\end{equation}
Out main result of this section is the following theorem
(stated as Theorem \ref{theorem: main 1} in Section \ref{section:
  introduction}).

\begin{Theorem}\label{theorem: infinity is not apparent implies}
Suppose that $Q(z)$ satisfies \eqref{equation: H} with
$\kappa_\infty\in\frac12\Z_{\ge0}$. Then the following
statements hold true.
\begin{enumerate}
  \item[(i)] The differential equation \eqref{equation: DE SL(2,Z)} is
    not apparent at $\infty$.
  \item[(ii)] Let $y_2(z)=y_+(z)$, where $y_+(z)$ is the solution of
  \eqref{equation: DE SL(2,Z)} of the form
  \eqref{equation: y+}. Let $\hat y_1(z)=\left(\hat y_+|_\ell
    S\right)(z)$.  Then $\hat y_1(z)=z\hat y_2(z)+\hat m_1(z)$ for
  some modular form $\hat{m}_1(z)$ of weight $\ell-1$ on $\SL(2,\Z)$.
  \item[(iii)] Using $y_1(z)$ and $y_2(z)$ as the basis, Bol's
    representation satisfies
    $$
    \hat\rho(\gamma)=\gamma,\quad \gamma\in\SL(2,\Z).
    $$
    In particular, the ratio $h(z)=y_1(z)/y_2(z)$ is equivariant,
    i.e., $h(\gamma z)=\gamma\cdot h(z)$ for all $\gamma\in\SL(2,\Z)$.
  \item[(iv)] Write $\hat y_+(z)$ as $\hat y_+(z)=\frac{\pi
      i}{6}\hat m_1(z)E_2(z)+\hat m_2(z)$. Then $\hat{m}_2(z)$ is a
    modular form of weight $\ell+1$ with respect to $\SL(2,\Z)$. In
    other words, $\hat y_+(z)$ is a quasimodular form of weight
    $\ell+1$ and depth $1$ on $\SL(2,\Z)$.
\end{enumerate}
\end{Theorem}


Theorem \ref{theorem: infinity is not apparent implies}
will be proved through a series of lemmas.

\begin{Lemma} \label{lemma: not apparent SL(2,Z)}
  The ODE \eqref{equation: DE SL(2,Z)} is not apparent at $\infty$.
\end{Lemma}

\begin{proof} Assume that \eqref{equation: DE SL(2,Z)} is apparent at
  $\infty$. Then for any fundamental solutions near $\infty$, we
  have $\hat\rho(T)=I$. Now since $\ell$ is odd, we have
  $\hat\rho(-I)=(-1)^\ell I=-I$. Thus, by \eqref{equation: S^2 and
    R^3}, $\hat\rho(S)^2=\hat\rho(R)^3=-I$. On the other hand, we have
  $\hat\rho(R)^2=(\hat\rho(T)\hat\rho(S))^2=-I$, which is absurd since
  $\hat\rho(R)^2$ and $\hat\rho(R)^3$ cannot be both equal to $-I$.
  This proves that \eqref{equation: DE SL(2,Z)} cannot be apparent at
  $\infty$.
\end{proof}

\begin{Remark} \label{remark: infinity is apparent}
  Observe that the parity of $\ell$ depends only on
  $\kappa_i$. If $\kappa_i\in\N$, i.e., if $\ell$ is even, then
  $\hat\rho(R^3)=\hat\rho(S^2)=\hat\rho(-I)=I$, which implies that
  $\hat\rho(S)=\pm I$ and either $\hat\rho(R)=I$ or
  $\hat\rho(R)^2+\hat\rho(R)+I=0$. In any case, there is a matrix $P$
  such that both $P\hat\rho(S)P^{-1}$ and $P\hat\rho(R)P^{-1}$ are
  diagonal, which implies that $P\hat\rho(T)P^{-1}$ is also
  diagonal, and the apparentness at $\infty$ follows. Thus, the
  condition $\kappa_i\notin\N$ is necessary in order for
  \eqref{equation: DE SL(2,Z)} to be not apparent at $\infty$.
\end{Remark}

\begin{Lemma}\label{lemma: y1,y2}
Let $y_2(z)=y_+(z)$ and $y_1(z)=\left(\hat y_+\big\vert_{\ell}
  S\right)(z)/F(z)$.
Then the following hold:
\begin{enumerate}
\item[(i)]
$y_i$, $i=1,2$, are linearly independent, and
\item[(ii)]
the ratio $h(z)=y_1(z)/y_2(z)$ is equivariant.
\end{enumerate}
\end{Lemma}

\begin{proof} By Lemma \ref{lemma: not apparent SL(2,Z)},
  \eqref{equation: DE SL(2,Z)} is not apparent at
  $\infty$. Hence, we have a solution $y_-(z)$ of the form
  $y_-(z)=dzy_+(z)+m(z)$ with $m(z)=q^{-\kappa_\infty}(1+\cdots)$ and
  $d\neq 0$. Then with respect to the choice $\hat Y(z)=(\hat
  y_+(z),\hat y_-(z))^t$, we have $\hat\rho(T)=\SM10d1$.
  
  Now suppose that $y_+(z)$ and $\left(\hat
    y_+\big|_{\ell}S\right)(z)/F(z)$ are not
  linearly independent, i.e., $\left(\hat{y}_2\big\vert_\ell
    S\right)(z)=\lambda\hat{y}_2(z)$ for some $\lambda$. Then
  $\hat\rho(S)=\SM \lambda0c{\overline{\lambda}}$, $c\in\C$. Because
  $\ell$ is odd, we have $\hat\rho(S^2)=\hat\rho(-I)=(-1)^\ell
  I=-I$. It follows that $\lambda=\pm i$ and
$$
\hat\rho(R)=\M10{d}1\M\lambda0c{\bar{\lambda}}
=\M{\lambda}0{d\lambda+c}{\bar{\lambda}},
$$
which implies $\lambda^3=-1$, a contradiction because $\lambda=\pm
i$. Thus ${y}_i$, $i=1,2$, are linearly independent.

Since \eqref{equation: DE SL(2,Z)} is not apparent at $\infty$ and we
have shown that $y_2(z)=y_+(z)$ and 
$y_1(z)=\left(\hat y_2\big|_{\ell}S\right)(z)/F(z)$ are linearly
independent, by \eqref{equation: second solution}, there is $d\neq 0$
such that
$y_1(z)=dzy_2(z)+m_1(z)$ for some function
$m_1(z)=q^{-\kappa_\infty}(a_0+\sum_{j\ge1}a_jq^j)$ with $a_0\neq
0$. We claim that
\begin{equation}\label{equation: d=1}
d=1.
\end{equation}
Obviously, with respect to the basis $(\hat y_1(z),\hat y_2(z))^t$,
$\hat\rho(T)=\SM1 {d}01$. On the other hand, since
$\hat{y}_1(z)=\left(\hat{y}_2\big\vert_\ell S\right)(z)$, and
$-\hat{y}_2(z)=\left(\hat{y}_2\big\vert_\ell
  S^2\right)(z)=\left(\hat{y}_1\big\vert_\ell S\right)(z)$, we have
\begin{equation} \label{equation: hat rho(S) SL(2,Z)}
\hat\rho(S)=\M0{-1}10
\end{equation}
and
$$
\hat\rho(R)=\M1{d}01\M0{-1}10=\M{d}{-1}10.
$$
By $\hat\rho(R^3)=-I$, i.e. $\hat\rho(R)^2=-\hat\rho(R)^{-1}$, we have
$$
\M {d}{-1}10\M {d}{-1}10=\M{d^2-1}{- d}{ d}{-1}
=-\M 0{1}{-1}{d},
$$
which yields $d=1$. This proves \eqref{equation: d=1}.

Using $h(z)=\hat{y}_1(z)/\hat{y}_2(z)$, we have
$h(z+1)=(\hat{y}_1(z)+\hat{y}_2(z))/\hat{y}_2(z)=h(z)+1$ and
$$
h(-1/z)=\frac{\hat{y}_1(-1/z)}{\hat{y}_2(-1/z)}=\frac{\left(\hat{y}_1\big\vert_\ell
    S\right)(z)}{\left(\hat{y}_2\big\vert_\ell
    S\right)(z)}=\frac{-\hat{y}_2(z)}{\hat{y}_1(z)}=-1/h(z).
$$
This proves that $h(z)$ is equivariant.
\end{proof}

We recall the identity $\hat{y}_1(z)=\left(\hat{y}_2\big\vert_\ell
  S\right)(z)=z\hat{y}_2(z)+\hat{m}_1(z)$
obtained in the proof of Lemma \ref{lemma: y1,y2}.

\begin{Lemma}\label{lemma: hat m1 is a modular function}
We have
\begin{enumerate}
\item[(i)] $\hat m_1(z+1)=\hat m_1(z)$, and
\item[(ii)] $\left(\hat m_1\big\vert_{\ell
        -1}S\right)(z)=\hat m_1(z)$.
\end{enumerate}
Hence, $\hat m_1(z)$ is a modular form of weight $\ell-1$ on $\SL(2,\Z)$.  
\end{Lemma}

\begin{proof}
The identity (i) is obvious.
By the definition, $h(z)$ of Lemma \ref{lemma: y1,y2}(ii)
can be written as
\begin{equation}
  h(z)=z+\frac{m_1(z)}{y_2(z)}=z
  +\frac{\hat m_1(z)}{\hat{y}_2(z)}.
\end{equation}
We have
\begin{equation}\label{equation: -1/h}
  -\frac{1}{h(z)}=h\left(-\frac{1}{z}\right)=-\frac1{z}
  +\frac{\hat m_1(-1/z)}{\hat y_2(-1/z)}.
\end{equation}
Thus by \eqref{equation: -1/h},
\begin{align*}
  \frac{\hat m_1(-1/z)}{\hat y_1(z)}
  &=\frac{z^\ell\hat m_1(-1/z)}{\hat y_2(-1/z)}
    =z^\ell\left(\frac1{z}-\frac{1}{h(z)}\right)
    =\frac{(h(z)-z)z^\ell}{zh(z)}\\
  &=\frac1{z^{1-\ell}h(z)}\frac{\hat m_1(z)}{\hat y_2(z)}
    =\frac{\hat m_1(z)}{z^{1-\ell}\hat y_1(z)},
\end{align*}
which implies
$\hat m_1(-1/z)=z^{\ell-1}\hat m_1(z)$. This proves (ii). 
\end{proof}

\begin{Lemma}\label{lemma: hat(m2) is modular form}
Write $y_+(z)$ as $ y_+(z)=\frac{\pi i}{6}m_1(z)E_2(z)+m_2(z)$. Then
$\hat m_2(z)$ is a modular form of weight $\ell+1$ with
respect to $\SL(2,\Z)$.
\end{Lemma}

\begin{proof} By the previous lemma, we have
\begin{align*}
\hat{y}_1(z)&=\left(\hat y_+\big\vert_{\ell}S\right)(z)\\
&=z^{-\ell}\left(\frac{\pi i}{6}\hat m_1(-1/z)
     E_2(-1/z)+\hat m_2(-1/z)\right)\\
&=z^{-\ell}\left(\frac{\pi i}{6}z^{\ell-1}\hat m_1\left(z\right)\left(z^2E_2(z)+\frac{6}{\pi i}z\right)+\hat m_2(-1/z)\right)\\
&=\frac{\pi i}{6}\hat m_1\left(z\right)\left(zE_2(z)+\frac{6}{\pi i}\right)+z^{-\ell}\hat m_2(-1/z)\\
&=\left(z\hat y_+(z)-z\hat m_2(z)+\hat
                                                                                                        m_1(z)\right)+z^{-\ell}m_2(-1/z).
\end{align*}
On the other hand, recall from the proof of Lemma \ref{lemma: y1,y2}
that the left-hand side is equal to $z\hat y_+(z)+\hat
m_1(z)$. Therefore, $\hat m_2(-1/z)=z^{\ell+1}\hat m_2(z)$. That is,
$\left(\hat{m}_2\big\vert_{\ell+1}S\right)(z)=\hat{m}_2(z)$.

Finally, since $\hat y_+(z+1)=\hat y_+(z)$ and
$\hat m_1(z+1)=\hat m_1(z)$, we have $\hat
m_2(z+1)=\hat m_2(z)$. We conclude that $\hat m_2(z)$ is a modular
form of weight $\ell+1$ on $\SL(2,\Z)$.
\end{proof}

Having proved Theorem \ref{theorem: infinity is not
  apparent implies}, we now give some examples.

\begin{Example} Consider the differential equation
  \begin{equation} \label{equation: KK}
  y''(z)=-4\pi^2\kappa_\infty^2E_4(z)y(z), \qquad\kappa_\infty\in\frac12\N.
  \end{equation}
  Since $E_4(z)$ is a holomorphic modular form, this differential
  equation has no singularities in $\H$. Hence, we have
  $\kappa_\rho=\kappa_i=1/2$, $F(z)=\Delta(z)^{\kappa_\infty}$ and the
  integer $\ell$ in \eqref{equation: ell} is $-1+12\kappa_\infty$.
  Also, the local exponents at $\infty$ are clearly
  $\pm\kappa_\infty$. Thus, Theorem \ref{theorem: infinity is not 
    apparent implies} predicts that
  $\Delta(z)^{\kappa_\infty}y_+(z)$ is a quasimodular form of
  weight $\ell+1$ and depth $1$ on $\SL(2,\Z)$.
  Indeed, as mentioned in Section \ref{subsection: definition of
    quasi}, by the works of Kaneko and Koike \cite{Kaneko-Koike-2003,
    Kaneko-Koike} and Pellarin \cite[Proposition 3.1]{Pellarin}, the
  differential equation \eqref{equation: KK} has
  $y(z)=f(z)/\Delta(z)^{\kappa_\infty}$ as a solution, where $f(z)$ is
  an extremal quasimodular form of weight $12\kappa_\infty=\ell+1$ and
  depth $1$, agreeing with our Theorem \ref{theorem: infinity is not
    apparent implies}.
\end{Example}

\begin{Example} Consider the differential equation
  \begin{equation} \label{equation: DE SL(2,Z) 2}
  y''(z)=-4\pi^2\left(\kappa_\infty^2E_4(z)
    -144\frac{\Delta(z)}{E_4(z)^2}\right)y(z), \quad
  \kappa_\infty\in\frac12\Z_{\ge0}.
  \end{equation}
  The differential equation has only singularities at points
  equivalent to $\rho$ under $\SL(2,\Z)$. By \eqref{equation:
    indicial SL(2,Z)} (or \cite[Theorem 1.7 and
  Corollary 3.10]{Guo-Lin-Yang}), the indicial equation as $\rho$ is
  $x^2-x-3/4=0$. Hence the local exponents at $\rho$ are
  $-1/2$ and $3/2$, i.e., $\kappa_\rho=1$, and the differential
  equation is apparent on $\H$, by Theorem \ref{corollary: apparent
    at elliptic}. The integer $\ell$ in \eqref{equation:
    ell} is $1+12\kappa_\infty$. Theorem \ref{theorem: infinity is not
    apparent implies} then asserts that if we let
  $y_+(z)=q^{\kappa_\infty}+\cdots$ be the 
  solution with exponent $\kappa_\infty$ at $\infty$, then
  $\Delta(z)^{\kappa_\infty}(z)E_4(z)^{1/2}y_+(z)$ is a quasimodular
  form of weight $\ell+1$ and depth $1$ on $\SL(2,\Z)$.
  To check that this is true, we recall that Kaneko and Koike
  \cite[Theorem 2.1]{Kaneko-Koike} has also proved that if $k\equiv
  2\mod 6$, then an extremal quasimodular form $f(z)$ of weight $k$
  and depth $1$ on $\SL(2,\Z)$ satisfies
  $$
  D_q^2f-\left(\frac k6-\frac13\frac{E_6}{E_4}\right)D_qf
  +\left(\frac{k(k-1)}{12}D_qE_2-\frac{k-1}{18}\frac{D_qE_6}{E_4}\right)f=0.
  $$
  A straightforward calculation shows that this is equivalent to the
  assertion that $f(z)/\Delta(z)^{\kappa_\infty}E_4(z)^{1/2}$ is a
  solution of \eqref{equation: DE SL(2,Z) 2} with
  $k=12\kappa_\infty+2=\ell+1$. This agrees with our result.
\end{Example}

\begin{Example} Consider the differential equation
  $$
  y''(z)=-4\pi^2\left(\frac14E_4(z)
    +864\frac{E_4(z)\Delta(z)}{E_6(z)^2}\right)y(z).
  $$
  We have $\kappa_\infty=1/2$ and by \eqref{equation: indicial
    SL(2,Z)}, $\kappa_i=3/2$. The integer $\ell$ in this case is
  $11$ and Theorem \ref{theorem: infinity is not apparent implies}
  implies that $\Delta(z)^{1/2}E_6(z)y_+(z)$ is a quasimodular form of
  weight $12$ and depth $1$ on $\SL(2,\Z)$. Indeed, we compute that
  $$
  y_+(z)=q^{1/2}\left(1+462q+247494q^2+132490928q^3
    +\cdots\right)
  $$
  and find that
  $$
  11088\Delta(z)^{1/2}E_6(z)y_+(z)=E_2(z)E_4(z)E_6(z)+6E_4(z)^3-
  7E_6(z)^2.
  $$
\end{Example}
\section{MODE in the case of $\Gamma_0^+(2)$}
\label{section: Gamma0+(2)}
In this section, we will obtain results analogous to Theorem
\ref{theorem: infinity is not apparent implies} for the group
$\Gamma^+_0(2)$.

First of all, we recall that the graded ring of holomorphic modular
forms with respect to $\Gamma^+_0(2)$ is generated by three modular
forms $M_4(z)$, $M_6(z)$ and $M_8(z)$, whose weights are $4$, $6$ and
$8$, respectively. These modular forms can be written explicitly in
terms of the Eisenstein series on $\SL(2,\Z)$ and the Dedekind eta
function:
\begin{equation}\label{equation: Mi on Gamm^+a0(2)}
\begin{split}
  M_4(z)&=(4E_4(2z)+E_4(z))/5,\\
  M_6(z)&=(8E_6(2z)+E_6(z))/9, \\
  M_8(z)&=\eta(z)^8\eta(2z)^8.
\end{split}
\end{equation}
Noticing that $\dim\sM_{12}(\Gamma_0^+(2))=2$, there must be a
relation among $M_4(z)^3$, $M_6(z)^2$, and $M_4(z)M_8(z)$. We find
that it is $M_6(z)^2=M_4(z)(M_4(z)^2-256M_8(z))$. For our purpose, we also need
$$
M_2(z)=2E_2(2z)-E_2(z),
$$
which is a modular form in $\sM_2(\Gamma_0^+(2),-)$ and satisfies
$M_2(z)^2=M_4(z)$. Define
$$
M^*_2(z)=\frac{1}{2\pi i}\frac{M'_8(z)}{M_8(z)}
=\frac{2E_2(2z)+E_2(z)}3.
$$
The holomorphic function $M^*_2(z)$ is not a modular form. Instead, it
satisfies
\begin{equation} \label{equation: M2* Gamma0+(2)}
  M_2^*(\gamma z)=(cz+d)^2M_2^*(z)+\frac{4}{\pi i}c(cz+d),
  \quad\gamma=\M abcd\in\Gamma^+_0(2).
\end{equation}

We use the notations $S=\frac{1}{\sqrt{2}}\SM 0{-1}20$,
$T=\SM1{1}01$, $R=TS=\frac{1}{\sqrt{2}}\SM2{-1}20$ such that
$S^2=R^4=-I$. We collect some facts about $\Gamma^+_0(2)$. 
\begin{enumerate}
\item[(i)]
There are only two elliptic points $\rho_1=i/\sqrt{2}$ and
$\rho_2=(1+i)/2$ of order $2$ and $4$, respectively. Their stabilizer
subgroups are generated by $S$ and $R$, respectively. Also, $\infty$
is the only cusp and has width $1$.
\item[(ii)] The modular form $M_2(z)$ has only one zero at
  $\rho_2$. The zero is simple. The modular form $M_4(z)$ has only one
  zero (a double zero) at $\rho_2$ and $M_6(z)$ has only zeros at
  $\rho_1$ and $\rho_2$. Both zeros are simple.
\item[(iii)] $M_8(z)\neq 0$, $\forall z\in\H$, and $\infty$ is the
  simple zero of $M_8(z)$.
\end{enumerate}
(Here when we say ``a modular form $f(z)$ has only a zero at a point
$z_0$'', it should be understood as ``$f(z)$ has only zeros at points
that are $\Gamma_0^+(2)$-equivalent to $z_0$''.)

Consider the differential equation
\begin{equation} \label{equation: DE Gamma0+(2)}
  y''(z)=-4\pi^2Q(z)y(z),
\end{equation}
where $Q(z)$ is a meromorphic modular form of weight $4$ on
$\Gamma_0^+(2)$ such that the conditions \eqref{equation: H} are
satisfied. Let $z_1,\ldots,z_m$ be $\Gamma_0^+(2)$-inequivalent poles
of $Q(z)$ other than $\rho_1$ and $\rho_2$, and assume that the local
exponents of \eqref{equation: DE Gamma0+(2)} at $\infty$, $\rho_1$,
$\rho_2$, and $z_1,\ldots,z_m$ are $\pm\kappa_\infty$,
$1/2\pm\kappa_{\rho_1}$, $1/2\pm\kappa_{\rho_2}$, and
$1/2\pm\kappa_1,\ldots,1/2\pm\kappa_m$, respectively, where the
$\kappa$'s satisfy \eqref{equation: H}.

For $j=1,\ldots,m$, set $t_j=M_4(z_j)^2/M_8(z_j)$ and
$F_j(z)=M_4(z)^2-t_jM_8(z)$. This modular form $F_j$ has only one
simple zero at $z_j$. Let
\begin{equation} \label{equation: F Gamma0+(2)}
F(z)=M_8(z)^{\kappa_\infty}\left(\frac{M_6(z)}{M_2(z)}
\right)^{\kappa_{\rho_1}-1/2}
M_2(z)^{\kappa_{\rho_2}-1/2}\prod_{j=1}^mF_j(z)^{\kappa_j-1/2},
\end{equation}
and
\begin{equation} \label{equation: ell Gamma0+(2)}
  \ell=-1+8\kappa_\infty+4\left(\kappa_{\rho_1}-\frac12\right)
  +2\left(\kappa_{\rho_2}-\frac12\right)
+8\sum_{j=1}^m\left(\kappa_j-\frac12\right),
\end{equation}
so that $F(z)^2$ is a modular form of weight $2(\ell+1)$ on
$\Gamma_0^+(2)$. We note that by the condition \eqref{equation: H},
the integer $\ell$ is always odd.

For any function $f(z)$ on $\H$, let $\hat f(z)$ denote the
function $F(z)f(z)$. When $f(z)=y(z)$ is a solution of
\eqref{equation: DE Gamma0+(2)}, by construction, $\hat y(z)$ is a
single-valued function holomorphic throughout $\H$ and its order at
$z_j$ is either $0$ or $2\kappa_j$. If a fundamental solution
$Y(z)=(y_1(z),y_2(z))^t$ of \eqref{equation: DE Gamma0+(2)} is chosen,
then we have a homomorphism $\hat\rho:\Gamma_0^+(2)\to\GL(2,\C)$ with
$\hat\rho(\gamma)$ defined by
$$
\left(\hat Y\big|_\ell\gamma\right)(z)=\hat\rho(\gamma)\hat Y(z).
$$
As in the case of $\SL(2,\Z)$, we can show that the image of
$\hat\rho$ actually lies in $\SL(2,\C)$.

\begin{Lemma} We have $\det\hat\rho(\gamma)=1$ for all
  $\gamma\in\Gamma_0^+(2)$.
\end{Lemma}

\begin{proof}
The proof of the lemma follows exactly that of Lemma \ref{lemma: det
  rho SL(2,Z)}. We find that for $\gamma\in\Gamma_0^+(2)$, we have
$$
\det\hat\rho(\gamma)=\frac{\left(F(z)^2\big|_{2(\ell+1)}\gamma\right)}
{F(z)^2}=1
$$
because $F(z)^2$ is a modular form of weight $2(\ell+1)$ on
$\Gamma_0^+(2)$.
\end{proof}

Now let $y_+(z)$ be the unique solution of \eqref{equation: DE
  Gamma0+(2)} of the form
$$
y_+(z)=q^{\kappa_\infty}\left(1+\sum_{j\ge 1}c_jq^j\right).
$$
Then analogous to Theorem \ref{theorem: infinity is not apparent
  implies}, we have the following result for $\Gamma_0^+(2)$
(Theorem \ref{theorem: main 2} in Section \ref{section:
  introduction}).

\begin{Theorem} \label{theorem: main Gamma0+(2)}
  Suppose that $Q(z)$ satisifies \eqref{equation: H} with
  $\kappa_\infty\in\frac12\Z_{\ge0}$. Then the following
  statements hold.
  \begin{enumerate}
  \item[(i)] The differential equation \eqref{equation: DE Gamma0+(2)}
    is not apparent at $\infty$.
  \item[(ii)] Let $y_2(z)=y_+(z)$ and $\hat y_1(z)=\left(\hat
      y_+|_\ell S\right)(z)$. Then $\hat y_1(z)=z\hat y_2(z)+\hat
    m_1(z)$ for some modular form $\hat m_1(z)$ in
    $\sM_{\ell-1}(\Gamma_0^+(2),\JS2\ell)$, where $\JS2\cdot$ is the
    Legendre symbol whose values are given by
    $$
    \JS2\ell=\begin{cases}
      1, &\text{if }\ell\equiv1,7\mod 8, \\
      -1, &\text{if }\ell\equiv3,5\mod 8. \end{cases}
    $$
  \item[(iii)] The ratio $h(z)=\JS2\ell y_1(z)/\sqrt2 y_2(z)$ is
    equivariant. That is, for all $\gamma\in\Gamma_0^+(2)$, we have
    $h(\gamma z)=\gamma\cdot h(z)$.
  \item[(iv)] Write $\hat y_+(z)$ as
    $$
    \hat y_+(z)=\JS2\ell\frac{\pi i}{4\sqrt2}\hat
    m_1(z)M_2^\ast(z)+\hat m_2(z).
    $$
    Then $\hat m_2(z)$ is a modular form in
    $\sM_{\ell+1}(\Gamma_0^+(2),\JS2\ell)$. Hence, $\hat y_+(z)$
    is a quasimodular form in
    $\wt\sM_{\ell+1}^{\le1}(\Gamma_0^+(2),\JS2\ell)$.
  \end{enumerate}
\end{Theorem}

The proof of (i) is the same as that of Lemma \ref{lemma: not apparent
  SL(2,Z)} and is omitted. We remark that if $\kappa_{\rho_2}\in\N$,
then it can be shown that the differential equation \eqref{equation:
  DE Gamma0+(2)} is apparent at $\infty$. (C.f. Remark \ref{remark:
  infinity is apparent}.) Thus, the condition that
$(2\kappa_{\rho_2},4)=1$ is necessary in order for \eqref{equation:
  DE Gamma0+(2)} to be not apparent at $\infty$. We
now prove the other parts of the theorem.

By Part (i) of the theorem, there is another solution $y_-(z)$ of
\eqref{equation: DE Gamma0+(2)} of the form
$$
y_-(z)=zy_+(z)+q^{-\kappa_\infty}\left(b_0+\sum_{j\ge 1}b_jq^j\right),
\qquad b_0\neq 0.
$$
Set
\begin{equation} \label{equation: y1 Gamma0+(2)}
  y_1(z)=\left(\hat y_+\big|_\ell
    S\right)(z)/F(z) \qquad
  y_2(z)=y_+(z).
\end{equation}
We first prove that $y_1$ and $y_2$ are linearly
independent. Hence, there are $d\neq 0$ and $m_1(z)$ such that
\begin{equation} \label{equation: m1 Gamma0+(2)}
  y_1(z)=dzy_2(z)+m_1(z),
\end{equation}
where $m_1(z)$ has a $q$-expansion of the
form $q^{-\kappa_\infty}(c_0+\sum_{j\ge 1}c_jq^j)$, $c_0\neq 0$.

\begin{Lemma} \label{lemma: linearly independent Gamma0+(2)}
  The solutions $y_i$, $i=1,2$, are linearly independent,
  and we have $d=\pm\sqrt2$.
\end{Lemma}

\begin{proof} The proof of linear independence of $y_1$ and $y_2$ is
  similar to that of Part (i) of Lemma \ref{lemma: y1,y2} and is
  skipped.

  To compute $d$, we use $\hat Y(z)=(\hat y_1(z),\hat y_2(z))^t$ as
  the basis. We have $\hat\rho(T)=\SM1d01$ and
  $\hat\rho(S)=\SM0{-1}10$ (see the proof of \eqref{equation: hat
    rho(S) SL(2,Z)}). Hence,
  $$
  \hat\rho(R)=\hat\rho(T)\hat\rho(S)
  =\M1d01\M0{-1}10=\M d{-1}10.
  $$
  Then $\hat\rho(R)^2=-\hat\rho(R)^{-2}$ yields
  $$
  \M{d^2-1}{-d}d{-1}=-\M{-1}d{-d}{d^2-1}.
  $$
  It follows that $d=\pm\sqrt{2}$, and the proof of the lemma is
  completed.
\end{proof}

\begin{Lemma} \label{lemma: h Gamma0+(2)}
Let $h(z)=y_1(z)/(dy_2(z))$. Then $h(Sz)=S\cdot h(z)$.
\end{Lemma}

\begin{proof}
By using the basis $\left(\hat{y}_1,\hat{y}_2\right)$, we have
$\hat{\rho}(S)=\SM0{-1}10$, which implies that
\begin{align*}
  h(Sz)=\frac{-y_2(z)}{dy_1(z)}
  =\frac{1}{d^2}\frac{-dy_2(z)}{y_1(z)}=\frac{-1}{2h(z)}
  =\frac{-1/\sqrt{2}}{2h(z)/\sqrt{2}}=S\cdot h(z).
\end{align*}
This proves the lemma.
\end{proof}

Note that together with Lemma \ref{lemma: epsilon Gamma0+(2)} below,
this lemma yields Part (iii) of the theorem.

\begin{Lemma} \label{lemma: m1 Gamma0+(2)}
Let $y_1(z)$ and $m_1(z)$ be defined as in \eqref{equation: y1
  Gamma0+(2)} and \eqref{equation: m1 Gamma0+(2)}, respectively.
Then $\hat m_1(z)$ satisfies
\begin{enumerate}
\item[(i)] $\hat m_1(z+1)=\hat m_1(z)$,
\item[(ii)] $\left(\hat m_1\big|_{\ell-1}S\right)(z)=\epsilon\hat
  m_1(z)$ if $d=\epsilon\sqrt2$, $\epsilon\in\{\pm 1\}$, and
\item[(iii)] $\hat m_1(z)\in\sM_{\ell-1}(\Gamma_0^+(2),\epsilon)$.
\end{enumerate}
\end{Lemma}

\begin{proof}
Part (i) is obvious. The proof of (ii) is similar to that of Lemma
\ref{lemma: hat(m2) is modular form}. For the convenience of readers,
we repeat it here. Write
$$
h(z)=\frac{y_1(z)}{dy_2(z)}=z+\frac{m_1(z)}{dy_2(z)}.
$$
By the previous lemma,
$$
-\frac{1}{2h(z)}=h(Sz)=h\left(-\frac{1}{2z}\right)
=-\frac{1}{2z}+\frac{m_1(Sz)}{dy_2(Sz)},
$$
and then
\begin{align*}
  \frac{m_1(Sz)}{dy_2(Sz)}=\frac{-1}{2h(z)}+\frac{1}{2z}
  =\frac{h(z)- z}{2h(z)z}=\frac{m_1(z)}{2dzh(z)y_2(z)}
  =\frac{m_1(z)}{2zy_1(z)}.
\end{align*}
Since
$$
\hat y_1(z)=\left(\hat y_2\big|_{\ell}
  S\right)(z)=(\sqrt{2}z)^{-\ell}\hat y_2(Sz),
$$
we deduce that
\begin{align*}
  \hat m_1(S z)&=\frac{d\hat m_1(z)}{2z}
                 \frac{\hat y_2(Sz)}{\hat y_1(z)}
            =\epsilon\frac{\hat m_1(z)}{\sqrt{2}z}
            \cdot(\sqrt{2}z)^{\ell}\\
&={\epsilon}(\sqrt{2}z)^{\ell-1}\hat m_1(z)\quad\text{if }
            d=\epsilon\sqrt2,
\end{align*}
i.e., $\left(\hat m_1\big|_{\ell-1}S\right)(z)={\epsilon}\hat m_1(z)$.

Now we have $R=TS$. Hence, by (i) and (ii), $\left(\hat
  m_1\big|_{\ell-1}R^2\right)(z)=\hat m_1(z)$. Since
$R^2=\SM1{-1}2{-1}$ and $T$ generate the group $\Gamma_0(2)$, we find
that $\hat m_1(z)$ is a modular form of weight $\ell-1$ on
$\Gamma_0(2)$. It is clear from (ii) that it belongs to the
Atkin-Lehner subspace with eigenvalue $\epsilon$.
\end{proof}

Define $m_2(z)$ by
\begin{equation} \label{equation: m2 Gamma0+(2)}
  y_+(z)=\frac{\pi i}{4d}m_1(z)M_2^\ast(z)+m_2(z).
\end{equation}
Note that since $m_1(z)=q^{-\kappa_\infty}(a_0+O(q))$, $a_0\neq0$, we
have $m_2(z)=q^{-\kappa_\infty}(b_0+O(q))$ for some $b_0\neq 0$.

\begin{Lemma}\label{lemma: transformation of hat m2} We have
  $\hat m_2(z)\in\sM_{\ell+1}(\Gamma_0^+(2),\epsilon)$,
  where $\epsilon\in\{\pm 1\}$ is determined by $d=\epsilon\sqrt2$.
\end{Lemma}

\begin{proof}
  As seen in the proof of \eqref{lemma: m1 Gamma0+(2)}, it suffices to
  prove that
  $$
  \hat m_2(z+1)=\hat m_2(z), \qquad
  \left(\hat m_2\big|_{\ell+1}S\right)(z)=\epsilon\hat m_2(z).
  $$
  The first identity is obvious. To prove the second identity, by
  \eqref{equation: M2* Gamma0+(2)}, \eqref{equation: m2 Gamma0+(2)}
  and Lemma \ref{lemma: m1 Gamma0+(2)}, we have
\begin{align*}
  \left(\hat y_+\big|_{\ell} S\right)(z)
  &=\sqrt{2}z\frac{\pi i}{4d}\left(\hat m_1\big|_{\ell-1}S\right)(z)
    \left(M_2^\ast\big|_2 S\right)(z)
    +\sqrt{2}z\left(\hat m_2\big|_{\ell+1}S\right)(z)\\
  &=\sqrt{2}z\left(\frac{\pi i}{4d}
    {\epsilon}\hat m_1(z)M_2^\ast(z)
    +\frac{\hat m_1(z)}{\sqrt 2z}\right)
    +\sqrt{2}z\left(\hat m_2\big|_{\ell+1}S\right)(z)\\
  &=\frac{\pi iz}4\hat m_1(z)M_2^\ast(z)+\hat m_1(z)+
    \sqrt 2z\left(\hat m_2\big|_{\ell+1}S\right)(z) \\
  &=dz\left(\hat y_+(z)-\hat m_2(z)\right)+\hat m_1(z)
    +\sqrt2z\left(\hat m_2\big|_{\ell+1}S\right)(z) \\
  &=\left(dz\hat y_+(z)+\hat m_1(z)\right)+\sqrt2 z\left(
    \left(\hat m_2\big|_{\ell+1}S\right)(z)-\epsilon\hat m_2(z)\right).
\end{align*}
We note that the first sum of the right-hand side is equal to
$\left(\hat y_+\big|_{\ell} S\right)(z)$. Therefore, we have
$$
\left(\hat m_2\big|_{\ell+1}S\right)(z)={\epsilon}\hat m_2(z).
$$
This completes the proof.
\end{proof}

It remains to determine $\epsilon$. We first prove a lemma.

\begin{Lemma} \label{lemma: no common zero}
  The two modular forms $\hat m_1(z)$ and $\hat m_2(z)$
  cannot have a common zero on $\H$.
\end{Lemma}

\begin{proof} Suppose that $\hat m_1(z)$ and $\hat m_2(z)$ have a
  common zero at $z_0\in\H$. Then $\hat y_1(z_0)=\hat y_2(z_0)=0$,
  which implies that both $y_+$ and $y_+|_{-1}S$ have
  the behavior $c(z-z_0)^{1/2+\kappa}(1+O(z-z_0))$ for some $c\neq 0$
  near $z_0$. Consequently, $y_+$ and $y_+|_{-1}S$ are linearly
  dependent, contradicting to Lemma \ref{lemma: linearly independent
    Gamma0+(2)}. We conclude that $\hat m_1(z)$ and $\hat m_2(z)$ have
  no common zero on $\H$.
\end{proof}
We now use this lemma to determine $\epsilon$.

\begin{Lemma} \label{lemma: epsilon Gamma0+(2)}
  We have
  $$
  \epsilon=\JS2\ell=\begin{cases}
    1, &\text{if }\ell\equiv1,7\mod 8, \\
    -1, &\text{if }\ell\equiv3,5\mod 8. \end{cases}.
  $$
\end{Lemma}

\begin{proof} We will prove case by case that the assumption
  $\epsilon=-\JS2\ell$ will imply that $\hat m_1(z)$ and $\hat m_2(z)$
  have a common zero at $\rho_2$. By the lemma above, this is absurd.
  Hence, $\epsilon$ must be equal to $\JS2\ell$.

  In general, \emph{if $z_0$ is an elliptic point of order $e$ of a
    subgroup $\Gamma$ of $\SL(2,\R)$ commensurable with $\SL(2,\Z)$,
    then any modular form of even weight $k$ with $(2e)\nmid k$ has a
    zero at $z_0$}
  (see, for instance, Proposition \ref{proposition: vanishing
    coefficients} in the appendix).
  From this property, we see that
  \begin{enumerate}
    \item[(i)] If $k\equiv 2\mod 4$, then any modular form of weight
      $k$ on $\Gamma_0(2)$ has a zero at $\rho_2$ since
      $\rho_2=(1+i)/2$ is an elliptic point of order $2$ of
      $\Gamma_0(2)$.
    \item[(ii)] If $k\equiv 4\mod 8$, then any modular form of weight
      $k$ on $\Gamma_0^+(2)$ has a zero at $\rho_2$ since $\rho_2$ is
      an elliptic point of order $4$ of $\Gamma_0^+(2)$.
  \end{enumerate}
    
  Consider first the case $\ell\equiv 3,5\mod 8$. Suppose that
  $\epsilon=1$. Then $\hat m_i(z)$ are modular forms on
  $\Gamma_0^+(2)$ of weight $\ell-1$ and $\ell+1$. By (i) and (ii)
  above, $\rho_2$ is a common zero of $\hat m_1$ and $\hat m_2(z)$, a
  contradiction to Lemma \ref{lemma: no common zero}. This proves
  $\epsilon=-1$.

  For the remaining cases where $\ell+1$ or $\ell-1$ is divisible by
  $8$, we claim that if $k$ is divisible by $8$, then any modular form
  in $\sM_k(\Gamma_0^+(2),-)$ has a zero at $\rho_2$, which then
  implies that the sign $\epsilon$ must be $1=\JS2\ell$. Indeed, if
  $8|k$ and $f\in\sM_k(\Gamma_0^+(2),-)$, then evaluating the two
  sides of the identity
  $$
  f\left(\frac{2z-1}{2z}\right)=f(Rz)
  =-(\sqrt2 z)^kf(z)
  $$
  at $\rho_2$, we see that $f(\rho_2)$ must be $0$ because $(\sqrt
  2\rho_2)^8=1$. This completes the
  proof of the lemma.
\end{proof}

With this lemma, we conclude the proof of Theorem \ref{theorem: main
  Gamma0+(2)}. We now give three examples and show they agree with our
theorem.

\begin{Example} Consider
  $$
  y''(z)=-4\pi^2\left(\frac k2\right)^2M_4(z)y(z),
  $$
  where $k\in\N$. By Theorem \ref{theorem: main Gamma0+(2)}, if
  $y_+(z)$ is the unique solution of the form
  $y_+(z)=q^{k/2}(1+\sum_{j\ge1}a_jq^j)$, then $M_8(z)^{k/2}y_+(z)$ is
  a quasimodular form in $\wt\sM_{4k}^{\le1}(\Gamma_0^+(2),(-1)^k)$.
  (Note that the function $F(z)$ and the integer $\ell$ in this case
  are $M_8(z)^{k/2}$ and $-1+4k$, respectively.) Moreover, the
  vanishing order of $M_8(z)^{k/2}y_+(z)$ at $\infty$ is
  $k=\dim\wt\sM_{4k}^{\le1}(\Gamma_0^+(2),(-1)^k)-1$. Hence,
  $M_8(z)^{k/2}y_+(z)$ is an extremal quasimodular form. This agrees
  with Part (i) of Theorem \ref{theorem: DE Gamma0+(2)}.
\end{Example}

\begin{Example} Consider
  $$
  y''(z)=-4\pi^2\left(\left(\frac k2\right)^2M_4(z)
    -32\frac{M_8(z)}{M_4(z)}\right)y(z),
  $$
  where $k\in\Z_{\ge 0}$.
  We have $\kappa_\infty=k/2$, $\kappa_{\rho_1}=1/2$,
  $\kappa_{\rho_2}=3/2$, and $\kappa=1/2$ elsewhere. Hence, the
  function $F(z)$ and the integer $\ell$ in this case are
  $F(z)=M_8(z)^{k/2}M_2(z)$ and $\ell=-1+4k+2=4k+1$. Theorem
  \ref{theorem: main Gamma0+(2)} asserts that if $y_+(z)$ is the
  unique solution of the form $y_+(z)=q^{k/2}(1+\cdots)$, then
  $M_8(z)^{k/2}M_2(z)y_+(z)$ is a quasimodular form in
  $\wt\sM_{4k+2}^{\le1}(\Gamma_0^+(2),(-1)^k)$. As in the previous
  example, we find that it is an extremal quasimodular form, agreeing
  with Part (ii) of Theorem \ref{theorem: DE Gamma0+(2)}.
\end{Example}

\begin{Example} Consider
  $$
  y''(z)=-4\pi^2\left(\left(\frac k2\right)^2M_4(z)
    +\frac{128M_4(z)M_8(z)}{M_4(z)^2-256M_8(z)}\right)y(z),
  \quad k\in\Z_{\ge 0}.
  $$
  We have $\kappa_\infty=k/2$ and by \eqref{equation: indicial
    Gamma0+(2)}, $\kappa_{\rho_1}=3/2$. The integer $\ell$ in Theorem
  \ref{theorem: main Gamma0+(2)} is $k+3$. According to the theorem,
  $M_8(z)^{k/2}M_6(z)y_+(z)/M_2(z)$ is a quasimodular form in
  $\wt\sM_{4k+4}^{\le1}(\Gamma_0^+(2),(-1)^{k-1})$. Indeed, when
  $k=1$, we have
  $$
  y_+(z)=q^{1/2}\left(1+70q+5926q^2+503696q^3+42822181q^4
  +\cdots\right)
  $$
  and
  $$
  1120M_8(z)^{1/2}\frac{M_6(z)}{M_2(z)}y_+(z)
  =M_2^\ast(z)M_6(z)-M_4(z)^2+1280M_8(z),
  $$
  which belongs to $\wt\sM_8^{\le 1}(\Gamma_0^+(2))$, while
  when $k=2$, we have
  $$
  y_+(z)=q+\frac{176}3q^2+\frac{13706}3q^3+\frac{1151072}3q^4+\cdots
  $$
  and
  \begin{equation*}
    \begin{split}
      &19008M_8(z)\frac{M_6(z)}{M_2(z)}y_+(z)
      =M_2^\ast(z)M_2(z)\left(-M_4(z)^2+384M_8(z)\right) \\
      &\qquad\qquad+M_2(z)M_4(z)M_6(z)-288M_4^-(z)M_8(z),
    \end{split}
  \end{equation*}
  which lies in $\wt\sM_{12}^{\le1}(\Gamma_0^+(2),-)$,
  where $M_4^-(z)=(4E_4(2z)-E_4(z))/3\in\sM_4(\Gamma_0^+(2),-)$.
\end{Example}

\section{MODE in the case of $\Gamma_0^+(3)$}
\label{section: Gamma0+(3)}
We first recall the following information about $\Gamma_0^+(3)$.
\begin{enumerate}
\item[(i)] The group $\Gamma_0^+(3)$ is generated by
  $S=\frac{1}{\sqrt{3}}\SM 0{-1}30$ and
  $R=TS=\frac{1}{\sqrt{3}}\SM3{-1}30$, where $T=\SM1101$.
  They satisfy $S^2=R^6=-I$.
\item[(ii)] $\Gamma$ has only one cusp at $\infty$ with width
  $1$, and two elliptic points $\rho_1=i/\sqrt{3}$ (order $2$) and
  $\rho_2=(3+\sqrt{-3})/6$ (order $6$). Their stabilizer subgroups are
  generated by $T$, $S$, and $R$, respectively.
\item[(iii)] The ring $\oplus_{k=0}^\infty\sM_k(\Gamma_1(3))$,
  where $k$ runs through both even and odd integers, is freely
  generated by
  $$
  M_1(z)=\sum_{m,n\in\Z}q^{m^2+mn+n^2}
  $$
  and
  $$
  M_3(z)=\frac{9E_4(3z)-E_4(z)}{8M_1(z)}
  =\frac{\eta(z)^9}{\eta(3z)^3}-27\frac{\eta(3z)^9}{\eta(z)^3}
  $$
  of weight $1$ and $3$, respectively.
\item[(iv)] Let $\chi$ be the character of $\Gamma_0^+(3)$ determined
  by $\chi(S)=\chi(R)=-i$. Then
  $$
  M_1(z)\in\sM_1(\Gamma_0^+(3),\chi), \qquad
  M_3(z)\in\sM_1(\Gamma_0^+(3),\chi).
  $$
\end{enumerate}
For our purpose, we also need the modular forms
\begin{equation*}
  \begin{split}
    M_4^+(z)&=\frac{9E_4(3z)+E_4(z)}{10}=M_1(z)^4
    \in\sM_4(\Gamma_0^+(3)), \\
    M_6^-(z)&=\eta(z)^6\eta(3z)^6=\frac{M_1(z)^6-M_3(z)^2}{108}
    \in\sM_6(\Gamma_0^+(3),-).
  \end{split}
\end{equation*}
Up to $\Gamma^+_0(3)$-equivalence, $M_1(z)$, $M_3(z)$, and $M_6^-(z)$
have simple zeros at $\rho_2$, $\rho_1$, and $\infty$, respectively.
The ratio $M_1(z)^6/M_6^-(z)$ is a Hauptmodul for $\Gamma_0^+(3)$.
Let also
$$
M_2^\ast(z):=\frac1{2\pi i}\frac{(M_6^-)'(z)}{M_6^-(z)}
=\frac{3E_2(3z)+E_2(z)}4,
$$
which is a quasimodular form of weight $2$ and depth $1$ on
$\Gamma_0^+(3)$ and satisfies
\begin{equation} \label{equation: M* transform Gamma0+(3)}
\left(M_2^\ast|_2\gamma\right)(z)=M_2^\ast(z)
+\frac{3c}{\pi i(cz+d)}, \qquad
\gamma=\M abcd\in\Gamma_0^+(3).
\end{equation}

Consider the differential equation
\begin{equation} \label{equation: DE Gamma0+(3)}
y''(z)=-4\pi^2Q(z)y(z),
\end{equation}
where $Q(z)$ is a meromorphic modular form of weight $4$ on
$\Gamma_0^+(3)$, and assume that it satisfies the condition
\eqref{equation: H}. Let $z_1,\ldots,z_m$ be
$\Gamma_0^+(3)$-inequivalent singularities of \eqref{equation: DE
  Gamma0+(3)} other than $\rho_1$ 
and $\rho_2$. Let $\pm\kappa_\infty$, $1/2\pm\kappa_{\rho_1}$,
$1/2\pm\kappa_{\rho_2}$, and $1/2\pm\kappa_j$, $j=1,\ldots,m$, be the
local exponents of \eqref{equation: DE Gamma0+(3)} at $\infty$,
$\rho_1$, $\rho_2$, and $z_j$, respectively. For $j=1,\ldots,m$, let
$t_j=M_2^-(z_j)^3/M_6^-(z_j)$, and set
$$
F_j(z)=M_2^-(z)^3-t_jM_6^-(z).
$$
Define
\begin{equation} \label{equation: F Gamma0+(3)}
F(z)=M_6^-(z)^{\kappa_\infty}M_3(z)^{\kappa_{\rho_1}-1/2}
M_1(z)^{\kappa_{\rho_2}-1/2}\prod_{j=1}^m
F_j(z)^{\kappa_j-1/2},
\end{equation}
and set
\begin{equation} \label{equation: ell Gamma0+(3)}
\ell=-1+6\kappa_\infty+3\left(\kappa_{\rho_1}-\frac12\right)
+\left(\kappa_{\rho_2}-\frac12\right)
+6\sum_{j=1}^m\left(\kappa_j-\frac12\right).
\end{equation}
Note that by Condition \eqref{equation: H}, both $\kappa_{\rho_1}-1/2$
and $\kappa_{\rho_2}-1/2$ are integers, and
$(2\kappa_{\rho_2},3)=1$. Hence $\ell$ is an integer not divisible by
$3$. For any function $f(z)$ on $\H$, we let $\hat f(z)$
denote
$$
\hat f(z)=F(z)f(z).
$$
Note that $F(z)^2$ is a holomorphic modular form of weight
$2(\ell+1)$ on $\Gamma_0^+(3)$ with some character depending on
$\ell$. By construction, for
any solution $y(z)$ of \eqref{equation: DE Gamma0+(3)}, the function
$\hat y(z)$ is a single-valued function holomorphic throughout
$\H$. Also, its order at a singularity $z_j$ is either $0$ or
$2\kappa_j$, and a similar property holds for $\infty$, $\rho_1$, and
$\rho_2$. Hence, for any fundamental solution $(y_1(z),y_2(z))$ of
\eqref{equation: DE Gamma0+(3)}, we have a representation
$\hat\rho:\Gamma_0^+(3)\to\GL(2,\C)$ given by
$$
\column{\hat y_1\big|_\ell\gamma}{\hat y_2\big|_\ell\gamma}
=\hat\rho(\gamma)\column{\hat y_1}{\hat y_2}.
$$
We now state the main result of this section (Theorem \ref{theorem: DE
  Gamma0+(3)} of Section \ref{section: introduction}).

\begin{Theorem} \label{theorem: main Gamma0+(3)}
  Suppose that $Q(z)$ satisify \eqref{equation: H} with
  $\kappa_\infty\in\frac12\Z_{\ge0}$. Let
  $\ell$ be the integer defined by \eqref{equation: ell Gamma0+(3)}.
  Set
  $$
  \delta=\begin{cases}
    \chi^0,&\text{if }\ell\equiv 1,11\mod12, \\
    \chi^1, &\text{if }\ell\equiv 2,4\mod 12, \\
    \chi^2, &\text{if }\ell\equiv 5,7\mod 12, \\
    \chi^3, &\text{if }\ell\equiv 8,10\mod 12, \end{cases}
  $$
  where $\chi$ is the character of $\Gamma_0^+(3)$ defined by
  $\chi(S)=\chi(R)=-i$. The following statements hold.
  \begin{enumerate}
  \item[(i)] The differential equation \eqref{equation: DE Gamma0+(3)}
    is not apparent at $\infty$.
  \item[(ii)] Let $y_2(z)=y_+(z)$ and $\hat y_1(z)=\left(\hat
      y_+|_\ell S\right)(Z)$. Then $\hat y_1(z)=z\hat y_2(z)+\hat
    m_1(z)$ for some modular form $\hat m_1(z)$ in
    $\sM_{\ell-1}(\Gamma_0^+(3),\delta)$.
  \item[(iii)] The ratio 
    $h(z)=\delta(S)^{-1}y_1(z)/\sqrt 3y_2(z)$ is
    equivariant. That is, for all $\gamma\in\Gamma_0^+(3)$, we have
    $h(\gamma z)=\gamma\cdot h(z)$.
  \item[(iv)] Write $\hat y_+(z)$ as
    $$
    \hat y_+(z)=\delta(S)^{-1}\frac{\pi i}{3\sqrt3}\hat
    m_1(z)M_2^\ast(z)+\hat m_2(z).
    $$
    Then $\hat m_2(z)$ is a modular form in
    $\sM_{\ell+1}(\Gamma_0^+(3),\delta)$. Hence, $\hat y_+(z)$
    is a quasimodular form in
    $\wt\sM_{\ell+1}^{\le1}(\Gamma_0^+(3),\delta)$.
  \end{enumerate}
\end{Theorem}

Note that when $k$ is even, we have $\sM_k(\Gamma_0^+(3),\chi^0)$ and
$\sM_k(\Gamma_0^+(3),\chi^2)$ are simply $\sM_k(\Gamma_0^+(3))$ and
$\sM_k(\Gamma_0^+(3),-)$, respectively, so
$\wt\sM_{\ell+1}^{\le1}(\Gamma_0^+(3),\delta)$ in the case of odd
$\ell$ can also be written as $\wt\sM_{\ell+1}^{\le1}
\left(\Gamma_0^+(3),\JS{12}\cdot\right)$.

Apart from the difference that the integer $\ell$ can be even, the
proof is very similar to those of Theorems \ref{theorem: infinity
  is not apparent implies} and \ref{theorem: main Gamma0+(2)}.

\begin{Lemma} \label{lemma: det rho Gamma0+(3)}
  We have $\det\hat\rho(T)=1$ and
  $\det\hat\rho(S)=(-1)^{\ell-1}$.
\end{Lemma}

\begin{proof} Following the proof of Lemma \ref{lemma: det rho
    SL(2,Z)}, we have
  $$
  \det\hat\rho(\gamma)
  =\frac{\left(F(z)^2\big|_{2(\ell+1)}\gamma\right)}{F(z)^2}
  $$
  for $\gamma\in\Gamma_0^+(3)$. It is clear that when $\gamma=T$, we
  have $\det\hat\rho(T)=1$. 
  For the case $\gamma=S$, we note that
  $$
  \frac{M_6^-|_6S}{M_6^-}=\frac{M_3^2|_6S}{M_3^2}
  =\frac{M_1^2|_2S}{M_1^2}=\frac{F_j|_6S}{F_j}=-1.
  $$
  It follows that
  $$
  \det\hat\rho(S)
  =(-1)^{2\kappa_\infty+(\kappa_{\rho_1}-1/2)
    +(\kappa_{\rho_2}-1/2)+\sum_j(2\kappa_j-1)}
  =(-1)^{\ell-1}.
  $$
  This proves the lemma.
\end{proof}

The proofs of the next two lemmas are different from the ones in the
case of $\SL(2,\Z)$ and $\Gamma_0^+(2)$ because $R^6=-I$.

\begin{Lemma} \label{lemma: infinity is not apparent Gamma0+(3)}
  The differential equation \eqref{equation: DE
    Gamma0+(3)} is never apparent at $\infty$.
\end{Lemma}

\begin{proof} If \eqref{equation: DE Gamma0+(3)} is apparent at
  $\infty$, then with respect to any basis, we have
  $\hat\rho(T)=I$. Therefore, there exists a basis $(\hat y_1(z),\hat
  y_2(z))$ such that $\hat\rho(T)$, $\hat\rho(S)$, and
  $\hat\rho(R)=\hat\rho(S)$ are all diagonal with the diagonal entries
  of $\hat\rho(R)$ lying inside $\{\pm1,\pm i\}$, which
  implies that $\hat y_1^2$ and $\hat y_2^2$ are both modular forms of
  weight $2\ell$ on $\Gamma_0(3)$. (Note that $\Gamma_0(3)$ is
  generated by $T$ and $R^2$.) Hence, the order of $\hat y_1^2$ at
  $\rho_2$ is congruent to $2\ell$ modulo $3$ (see Proposition
  \ref{proposition: vanishing coefficients}), which is congruent to
  $2\kappa_{\rho_2}$ modulo $3$. However, this contradicts to the fact
  that the order of $\hat y_1^2$ at $\rho_2$ is either $0$ or
  $4\kappa_{\rho_2}$. (Recall that by \eqref{equation: H},
  $3\nmid2\kappa_{\rho_2}$.) We conclude that 
  \eqref{equation: DE Gamma0+(3)} cannot be apparent at $\infty$,
  under the condition \eqref{equation: H}.
\end{proof}

\begin{Lemma} The solutions $y_+\big|_{-1}S$ and $y_+$ are linearly
  independent.
\end{Lemma}

\begin{proof} Suppose that $y_+\big|_{-1}S$ and $y_+$ are linearly
  dependent. Then $\hat y_+\big|_{\ell}S=\lambda y_+$ for some
  $\lambda\in\C$. Then with respect to the basis $(\hat y_-,\hat
  y_+)$, we have
  $$
  \hat\rho(T)=\M1b01, \qquad\hat\rho(S)
  =\M{(-1)^{\ell-1}/\lambda}c0\lambda
  $$
  for some complex numbers $b\neq0$ and $c$. Since
  $\hat\rho(S)^2=\hat\rho(-I)=(-1)^\ell I$, we have
  $\lambda\in\{\pm i\}$ when $\ell$ is odd and $\lambda\in\{\pm 1\}$
  when $\ell$ is even. In either case, $\hat y_+^2$ is a modular form
  of weight $2\ell$ on $\Gamma_0(3)$. By the same reasoning as in the
  proof of Lemma \ref{lemma: infinity is not apparent Gamma0+(3)},
  we find that this yields a contradiction. Hence, $y_+\big|_{-1}S$ is
  linearly independent of $y_+$.
\end{proof}

Since \eqref{equation: DE Gamma0+(3)} is not apparent at $\infty$ and $y_+\big|_{-1}S$ and $y_+$ are linearly independent, there is a
nonzero constant $d$ and $m_1(z)$ with
$m_1(z)=q^{-\kappa_\infty}(a_0+\sum_{j\ge 1}a_jq^j)$, $a_0\neq 0$,
such that
\begin{equation} \label{equation: hat m1 Gamma0+(3)}
\left(\hat y_+\big|_\ell S\right)(z)=dz\hat y_+(z)+\hat m_1(z)
\end{equation}
(and $dzy_+(z)+m_1(z)$ is a solution of \eqref{equation: DE
  Gamma0+(3)}).

\begin{Lemma} If $\ell$ is even, then $d=\pm\sqrt{-3}$ and if $\ell$
  is odd, then $d=\pm\sqrt3$.
\end{Lemma}

\begin{proof} We use $\hat y_+\big|_\ell S$ and $\hat y_+$ as a
  basis. By Lemma \ref{lemma: det rho Gamma0+(3)},
  $$
  \hat\rho(R)=\hat\rho(T)\hat\rho(S)
  =\M1d01\M0{(-1)^\ell}10=\M d{(-1)^\ell}10,
  $$
  which shows that the characteristic polynomial of $\hat\rho(R)$ is
  $x^2-dx+(-1)^{\ell-1}$. On the other hand, we have
  $\hat\rho(R)^6=\hat\rho(-I)=(-1)^\ell I$. When $\ell$ is even, the
  eigenvalues of $\hat\rho(R)$ are two sixth roots of unity
  $\lambda_1$ and $\lambda_2$ such that $\lambda_1\lambda_2=-1$ and
  $\lambda_1+\lambda_2=d\neq 0$. We find that $d$ must be one of
  $\pm\sqrt{-3}$. Likewise, when $\ell$ is odd, the eigenvalues
  $\lambda_1$ and $\lambda_2$ of $\hat\rho(R)$ satisfy
  $\lambda_1\lambda_2=1$ and $\lambda_1^6=\lambda_2^6=-1$, but
  $\lambda_1+\lambda_2=d\neq 0$. We find that $d$ must be of one
  $\pm\sqrt3$.
\end{proof}

\begin{Lemma} Let $h(z)=y_1(z)/dy_2(z)$. Then $h(Sz)=S\cdot h(z)$.
\end{Lemma}

\begin{proof} The proof is almost the same as that of Lemma
  \ref{lemma: h Gamma0+(2)}. The only difference is that when $\ell$
  is even, we have $\hat\rho(S)=\SM0110$, but in this case we have
  $d^2=-3$, by the lemma above, and $h(Sz)=S\cdot h(z)$ still holds.
\end{proof}

Define $m_2(z)$ by
\begin{equation} \label{equation: m2 Gamma0+(3)}
  y_+(z)=\frac{\pi i}{3d}m_1(z)M_2^\ast(z)+m_2(z),
\end{equation}
and let $\epsilon=d/\sqrt3$.

\begin{Lemma} The following holds.
  \begin{enumerate}
  \item[(i)] $\hat m_1(z+1)=\hat m_1(z)$ and $\left(\hat
      m_1\big|_{\ell-1}S\right)(z)=\epsilon\hat m_1(z)$.
  \item[(ii)] $\hat m_2(z+1)=\hat m_2(z)$ and $\left(\hat
      m_2\big|_{\ell+1}S\right)(z)=\epsilon\hat m_2(z)$.
  \end{enumerate}
\end{Lemma}

\begin{proof} The proof follows exactly those of Lemmas \ref{lemma: m1
    Gamma0+(2)} and \ref{lemma: transformation of hat m2} and is
  omitted.
\end{proof}

It remains to determine $\epsilon$.

\begin{Lemma} We have $\epsilon=\delta(S)$.
\end{Lemma}

\begin{proof} We first note that Lemma \ref{lemma: no common zero}
  is also valid here. Our strategy is to show that if
  $\epsilon\neq\delta(S)(=\delta(R))$, then $\hat m_1(z)$ and $\hat
  m_2(z)$ have a common zero at $\rho_2$.

  When $\ell$ is odd, the argument is the same as that in the proof of
  Lemma \ref{lemma: epsilon Gamma0+(2)}. We shall skip the proof in
  this case.

  Now consider the case $\ell$ is even. Since $\ell$ is not divisible
  by $3$, exactly one of $\ell-1$ and $\ell+1$ is not a multiple of
  $3$. Let $k$ be the one in $\{\ell-1,\ell+1\}$ that is a
  multiple of $3$, and $k'$ be the other. Since $\rho_2$ is an
  elliptic point of order $3$ on $X_0(3)$, any modular form of weight
  $k'$ has a zero at $\rho_2$. On the other hand, if
  $f\in\sM_k(\Gamma_0^+(3),\chi)$, then we have
  $$
  f(Rz)=\chi(R)(\sqrt 3z)^kf(z).
  $$
  Observe that $\sqrt 3\rho_2=(\sqrt3+i)/2=e^{2\pi i/12}$. If $k\equiv
  9\mod 12$, then $\chi(R)(\sqrt3\rho_2)=(-i)(-i)=-1$ and hence
  $f(\rho_2)=0$. Therefore, if $k\equiv9\mod 12$, i.e., if $\ell\equiv
  8,10\mod 12$, the assumption that $\epsilon=-i$ leads to a
  contradiction. Likewise, if $\ell\equiv 2,4\mod 12$, then $\epsilon$
  cannot be $i$. This proves the lemma.
\end{proof}

Having completed the proof of the theorem, we now give some examples.

\begin{Example} Consider the differential equation
  $$
  y''(z)=-4\pi^2\left(\frac k2\right)^2M_4^+(z)y(z),
  $$
  where $k$ is a positive integer. The function $F(z)$ and the integer
  $\ell$ in this case are $M_6^-(z)^{k/2}$ and $-1+3k$. The character
  $\delta$ in Theorem \ref{theorem: main Gamma0+(3)} is $\chi^k$ and
  the theorem asserts that $F(z)y_+(z)$ is a quasimodular form in
  $\wt\sM_{3k}^{\le1}(\Gamma_0^+(3),\chi^k)$. Indeed, according to
  Part (i) of Theorem \ref{theorem: DE Gamma0+(3)}, if $f(z)$ is an
  extremal quasimodular form in
  $\wt\sM_{3k}^{\le1}(\Gamma_0^+(3),\chi^k)$, then $f(z)/M_6^-(z)^{k/2}$
  is a solution of the differential equation above.
\end{Example}

\begin{Example} Consider the differential equation
  $$
  y''(z)=-4\pi^2\left(\left(\frac k2\right)^2M_4^+(z)
    -18\frac{M_6^-(z)}{M_2(z)}\right)M_4^+(z)y(z),
  $$
  where $k$ is a positive integer and
  $M_2(z)=(3E_2(3z)-E_2(z))/2=M_1(z)^2$. The function $F(z)$ and the
  integer $\ell$ in this case are $M_6^-(z)^{k/2}M_2(z)$ and $3k+1$,
  respectively, with $\kappa_{\rho_2}=5/2$. The character
  $\delta$ in this case is again $\chi^k$ and
  the theorem asserts that $F(z)y_+(z)$ is a quasimodular form in
  $\wt\sM_{3k+2}^{\le1}(\Gamma_0^+(3),\chi^k)$. This agrees with Part
  (ii) of Theorem \ref{theorem: DE Gamma0+(3)}.
\end{Example}

\begin{Example} Consider the differential equation
  $$
  y''(z)=-4\pi^2\left(\left(\frac k2\right)^2M_4^+(z)
    +54\frac{M_4^+(z)M_6^-(z)}{M_3(z)^2}\right)y(z), \quad
  k\in\Z_{\ge 0}.
  $$
  We have $\kappa_\infty=k/2$ and by \eqref{equation: indicial
    Gamma0+(3)}, $\kappa_{\rho_1}=3/2$. The integer $\ell$ in Theorem
  \ref{theorem: main Gamma0+(3)} is $3k+2$ and the character $\delta$
  is $\chi^{k+1}$. The theorem predicts that
  $M_6^-(z)^{k/2}M_3(z)y_+(z)$ belongs to the space
  $\wt\sM_{3k+3}^{\le1}(\Gamma_0^+(3),\chi^{k+1})$. Indeed, for $k=0$,
  we compute that
  $$
  y_+(z)=1+54q+1944q^2+73092q^3+2749032q^4+\cdots
  $$
  and find that
  $$
  2M_3(z)y_+(z)=3M_2^\ast(z)M_1(z)-M_3(z)\in
  \wt\sM_3^{\le1}(\Gamma_0^+(3),\chi).
  $$
  For $k=1$, we have
  $$
  y_+(z)=q^{1/2}\left(1+30q+1119q^2+42077q^3+1582920q^4+\cdots\right)
  $$
  and
  $$
  360M_6(z)^{1/2}M_3(z)y_+(z)=M_2^\ast(z)M_1(z)M_3(z)
  +3M_1(z)^6-4M_3(z)^2,
  $$
  which does lie in $\wt\sM_6^{\le
    1}(\Gamma_0^+(3),\chi^2)=\wt\sM_6^{\le1}(\Gamma_0^+(3),-)$.
  For $k=2$, we find that
  $$
  y_+(z)=q+26q^2+888q^3+32818q^4+1231645q^5+\cdots
  $$
  and
  \begin{equation*}
    \begin{split}
      18144M_6(z)M_3(z)y_+(z)
      &=M_2^\ast(z)(3M_1(z)^7-9M_1(z)M_3(z)^2) \\
      &\qquad\qquad  -M_1(z)^6M_3(z)+7M_3(z)^3,
    \end{split}
  \end{equation*}
  which indeed belongs to $\wt\sM_9^{\le1}(\Gamma_0^+(3),\chi^3)$.
  For $k=3$, we have
  $$
  y_+(z)=q^{3/2}\left(1+27q+\frac{4131}5q^2
    +\frac{146031}5q^3+\frac{5426163}5q^4+\cdots\right)
  $$
  and
  \begin{equation*}
    \begin{split}
      &2138400M_6(z)^{3/2}M_3(z)y_+(z) \\
      &\qquad=M_2^\ast(z)\left(M_1(z)^7M_3(z)
        +35M_1(z)M_3(z)^3\right) \\
      &\qquad+3M_1(z)^{12}-19M_1(z)^6M_3(z)^2-20M_3(z)^4,
    \end{split}
  \end{equation*}
  which is indeed lying in $\wt\sM_{12}^{\le 1}(\Gamma_0^+(3))$.
\end{Example}

\appendix
\section{Existence of $Q(z)$ satisfying \eqref{equation: H}}
Let $\Gamma$ be a subgroup of $\SL(2,\R)$ commensurable with
$\SL(2,\Z)$ and $Q(z)$ be a meromorphic modular form of weight $4$ on
$\Gamma$. Consider the modular differential equation
\begin{equation} \label{equation: y''=-Qy appendix}
y''(z)=-4\pi^2Q(z)y(z), \qquad z\in\H.
\end{equation}
We assume that the differential equation is Fuchsian, i.e., that
the order of pole at any pole of $Q(z)$ is $\le 2$. Let $z_0\in\H$ be
a pole of $Q(z)$. In Section \ref{section: ODE and quasi}, we have
discussed how to compute the local exponents and determine whether
\eqref{equation: y''=-Qy appendix} is apparent at $z_0$
using the standard Laurent expansions. However, in order to take
advantage of the modular properties of $Q(z)$, here we shall introduce
another type of expansions of modular forms.

For an integer $k$, let $\partial_k$ be the Shimura-Maass operator
defined in Definition \ref{definition: Shimura-Maass}.
By Lemma \ref{lemma: Maass}, when $f$ is a nearly holomorphic modular
form of weight $k$ on $\Gamma$, $\partial_kf$ is a nearly holomorphic
modular form of weight $k+2$ on $\Gamma$.
For a positive integer $n$, we let
$$
\partial_k^n=\partial_{k+2n-2}\partial_{k+2n-4}\ldots
\partial_{k+2}\partial_k.
$$
If the weight of $f$ is clear from the context, we will simply write
$\partial_k^n$ as $\partial^n$.

\begin{Proposition}[{\cite[Proposition 17]{Zagier123}}]
  \label{proposition: power series expansion}
  Assume that $f(z)$ is a (holomorphic) modular form
  of weight $k$ on $\Gamma$. Then for a point $z_0\in\H$, we have
  \begin{equation} \label{equation: power series expansion}
  f(z)=(1-w)^k\sum_{n=0}^\infty
  \frac{(\partial_k^nf)(z_0)(-4\pi\Im z_0)^n}{n!}w^n, \quad
  w=\frac{z-z_0}{z-\overline z_0}.
  \end{equation}
\end{Proposition}

We will call the expansion of a meromorphic modular form
$f(z)$ of weight $k$ of the form
$$
f(z)=(1-w)^k\sum_{n=n_0}^\infty a_nw^n, \qquad
w=\frac{z-z_0}{z-\overline z_0},
$$
the \emph{power series expansion} of $f(z)$ at $z_0$.
Note that there is a misprint in \cite[Proposition 17]{Zagier123}, in
which a minus sign is missing.

One of the main advantages to use the power series expansion in
$w=(z-z_0)/(z-\overline z_0)$ instead of the usual expansion in
$z-z_0$ is that the coefficients can be recursively computed. To
describe the recursion, let us first recall the definition of Serre's
derivative. In the case of $\Gamma=\SL(2,\Z)$, the well-known
Ramanujan's identities state that
\begin{equation} \label{equation: Ramanujan}
  \begin{split}
    D_qE_2(z)&=\frac{E_2(z)^2-E_4(z)}{12}, \\
    D_qE_4(z)&=\frac{E_2(z)E_4(z)-E_6(z)}3, \\
    D_qE_6(z)&=\frac{E_2(z)E_6(z)-E_4(z)^2}2,
  \end{split}
\end{equation}
where $D_q=qd/dq$. Generalizing these identities, we can easily verify
that if $f(z)$ is a modular form of weight $k$ on $\SL(2,\Z)$, then
$D_qf(z)-kE_2(z)f(z)/12$ is a modular form of weight $k+2$. Then the
Serre's derivative $\vartheta_k$ of weight $k$ in the case of
$\Gamma=\SL(2,\Z)$ is defined to
$$
\vartheta_kf:=D_qf-\frac k{12}E_2f.
$$
For a general group $\Gamma$, we choose a quasimodular form $\phi$ of
weight $2$ and depth $1$ on $\Gamma$ and let $\alpha$ be the nonzero
number such that
$$
\left(\phi|_2\gamma\right)(z)=\phi(z)+\frac{\alpha c}{2\pi i(cz+d)}
$$
for all $\gamma=\SM abcd\in\Gamma$. Then we can analogously prove that
if $f$ is a modular form of weight $k$ on $\Gamma$, then
$D_qf-k\phi f/\alpha$ is a modular form of weight $k+2$ on $\Gamma$.
Also,
$$
M(z):=D_q\phi(z)-\frac{\phi(z)^2}\alpha
$$
is a modular form of weight $4$ on $\Gamma$.

\begin{Definition} \label{definition: Serre}
  The differential operator
  $\vartheta_k:\sM_k(\Gamma)\to\sM_{k+2}(\Gamma)$ defined by
  $$
  \vartheta_kf:=D_qf-\frac k\alpha\phi f
  $$
  is called \emph{Serre's derivative} of weight $k$ (with respect to
  $\phi$). If the weight of $f$ is clear from the context, we will
  omit the subscript $k$. We then define recursively
  $\vartheta^{[n]}f$ by
  $$
  \vartheta^{[0]}f=f, \qquad
  \vartheta^{[1]}f=\vartheta f,
  $$
  and
  $$
  \vartheta^{[n+1]}f=\vartheta(\vartheta^{[n]}f)
  +\frac{n(k+n-1)}\alpha M\vartheta^{[n-1]}f
  $$
  for $n\ge 1$.
\end{Definition}

The Shimura-Maass differential operators and Serre's derivatives are
connected through the following combinatorial identity.

\begin{Proposition} Set
  $$
  \phi^\ast(z)=\phi(z)+\frac{\alpha}{2\pi i(z-\overline z)}.
  $$
  For a modular form $f(z)$ of weight $k$ on $\Gamma$, define two
  formal power series
  $$
  \wt f_\partial(z,X):=\sum_{n=0}^\infty
  \frac{\partial^nf(z)}{n!(k)_n}X^n
  $$
  and
  $$
  \wt f_\vartheta(z,X):=\sum_{n=0}^\infty
  \frac{\vartheta^{[n]}f(z)}{n!(k)_n}X^n,
  $$
  where $(k)_n=k(k+1)\ldots(k+n-1)$ is the Pochhammer symbol.
  Then we have
  $$
  \wt f_\partial(z,X)=e^{X\phi^\ast(z)/\alpha}\wt f_\vartheta(z,X)
  $$
  as formal power series in $X$.
\end{Proposition}

\begin{proof} See \cite[Section 5.2]{Zagier123}.
\end{proof}

Another way to state the proposition is that
$$
\partial^nf(z_0)=\sum_{r=0}^n\frac{n!}{r!}\binom{k+n-1}{k+r-1}
\left(\frac{\phi^\ast(z_0)}{\alpha}\right)^{n-r}\vartheta^{[r]}f(z_0)
$$
for all $n$.
This gives a formula for the coefficients of the power
series expansion of $f$. However, while the values
$\vartheta^{[r]}f(z_0)$ can be computed recursively
using identities analogous to \eqref{equation: Ramanujan},
the determination of $\phi^\ast(z_0)$ can be problematic in practice,
so we will introduce yet another type of expansions using a different
local parameter for $z_0$.

Letting $y_0=\Im z_0$, by \eqref{equation: power series expansion},
we have
\begin{equation*}
  \begin{split}
    f(z)&=(1-w)^k\sum_{n=0}^\infty w^n(-4\pi y_0)^n \\
    &\qquad\qquad\qquad\sum_{r=0}^n\frac1{r!}
    \frac{(k+n-1)!}{(n-r)!(k+r-1)!}
    \left(\frac{\phi^\ast(z_0)}{\alpha}\right)^{n-r}
    \vartheta^{[r]}f(z_0) \\
    &=(1-w)^k\sum_{r=0}^\infty\frac{\vartheta^{[r]}f(z_0)}{r!}
    (-4\pi y_0w)^r \\
    &\qquad\qquad\qquad\sum_{n=0}^\infty
    \frac{(k+n+r-1)!}{n!(k+r-1)!}
    \left(-\frac{4\phi^\ast(z_0)\pi y_0}\alpha w\right)^n
  \end{split}
\end{equation*}
Using
$$
(1+x)^{-k-r}=\sum_{n=0}^\infty\frac{(k+n+r-1)!}{n!(k+r-1)!}(-x)^n,
$$
we find that
$$
f(z)=\left(\frac{1-w}{1+Aw}\right)^k\sum_{r=0}^\infty
\frac{\vartheta^{[r]}f(z_0)}{r!}\left(-\frac{4\pi y_0w}{1+Aw}\right)^r,
\quad A=\frac{4\phi^\ast(z_0)\pi y_0}\alpha.
$$
Let
\begin{equation} \label{equation: wt w}
\wt w=\frac w{1+Aw},
\end{equation}
which is also a local parameter at $z_0$. Observing that
$$
\frac{1-w}{1+Aw}=1-(1+A)\wt w,
$$
we obtain the following series expansion of $f$.

\begin{Proposition} \label{proposition: expansion in wt w}
  We have
$$
f(z)=(1-(1+A)\wt w)^k\sum_{r=0}^\infty\frac{\vartheta^{[r]}f(z_0)}{r!}
(-4\pi y_0\wt w)^r.
$$
\end{Proposition}

The advantage of this expansion in $\wt w$ over that in $w$ is that
there is no need to compute $\phi^\ast(z_0)$.

In order to use Proposition \ref{proposition: expansion in wt w} to
compute the local exponents and determine whether \eqref{equation:
  y''=-Qy sect 3} is apparent at $z_0$, we need the following lemma.

\begin{Lemma} \label{lemma: local exponents}
  Let $z_0\in\H$ be a pole of $Q(z)$ and assume that $\wt
  Q(x)=\sum_{n\ge-2}a_nx^n$ is the power series such that
  $$
  Q(z)=(1-(1+A)\wt w)^4\wt Q(-4\pi(\Im z_0)\wt w),
  $$
  where $\wt w$ is given by \eqref{equation: wt w}. Then
  \begin{equation} \label{equation: y temp appendix}
  y(z)=\frac1{1-(1+A)\wt w}\sum_{n=0}^\infty c_n
  (-4\pi(\Im z_0)\wt w)^{n+\alpha}
  \end{equation}
  is a solution of \eqref{equation: y''=-Qy appendix} if and only if
  the series
  $\wt y(x)=\sum_{n=0}^\infty c_nx^{n+\alpha}$ satisfies
  \begin{equation} \label{equation: DE in wt y}
  \frac{d^2}{dx^2}\wt y(x)=\wt Q(x)\wt y(x),
  \end{equation}
  i.e., if and only if
  \begin{equation} \label{equation: indicial}
  \alpha^2-\alpha-a_{-2}=0
  \end{equation}
  and
  \begin{equation} \label{equation: consistency}
  \left((\alpha+n)(\alpha+n-1)-a_{-2}\right)c_n
  =\sum_{j=0}^{n-1}a_{n-j-2}c_j
  \end{equation}
  holds for all $n\ge 1$.
\end{Lemma}

\begin{proof}
  Recall that Bol's identity states that if $f$ is a
  twice-differential function and $\gamma=\SM abcd\in\GL(2,\C)$, then
  $$
  \left(f\big|_{-1}\gamma\right)''(z)
  =\left(f''\big|_3\gamma\right)(z).
  $$
  Let $x=\gamma z$. We observe that
  $$
  a-cx=a-c\frac{az+b}{cz+d}=\frac{\det\gamma}{cz+d}.
  $$
  Hence the factor $(\det\gamma)^{1/2}/(cz+d)$ appearing in the slash
  operator can be written as
  $$
  \frac{(\det\gamma)^{1/2}}{cz+d}=\frac{a-cx}{(\det\gamma)^{1/2}}.
  $$
  Therefore, Bol's identity can be also written as
  $$
  \frac{d^2}{dz^2}\left(\frac{(\det\gamma)^{1/2}}{a-cx}f(x)\right)
  =\frac{(a-cx)^3}{(\det\gamma)^{3/2}}\frac{d^2}{dx^2}f(x).
  $$
  Applying this version of Bol's identity with
  $$
  \gamma=\M{-4\pi\Im z_0}001\M10A1\M1{-z_0}1{-\overline z_0}
  =\M{-4\pi\Im z_0}{(4\pi\Im z_0)z_0}{1+A}{-Az_0-\overline z_0},
  $$
  $\det\gamma=-4\pi(\Im z_0)(z_0-\overline z_0)$, $x=\gamma
  z=-4\pi(\Im z_0)\wt w$, and $f(x)=\wt y(x)$, we obtain
  $$
  \frac{d^2}{dz^2}
  \frac1{1-(1+A)\wt w}\wt y(-4\pi(\Im z_0)\wt w)
  =-4\pi^2(1-(1+A)\wt w)^3\frac{d^2}{dx^2}\wt y(x).
  $$
  The left-hand side of this identity is simply $y''(z)$, where $y(z)$
  is the function in \eqref{equation: y temp appendix}. From this,
  we see that \eqref{equation: y temp appendix} is a solution of
  \eqref{equation: y''=-Qy appendix} if and only if $\wt y(x)$ is a
  solution of \eqref{equation: DE in wt y}.
\end{proof}

\begin{Remark} \label{remark: apparentness condition}
  In view of \eqref{equation: indicial}, we may write the
  two local exponents at $z_0$ at $1/2\pm\kappa$ for some $\kappa\ge
  0$. When $\kappa\notin\frac12\Z_{\ge0}$, the term
  $((\alpha+n)(\alpha+n-1)-a_{-2})$ on the left-hand
  side of \eqref{equation: consistency} is never $0$ for
  both $\alpha=1/2\pm\kappa$ and any $n\ge 1$. In such a case, we can
  recursively determine $c_n$ and get two linearly independent
  solutions of \eqref{equation: y''=-Qy appendix}. On the other hand,
  when $\kappa\in\frac12\Z$, we can always get a
  solution of the form $\wt w^{1/2+\kappa}(1+O(\wt w))$. However, when
  $\alpha=1/2-\kappa$, we can only solve for $c_n$ recursively up to
  $n=2\kappa-1$ because when $n=2\kappa$, the left-hand side of
  \eqref{equation: consistency} becomes $0$. In order for
  \eqref{equation: y''=-Qy appendix} to have a solution of the form
  $\wt w^{1/2-\kappa}(1+O(\wt w))$, i.e., in order for
  \eqref{equation: y''=-Qy appendix} to be apparent at $z_0$, we need
  the numbers $c_0,\ldots,c_{2\kappa-1}$ to satisfy
  \begin{equation} \label{equation: apparent condition}
    \sum_{j=0}^{2\kappa-1}a_{2\kappa-j-2}c_j=0.
  \end{equation}
  Note that this is the only condition required.
\end{Remark}

We now consider the special case where $z_0$ is an elliptic point of
$\Gamma$.

\begin{Proposition} \label{proposition: vanishing coefficients}
  Assume that the stabilizer subgroup $\Gamma_{z_0}$ of $z_0\in\H$
  in $\Gamma$ has order $N$. Let
  $$
  f(z)=(1-(1+A)\wt w)^k\sum_{n=n_0}^\infty a_n\wt w^n
  $$
  be the power series expansion in $\wt w$ of a meromorphic modular
  form $f$ of weight $k$ on $\Gamma$ at $z_0$, where $\wt w$ is given
  by \eqref{equation: wt w}. Then $a_n=0$ whenever
  $k+2n\not\equiv0\mod N$.
\end{Proposition}

\begin{proof}
  The analogous statement for \eqref{equation: power series
    expansion} is already known in literature (for
  example, see \cite{I-O-elliptic} for the case $k$ is even). For the
  convenience of the reader, we give a proof here.
  
  Since any meromorphic modular form is the ratio of two holomorphic
  modular forms, it suffices to prove the proposition under the
  assumption that $f$ is a holomorphic modular form. By
  Proposition \ref{proposition: expansion in wt w},
  $a_n$ is a multiple of $\vartheta^{[n]}f(z_0)$ and the proof reduces
  to showing that $\vartheta^{[n]}f(z_0)=0$ whenever
  $k+2n\not\equiv0\mod N$.
  
  Recall that $\Gamma_{z_0}$ is a finite cyclic group (see
  \cite[Proposition 1.16]{Shimura-book}), say, $\Gamma_{z_0}$ is
  generated by $\gamma=\SM abcd$ of order $N$. Since $z_0$ is fixed by
  $\gamma$, we have
  $$
  z_0=\frac{a-d+\sqrt{(d-a)^2+4bc}}{2c}
  =\frac{a-d+\sqrt{(a+d)^2-4}}{2c}
  $$
  and
  $$
  cz_0+d=\frac{a+d+\sqrt{(a+d)^2-4}}{2}.
  $$
  Observe that it is a root of the characteristic polynomial
  $x^2-(a+d)x+1$ of $\gamma$, i.e., an eigenvalue of $\gamma$. Hence,
  $cz_0+d$ is a primitive $N$th root of unity. Now $\vartheta^{[n]}f$
  is a holomorphic modular form of weight $k+2n$ on $\Gamma$, i.e.,
  $$
  \vartheta^{[n]}f(\gamma z)=(cz+d)^{k+2n}
  \vartheta^{[n]}f(z).
  $$
  Evaluating the two sides at $z_0$, we see that
  $\vartheta^{[n]}f(z_0)=0$ whenever $N\nmid(k+2n)$.
  This proves the proposition.
\end{proof}

Using this proposition, we immediately obtain a sufficient condition
for \eqref{equation: y''=-Qy appendix} to be apparent at an elliptic
point.

\begin{Theorem} \label{corollary: apparent at elliptic}
  Assume that $z_0$ is an elliptic point of order $e$
  on $X(\Gamma)$ and the local exponents of \eqref{equation: y''=-Qy
    appendix} at $z_0$ is $1/2\pm\kappa$ for some
  $\kappa\in\frac12\Z_{\ge 0}$ such that $e\nmid(2\kappa_{z_0})$, then
  \eqref{equation: y''=-Qy appendix} is apparent at $z_0$.
\end{Theorem}

\begin{proof} Let $Q(z)=(1-(1+A)\wt w)^4\sum_{n=-2}^\infty a_n\wt w^n$
  be the expansion of $Q(z)$ in $\wt w$. Since $Q(z)$ has weight $4$,
  we may assume that $\Gamma$ contains $-I$. Then the stabilizer
  subgroup $\Gamma_{z_0}$ of $z_0$ in $\Gamma$ has order $2e$ (see
  \cite[Proposition 1.20]{Shimura-book}). Hence, by Proposition
  \ref{proposition: vanishing coefficients},
  $a_n=0$ whenever $4+2n\not\equiv 0\mod 2e$. Now the differential
  equation \eqref{equation: y''=-Qy appendix} is apparent if and only
  if it has a solution of the form
  $$
  y(z)=\frac1{1-(1+A)\wt w}\wt w^{1/2-\kappa}\sum_{n=0}^\infty
  c_n\wt w^n, \qquad c_0\neq 0.
  $$
  By Lemma \ref{lemma: local exponents}, the coefficients $c_n$ must
  satisfy \eqref{equation: consistency}. Since $a_n=0$ whenever $n\neq
  -2\mod e$, we can inductively prove using \eqref{equation:
    consistency} that $c_n=0$ whenever $e\nmid n$, up to
  $n=2\kappa-1$. Then the apparentness condition \eqref{equation:
    apparent condition} holds because $2\kappa-j-2\equiv-2\mod e$ and
  $j\equiv 0\mod e$ cannot hold simultaneously, by the assumption that
  $e\nmid(2\kappa_{z_0})$. This proves the theorem.
\end{proof}

In the remainder of the section, we will specialize $\Gamma$ to one of
the three groups $\SL(2,\Z)$, $\Gamma_0^+(2)$, and $\Gamma_0^+(3)$,
and prove the existence of $Q(z)$ satisfying \eqref{equation: H}.

Recall that $\rho_1$ and $\rho_2$ are the two elliptic points of
$\Gamma$. Given $\{\rho_1,\rho_2,z_1,\ldots,z_m\}$ with
$\rho_1,\rho_2,z_1,\ldots,z_m$ mutually $\Gamma$-inequivalent and
$\{\kappa_{\rho_1},\kappa_{\rho_2},\kappa_1,\ldots,\kappa_m\}
\in\frac12\N$, where $(2\kappa_{\rho_i},e_i)=1$ for $i=1,2$, we want
to show the existence of a meromorphic modular form $Q(z)$ of weight
$4$ such that $\{\rho_1,\rho_2,z_1,\ldots,z_m\}$ is the set of
singular points of $Q(z)$ and \eqref{equation: H} holds. Note that
\eqref{equation: H} allows $\rho_i$ to be smooth if
$\kappa_{\rho_i}=1/2$.

When $\Gamma=\SL(2,\Z)$, we let $t_j=E_4(z_j)^3/E_6(z_j)^2$, which is
not $0$ nor $1$ since $z_j$ is not an elliptic point nor a cusp. We
observe that if $Q(z)$ satisfies \eqref{equation: H}, then $Q(z)$ can
be expressed as
\begin{equation} \label{equation: Q SL(2,Z)}
  \begin{split}
    Q(z)&=rE_4(z)+s\frac{\Delta(z)}{E_4(z)^2}
    +t\frac{E_4(z)\Delta(z)}{E_6(z)^2} \\
    &\qquad+E_4(z)\sum_{j=1}^m\left(r_1^{(j)}
      \frac{\Delta(z)^2}{F_j(z)^2}
      +r_2^{(j)}\frac{\Delta(z)}{F_j(z)}\right),
  \end{split}
\end{equation}
where $r$, $s$, $t$, $r_1^{(j)}$, and $r_2^{(j)}$ are complex
parameters and $F_j(z):=E_4(z)^3-t_jE_6(z)^2$ are modular forms of
weight $12$. Note that at any elliptic point $\rho_j$, Proposition
\ref{proposition: vanishing coefficients} implies that
$Q(z)=A(z-\rho_j)^{-2}+(\text{a holomorphic function near }\rho_j)$,
and the expression \eqref{equation: Q SL(2,Z)} follows immediately
from this fact. The indicial equations at $\rho$, $i$, and $z_j$ can
be computed using Lemma \ref{lemma: local exponents}, together with
Proposition \ref{proposition: expansion in wt w} and Ramanujan's
identities \eqref{equation: Ramanujan}. We find they are
\begin{equation} \label{equation: indicial SL(2,Z)}
x^2-x+\frac s{192}=0, \
x^2-x-\frac t{432}=0, \
x^2-x-\frac{r_1^{(j)}}{1728^2t_j}=0,
\end{equation}
respectively. Thus, the parameters $s$, $t$, and $r_j^{(1)}$ are
uniquely determined by the given data $\kappa_{\rho}$, $\kappa_i$,
and $\kappa_j$. The rest of parameters will be determined by the
apparentness condition. See Theorem \ref{theorem: existence of Q}
below.

When $\Gamma=\Gamma_0^+(2)$, we define modular forms
\begin{equation*}
  \begin{split}
    M_2^-(z)&=2E_2(2z)-E_2(z), \qquad
    M_4^-(z)=\frac{4E_4(2z)-E_2(z)}3, \\
    M_4^+(z)&=\frac{4E_4(2z)+E_4(z)}5, \qquad
    M_6^+(z)=\frac{8E_6(2z)+E_6(z)}9, \\
    M_8^+(z)&=\eta(z)^8\eta(2z)^8,
  \end{split}
\end{equation*}
where the integers in the subscripts denote the weights and the signs
in the superscripts denote the eigenvalues under the Atkin-Lehner
involution. Among them, $M_2^-(z)$ and $M_4^-(z)$ freely generate the
ring of modular forms on $\Gamma_0(2)$, while $M_4^+(z)$, $M_6^+(z)$,
and $M_8^+(z)$ generate the ring of modular forms on $\Gamma_0^+(2)$
with a single relation $M_6^+(z)^2=M_4^+(z)(M_4^+(z)^2-256M_8^+(z))$.
The expressions for $M_4^+$, $M_6^+$, and $M_8^+$ as products of
$M_2^-$ and $M_4^-$ are
\begin{equation} \label{equation: M4 in M2-}
  M_4^+=(M_2^-)^2, \quad
  M_6^+=M_2^-M_4^-, \quad
  M_8^+=\frac{(M_2^-)^4-(M_4^-)^2}{256}.
\end{equation}
Also, set $M_2^\ast(z)=(2E_2(2z)+E_2(z))/3$, which is a quasimodular
form of weight $2$ and depth $1$ on $\Gamma_0^+(2)$. We have
\begin{equation} \label{equation: Ramanujan Gamma0+(2)}
  \begin{split}
    D_qM_2^\ast(z)&=\frac{M_2^\ast(z)^2-M_4^+(z)}8, \\
    D_qM_2^-(z)&=\frac{M_2^\ast(z)M_2^-(z)-M_4^-(z)}4, \\
    D_qM_4^-(z)&=\frac{M_2^\ast(z)M_4^-(z)-M_2^-(z)^3}2.
  \end{split}
\end{equation}
The modular forms $M_2^-(z)$ and $M_4^-(z)$ have a simple zero at
$\rho_2=(1+i)/2$ and $\rho_1=i/\sqrt2$, respectively, and are
nonvanishing elsewhere, up to $\Gamma_0(2)$-equivalence.

As in the case of $\SL(2,\Z)$, we let $t_j=M_4^-(z_j)^2/M_4^+(z_j)^2$,
which is never $0$ nor $1$ since $z_j$ is not a cusp or an elliptic
point, and set $F_j(z)=M_4^-(z)^2-t_jM_4^+(z)^2$. We then observe that
the ratios $M_8^+(z)/M_4^+(z)$ and
$M_4^+(z)M_8^+(z)/(M_4^+(z)^2-256M_8^+(z))$ have double poles at
$\rho_2$ and $\rho_1$, respectively. Also, the modular
form $F_j(z)$ has a simple zero at the point $z_j$. Hence Proposition
\ref{proposition: vanishing coefficients} implies that
\begin{equation} \label{equation: Q Gamma0+(2)}
  \begin{split}
    Q(z)&=rM_4^+(z)+s\frac{M_8^+(z)}{M_4^+(z)}
    +t\frac{M_4^+(z)M_8^+(z)}{M_4^+(z)^2-256M_8^+(z)} \\
    &\qquad+M_4^+(z)\sum_{j=1}^m
    \left(r_1^{(j)}\frac{M_8^+(z)^2}{F_j(z)^2}
      +r_2^{(j)}\frac{M_8^+(z)}{F_j(z)}\right),
  \end{split}
\end{equation}
where $r$, $s$, $t$, $r_1^{(j)}$, and $r_2^{(j)}$, $j=1,\ldots,m$, are
complex parameters. Among the parameters, $s$, $t$, and $r_1^{(j)}$
are determined by the data $\kappa_{\rho_i}$ and $\kappa_j$ since the
indicial equations at $\rho_1$, $\rho_2$, and $z_j$ are
\begin{equation} \label{equation: indicial Gamma0+(2)}
  x^2-x-\frac t{64}=0, \
  x^2-x+\frac s{16}=0, \
  x^2-x-\frac{r_1^{(j)}}{256^2t_j}=0,
\end{equation}
respectively. (They are computed using Lemma \ref{lemma: local
  exponents}, Proposition \ref{proposition: expansion in wt w} and
\eqref{equation: Ramanujan Gamma0+(2)}.) Again, the remaining
parameters will be determined by the apparentness condition.

When $\Gamma=\Gamma_0^+(3)$, we define the modular forms
\begin{equation*}
  \begin{split}
    M_1(z)=\sum_{m,n\in\Z}q^{m^2+mn+n^2}, \qquad
    M_3(z)=\frac{\eta(z)^9}{\eta(3z)^3}
    -27\frac{\eta(3z)^9}{\eta(z)^3}
  \end{split}
\end{equation*}
of weight $1$ and $3$, respectively, on $\Gamma_1(3)$.
They generate the ring of modular forms on $\Gamma_1(3)$. Set
$M_2^\ast(z)=(3E_2(3z)+E_2(z))/4$. We have
\begin{equation} \label{equation: Ramanujan Gamma0+(3)}
  \begin{split}
    D_qM_2^\ast(z)&=\frac{M_2^\ast(z)^2-M_1(z)^4}6, \\
    D_qM_1(z)&=\frac{M_2^\ast(z)M_1(z)-M_3(z)}6, \\
    D_qM_3(z)&=\frac{M_2^\ast(z)M_3(z)-M_1(z)^5}2.
  \end{split}
\end{equation}
Also let
\begin{equation*}
  \begin{split}
    M_2^-(z)&=\frac{3E_2(2z)-E_2(z)}2=M_1(z)^2
    \in\sM_2(\Gamma_0^+(3),-), \\
    M_4^+(z)&=\frac{9E_4(3z)+E_4(z)}{10}=M_1(z)^4
    \in\sM_4(\Gamma_0^+(3)), \\
    M_6^-(z)&=\eta(z)^6\eta(3z)^6=\frac{M_1(z)^6-M_3(z)^2}{108}
    \in\sM_6(\Gamma_0^+(3),-).
  \end{split}
\end{equation*}
Then meromorphic modular forms $Q(z)$ of weight $4$ on $\Gamma_0^+(3)$
such that all poles have order $\le 2$ are of the form
\begin{equation} \label{equation: Q Gamma0+(3)}
  \begin{split}
    Q(z)&=rM_4^+(z)+s\frac{M_6^-(z)}{M_2^-(z)}
    +t\frac{M_4^+(z)M_6^-(z)}{M_3(z)^2} \\
    &\qquad+M_4^+(z)\sum_{j=1}^m\left(
      r_1^{(j)}\frac{M_6^-(z)^2}{F_j(z)^2}
      +r_2^{(j)}\frac{M_6^-(z)}{F_j(z)}\right),
  \end{split}
\end{equation}
where $F_j(z)$ is a modular form in $\sM_6(\Gamma_0^+(3),-)$ of the
form $F_j(z)=M_1(z)^6-t_jM_3(z)^2$ with $t_1,\ldots,t_m$ being
distinct complex numbers such that none of them is $0$ or $1$. The
ratio $M_6^-(z)/M_2^-(z)$ has a double pole at
$\rho_2=(3+\sqrt{-3})/6$, while $M_4^+(z)M_6^-(z)/M_3(z)^2$ has a
double pole at $\rho_1=i/\sqrt 3$. The modular form $F_j(z)$ has a
simple zero at the point $z_j$ such that $t_j=M_1(z_j)^6/M_3(z_j)^2$.

Again, $s$, $t$, and $r_1^{(j)}$ are uniquely determined by
the indicial equations
\begin{equation} \label{equation: indicial Gamma0+(3)}
  x^2-x-\frac t{27}=0, \
  x^2-x+\frac s3=0, \
  x^2-x-\frac{r_1^{(j)}}{108^2t_j}=0
\end{equation}
at $\rho_1$, $\rho_2$, and $z_j$, respectively, computed using
\ref{lemma: local exponents}, Proposition \ref{proposition: expansion
  in wt w} and \eqref{equation: Ramanujan Gamma0+(3)}. It remains to
determine the other parameters $r$ and $r_2^{(j)}$, $j=1,\ldots,m$.
Note that by Theorem \ref{corollary: apparent at elliptic}, the ODE
\eqref{equation: y''=-Qy appendix} is apparent at all the elliptic
points.

\begin{Theorem} \label{theorem: existence of Q}
  Given $\{\rho_1,\rho_2,z_1,\ldots,z_m\}$ and
  $\{\kappa_{\rho_1},\kappa_{\rho_2},\kappa_1,\ldots,\kappa_m\}
  \subset\frac12\N$ and $\kappa_\infty\ge 0$, there are
  $\prod_{j=1}^m(2\kappa_j)$ meromorphic modular forms $Q(z)$ of
  weight $4$ on $\Gamma$, counted with multiplicity, such that
  \eqref{equation: H} holds and the local exponents at $\infty$ are
  $\pm\kappa_\infty$.
\end{Theorem}

\begin{proof} The theorem was proved in \cite{Guo-Lin-Yang} when
  $\Gamma=\SL(2,\Z)$. For the other two groups $\Gamma_0^+(2)$ and
  $\Gamma_0^+(3)$, the proofs are basically the same. For the
  convenience of the reader, we sketch them here.

  Using the expansion in Lemma \ref{lemma: local exponents}, we have
  that \eqref{equation: y''=-Qy appendix} is apparent at $z_i$ if and
  only if there is a solution $y(z)=(1-(1+A)\wt w)^{-1}\wt
  w^{1/2-\kappa_i}\sum_{n=0}^\infty c_n\wt w^n$, $c_0=1$. By Remark
  \ref{remark: apparentness condition}, it is equivalent to
  \eqref{equation: apparent condition} holds for $c_j$, $1\le
  j\le2\kappa_i-1$, obtained by the recursive relation
  \eqref{equation: consistency}. Then the condition \eqref{equation:
    apparent condition} at $z_i$ yields a polynomial
  $P_j(r,r_2^{(1)},\ldots,r_2^{(m)})$ of degree $2\kappa_j$ and it is
  not difficult to see that the polynomial $P_j$ is of the form
  \begin{equation} \label{equation: P}
    P_j(r,r_2^{(1)},\ldots,r_2^{(m)})=B_j(r_2^{(j)})^{2\kappa_j}
      +(\text{terms of lower order})
  \end{equation}
  for some constant $B_j\neq 0$ independent of $r$ and
  $r_2^{(j)}$. (For details, we refer to \cite[Theorem
  1.3]{Guo-Lin-Yang}.) Suppose that $\kappa_\infty$ is given. Then we
  have $r=\kappa_\infty^2$. Thus, $Q(z)$ is apparent at all
  $z_j$ with the given data if and only if
  $(r_2^{(1)},\ldots,r_2^{(m)})$ is a common root of the $m$
  polynomials \eqref{equation: P}. It is easy to see that there are no
  common roots at $\infty$ when all polynomials are homogenized. By
  Bezout's theorem, there are $\prod_{j=1}^m(2\kappa_j)$ solutions,
  counted with multiplicities. This proves the theorem.
\end{proof}

\begin{Example} Consider the quasimodular form
  $f(z)=E_2(z)E_4(z)+aE_6(z)$ of weight $6$ and depth $1$ on
  $\SL(2,\Z)$, where $a$ is a complex number. We compute that
  $$
  W_f(z)=(-1-6a)E_4(z)^3-(a^2-4a)E_6(z)^2
  $$
  and $f(z)/\sqrt{W_f(z)}$ is a solution of \eqref{equation: y''=-Qy
    appendix} with
  $$
  Q(z)=E_4(z)\left(\frac34t_1\frac{1728^2\Delta(z)^2}{F_1(z)^2}
    -\frac{1+31a}{12(1+6a)}\frac{1728\Delta(z)}{F_1(z)}\right),
  $$
  where $t_1=(4a-a^2)/(1+6a)$ and $F_1(z)=E_4(z)^3-t_1E_6(z)^2$.
  Let $z_1$ be a point in $\H$ such that $t_1=E_4(z_1)^3/E_6(z_1)^2$.
  The example shows that for a generic point $z_1$, there are two
  meromorphic modular form $Q(z)$ of weight $4$ on $\SL(2,\Z)$
  (corresponding to the two complex numbers $a$ such that
  $t_1=(4a-a^2)/(1+6a)$) such that \eqref{equation: y''=-Qy appendix}
  is apparent at $z_1$ with local exponents $-1/2,3/2$.
  Theorem \ref{theorem: existence of Q} asserts that these are the
  only two such $Q(z)$.
\end{Example}

\end{document}